\documentclass[10pt,twoside]{amsart}
\usepackage{amsmath,amsfonts,amssymb,amsxtra,bm,setspace,xspace,graphicx,lmodern,psfrag,epsfig,color,latexsym,stmaryrd,scalerel,stackengine,enumerate,subcaption,caption}
\usepackage[colorlinks=true]{hyperref}\hypersetup{urlcolor=blue, citecolor=red}

\numberwithin{equation}{section}
\newtheorem{thm}{Theorem}[section]
\newtheorem{cor}[thm]{Corollary}
\newtheorem{lem}[thm]{Lemma}

\newtheorem{rem}[thm]{Remark}

\newtheorem{exm}[thm]{Example}

\stackMath
\newcommand\reallywidehat[1]{%
	\savestack{\tmpbox}{\stretchto{%
			\scaleto{%
				\scalerel*[\widthof{\ensuremath{#1}}]{\kern-.6pt\bigwedge\kern-.6pt}%
				{\rule[-\textheight/2]{1ex}{\textheight}}
			}{\textheight}%
		}{0.5ex}}%
	\stackon[1pt]{#1}{\tmpbox}%
}
\newcommand{\vertiii}[1]{{\left\vert\kern-0.25ex\left\vert\kern-0.25ex\left\vert #1 
		\right\vert\kern-0.25ex\right\vert\kern-0.25ex\right\vert}}

\newcommand{\R}{{\mathbb R}}

\newcommand{\N}{{\mathbb N}}

\begin{document}
	
	\title[LDG method for Vlasov-Burger]{Discontinuous Galerkin Methods with  Generalized Numerical Fluxes for the Vlasov-Viscous Burgers' System}
	
	\bibliographystyle{alpha}
	
	\author[Harsha Hutridurga]{Harsha Hutridurga}
	\address{H.H.: Department of Mathematics, Indian Institute of Technology Bombay, Powai, Mumbai 400076 India.}
	\email{hutri@math.iitb.ac.in}
	
	\author[Krishan Kumar]{Krishan Kumar}
	\address{K.K.: Department of Mathematics, Indian Institute of Technology Bombay, Powai, Mumbai 400076 India.}
	\email{krishankumar@math.iitb.ac.in}
	
	\author[Amiya K. Pani]{Amiya K. Pani* } {\thanks{*Corresponding Author}}
	\address{A.K.P.: Department of Mathematics, Birla Institute of Technology and Science, Pilani, KK Birla Goa Campus, NH 17 B, Zuarinagar, Goa 403726 India.}
	\email{amiyap@goa.bits-pilani.ac.in}

	\date{\today}
		\maketitle
	
	\begin{abstract}
		In this paper, semi-discrete numerical scheme for the approximation of the periodic Vlasov-viscous Burgers' system is developed and analyzed. The scheme is based on the coupling of discontinuous Galerkin approximations for the Vlasov equation and local discontinuous Galerkin approximations for the viscous Burgers' equation. Both these methods use generalized numerical fluxes. The proposed scheme is both mass and momentum conservative. Based on generalized Gauss-Radau projections, optimal rates of convergence in the case of smooth compactly supported initial data are derived. Finally, computational results confirm our theoretical findings.
	\end{abstract}



	\section{Introduction}
	
	The simplest kinematic model for nonevaporating dilute two phase flow which takes into account only the exchange of momentum between the two phases is described by the following coupled system of viscous Burgers' equation and a Vlasov type equation: 
\begin{equation}\label{eq:continuous-model}
	\left\{
	\begin{aligned}
		\partial_t f + v\partial_x f + \partial_v\left(\left(u - v\right)f\right) &= 0 \quad\qquad \mbox{in}\quad (0,T]\times I \times \R,
		\\
		f(0,x,v) &= f_0(x,v) \quad\mbox{in}\quad I\times\R, 
	\end{aligned}
	\right.
\end{equation}
\begin{equation}\label{contburger}
	\left\{
	\begin{aligned}
		\partial_tu + u\partial_x u - \epsilon \partial_x^2 u &= \rho V - \rho u  \quad \mbox{in} \quad (0,T]\times I,
		\\
		u(0,x) &= u_0(x) \qquad\mbox{in}\quad I,  
	\end{aligned}
	\right.
\end{equation}
with periodic boundary conditions:
\begin{equation*}
	f(t,L,v) = f(t,0,v), \,\,\,u(t,L) = u(t,0) \quad \mbox{and} \quad u_x(t,L) = u_x(t,0).
\end{equation*}
	Here, $I = [0,L]$, $u = u(t,x)$ represents the fluid velocity, $f = f(t,x,v)$ denotes  the distribution function and $\epsilon > 0$ represents the viscosity of the fluid. The coupling between two systems is due to drag force which is proportional to relative velocity $(u - v)$.
	
In the above model, $f(t,x,v)$ describes the dispersed phase whereas $u(t,x)$ describes the background continuous phase. Such dispersed two-phase systems are relevant in many applications for example, in modeling combustion phenomena in diesel engines, where a spray of droplets is injected in the device and mixed with the gas prior to combustion \cite{o1981collective,williams2018combustion}. Here, the dispersed phase is the spray, whereas the continuous phase is the surrounding gas.
		
	The well-posedness of \eqref{eq:continuous-model}-\eqref{contburger}, global existence and uniqueness of a solution $u \in C^0([0,T];C^2(\R))$ and $f \in C^0([0,T];C_0^1(\R \times \R))$ for $u_0 \in C^2(\R)$ and $f_0 \in C_0^1(\R \times \R)$, was proved by Domelevo and Roquejoffre in \cite{domelevo1999existence}. Further, the global existence of a weak solution $(u,f) \in L^2(0,T;H^1(\R)) \times L^\infty(0,T;\mathcal{M}^1(\R \times \R))$ for $(u_0,f_0) \in L^2(\R) \times L^1(\R \times \R)$ was shown by Goudon in \cite{goudon2001asymptotic}, where $\mathcal{M}^1(\R\times\R)$ stands for the set of bounded measures on the domain $\R\times\R$. In this paper, with domain $ I\times\R$ and periodic boundary conditions, we show in the appendix the existence of a unique strong solution $(u,f) \in \left(L^\infty(0,T;H^1(I))\cap L^2(0,T;H^2(I))\right) \times \left(L^\infty(0,T;L^1(I\times\R)\cap L^\infty(I\times\R))\cap L^\infty(0,T;H^1(I\times\R))\right)$ \\
for $(u_0,f_0) \in H^1(I)\times L^1(I\times\R) \cap L^\infty(I\times\R)\cap H^1(I\times\R)$.
	
	Discontinuous Galerkin (DG) method belongs to a class of finite element methods, which preserve the conservation property of approximating problems. It was first introduced by Reed and Hill in 1973 \cite{reed1973triangular} in order to solve a linear first order hyperbolic equation and further, for time-dependent nonlinear conservation law, it was developed by Cockburn and Shu in 1989 \cite{cockburn1989tvb}. The local discontinuous Galerkin (LDG) method, a special class of DG schemes for approximating solutions of the nonlinear hyperbolic equation, was first introduced by Cockburn and Shu in 1998 \cite{cockburn1998local}. For LDG method, we first rewrite systems containing high order derivative terms into an equivalent first order system and we discretized the latter by using the standard DG method with appropriate numerical fluxes. 
	
	This paper proposes DG method for the Vlasov equation and LDG method for the viscous Burgers' equation. In order to define DG method, the numerical fluxes are the most important ingredients which ensure stability and accuracy of the method. In this paper, we choose generalized numerical fluxes for both the equations, because several numerical experiments suggest a trade-off between the viscosity parameter $\epsilon$ and the generalized flux parameters. We mention here that, unlike the continuous case, it is difficult to prove non-negativity of the discrete distribution function  and this creates a major problem in the subsequent analysis.
	
    Let us mention some work present in the literature which are related to the current paper. The upwind biased numerical fluxes are considered for linear hyperbolic equation in \cite{meng2016optimal}, and generalized alternating fluxes for linear diffusion term in combination with the central flux for nonlinear convection term are considered in \cite{liu2015,ploymaklam2017priori}. Further, for solving linear convection-diffusion equation, the generalized numerical fluxes with two independent weight functions are defined in \cite{cheng2017application}.
    For generalized numerical fluxes, optimal error estimates can be derived by virtue of some special global projections. Motivated by \cite{cao2017superconvergence,meng2012superconvergence,cheng2017,li2019analysis,meng2016optimal,li2020discontinuous,liu2020optimal,ayuso2009discontinuous}, we define a global projection connected to Gauss-Radau projection. 
    Our major contributions in this paper are as follows:
	\begin{itemize}
		\item Applying DG scheme for the Vlasov equation and LDG scheme for the viscous Burgers' equation keeping time variable continuous and using generalized numerical fluxes, a semidiscrete system for the coupled problems \eqref{eq:continuous-model}-\eqref{contburger} is developed. Then, discrete conservation properties are established. 
		\item For the first time, based on generalized Gauss-Radau projections, optimal error estimates for a semi-discrete system 
are derived. Special care has been taken to avoid bound not depending on $e^{\epsilon t}$ but depending on $\frac{1}{\sqrt{\epsilon}}$ in the error analysis by using non-linear version of the Gr\"onwall type result.  
\item Some numerical experiments are conducted whose results confirm our theoretical findings. 
	\end{itemize}
The paper is structured as follows. In section \ref{cts}, we describe the continuous model with some properties and notations. In section  \ref{dg}, we introduce semi-discrete DG method with generalized numerical fluxes, some notations which will be used through out this paper and also some properties of the discrete solution. The error analysis for the semi-discrete method is detailed in section \ref{err} followed by some computational results in section \ref{computation} to validate our theoretical findings. Finally, in Appendix \ref{exststrong}, we proof the existence and uniqueness of a strong solution to the continuum model.

\section{Continuous Problem}\label{cts}

We set the local density $\rho$ and the local macroscopic velocity $V$, respectively, as
\[
\rho(t,x) = \int_{\R} f(t,x,v)\,{\rm d}v \quad \mbox{and} \quad V(t,x) = \frac{1}{\rho}\int_{\R} f(t,x,v)v\,{\rm d}v.
\]
Now define mass and momentum, respectively, by
	\[
	\int_{\R}\int_I f(t,x,v)\,{\rm d}x\,{\rm d}v \quad \mbox{and} \quad \int_{\R}\int_I vf(t,x,v)\,{\rm d}x\,{\rm d}v.
	\]
	We denote the $k^{th}$ order velocity moments by
	\[
	m_kf(t,x) = \int_{\R}|v|^kf(t,x,v)\,{\rm d}v, \quad \mbox{for}\quad k \in \N\cup\{0\}.
	\]
	Throughout this paper, we use standard notation for Sobolev spaces. We denote by $W^{m,p}$ the $L^p$-Sobolev space of order $m \geq 0$ and by $C_0^1(\R\times\R)$ the class of $C^1$ functions on $\R\times\R$ which are compactly supported.\\
	Throughout this manuscript, any function defined on $I$ is assumed to be periodic in the $x$-variable.

	The following theorem shows existence and uniqueness of the classical solution to \eqref{eq:continuous-model}-\eqref{contburger} posed on $\R\times\R$ (for detailed proof, see \cite[Theorem 2.1, p. 65]{domelevo1999existence}). 
	\begin{thm}\label{ctsexst}
		Let $u_0 \in C^2(\R)$ and $f_0 \in C_0^1(\R\times\R), f_0 \geq 0$ be given. Then, the Vlasov-viscous Burgers' system has a unique solution $\left(u, f\right) \in C^0([0,\infty);C^2(\R)) \times C^0([0,\infty);C^1_0(\R\times\R))$. 
	\end{thm}

	\subsection{Some properties of the solution}

	We begin this subsection by gathering certain conservation properties of \eqref{eq:continuous-model}-\eqref{contburger}, the proof of which can be found in \cite[Proposition 2.1, p. 1374]{goudon2001asymptotic}.
	\begin{lem}\label{lem:cons}
	The solution $(f,u)$ to the Vlasov - viscous Burgers' system has the following properties:
	\begin{enumerate}[(i)]
		\item \textbf{(Positivity preserving)} For any given non-negative initial data $f_0$, the solution $f$ is also non-negative.
		\item \textbf{(Mass conservation)} The total mass is conserved in the sense that
		\begin{equation*}
			\int_{\R}\int_I\,f(t,x,v)\,{\rm d}x\,{\rm d}v = \int_{\R}\int_I\,f_0(x,v)\,{\rm d}x\,{\rm d}v, \quad t \in [0,T].
		\end{equation*}
	   \item \textbf{(Total Momentum conservation)} The solution pair $\left(f,u\right)$ conserves total momentum in the following sense: 
	   \begin{equation*}
	   	\int_{\R}\int_I vf(t,x,v)\,{\rm d}x\,{\rm d}v + \int_Iu(t,x)\,{\rm d}x = \int_{\R}\int_I vf_0(x,v)\,{\rm d}x\,{\rm d}v + \int_Iu_0(x)\,{\rm d}x,\,\, t \in [0,T].
	   \end{equation*}
       \item \textbf{(Total energy dissipation)}
       The total energy of the system dissipates in the sense that
       \begin{equation}\label{energy}
       	\begin{aligned}
       		\int_{\R}\int_I\, v^2\,f(t,x,v)\,{\rm d}x\,{\rm d}v + \int_{I}u^2(t,x)\,{\rm d}x \leq \int_{\R}\int_I v^2 f_0(x,v)\,{\rm d}x\,{\rm d}v + \int_{I}u_0^2(x)\,{\rm d}x, \quad t \in [0,T]
       	\end{aligned}  
       \end{equation}
   provided $f(t,x,v)$ is non-negative.
	\end{enumerate}
	\end{lem}

    While proving the energy dissipation property \eqref{energy}, we also obtain the following identity:
	\begin{equation*}
		\begin{aligned}
			\frac{1}{2}\int_{\R}\int_I v^2f\,{\rm d}x\,{\rm d}v + \frac{1}{2}\int_I u^2\,{\rm d}x + \epsilon\int_{0}^t\int_I u_x^2\,{\rm d}x\,{\rm d}t &+ \int_0^t\int_{\R}\int_I\left(u - v\right)^2f\,{\rm d}x\,{\rm d}v\,{\rm d}t 
			\\
			&= \frac{1}{2}\int_{\R}\int_I v^2f_0\,{\rm d}x\,{\rm d}v + \frac{1}{2}\int_I u_0^2\,{\rm d}x.
		\end{aligned}
	\end{equation*} 
	If $ v^2f_0 \in L^1(I\times \R)$ and if $u_0 \in L^2(I)$, then the above equality shows
	\begin{equation}\label{3}
		u \in L^\infty(0,T;L^2(I)) \quad \mbox{and} \quad u_x \in L^2([0,T]\times I).
	\end{equation}
	A use of the Sobolev inequality yields
	\begin{equation}
		u \in L^2(0,T;L^\infty(I)).
	\end{equation}
	Note that for any $z \in L^\infty(I)$, 
	\begin{equation*}
		\|z\|_{L^4(I)} \leq C\|z\|_{L^2(I)}^\frac{1}{2}\|z\|^\frac{1}{2}_{L^\infty(I)}.
	\end{equation*}
	Hence, 
	\begin{equation}\label{7}
		\int_0^T\int_{I}|u|^4\,{\rm d}x\,{\rm d}t \leq C\|u\|_{L^\infty(0,T;L^2(I)}^2\|u\|^2_{L^2(0,T;L^\infty(I))} \leq C.
	\end{equation}


The following lemma yields integrability estimates on the local density and on the momentum. Since these appear as source terms in the viscous Burgers' equation, these estimates are crucial in deducing the regularity result of solution to \eqref{contburger}. The proof of the following result is similar to \cite[Lemma 2.2, p.56]{Hamdache_1998}. Hence, we skip the proof.	
	\begin{lem}\label{density}
		Let $p,r \geq 1$. Let $u \in L^r(0,T;L^{p+1}(I)), f_0 \in L^\infty(I\times\R)\cap L^1(I \times \R)$. Further, let
		\[
		\int_{\R}\int_I\,|v|^pf_0\,{\rm d}x\,{\rm d}v < \infty.
		\]
		 Then, the local density $\rho$ and the momentum $\rho V$ satisfy the following:
		 \begin{equation*}
		 	\rho \in L^\infty(0,T;L^{p+1}(I)) \quad \mbox{and} \quad \rho V \in L^\infty(0,T;L^{\frac{p+1}{2}}(I)).
		 \end{equation*}		
	\end{lem}

	\begin{rem}
		Taking $p = 3$ in Lemma \ref{density} yields
		\begin{equation}\label{rhos}
			\rho \in L^\infty(0,T;L^4(I)) \quad \mbox{and} \quad \rho V \in L^\infty(0,T;L^2(I)).
		\end{equation}
	\end{rem}

	
	The following lemma gives $L^\infty$ estimate for local density $\rho$ in time and space variable. The proof of this can be found in \cite[Proposition 4.6, p. 44]{han2019uniqueness}.

	\begin{lem}\label{rhoinfty}
		Let $u \in L^1(0,T;L^\infty(I))$. Let $f_0(x,v)$ be such that $\sup_{C^r_{t,v}}f_0 \in L^\infty_{loc}\left(\R_+;L^1(\R)\right)$ for all $r > 0$, where $C^r_{t,v} := I \times B(e^tv,r)$. Here $B(e^{tv},r)$ denotes the ball of radius $r$ with center at $e^{tv}$. Then, the following estimate holds:
		\begin{equation*}
			\|\rho\|_{L^{\infty}((0,T)\times I)} \leq e^{T}\sup_{t \in [0,T]}\|\sup_{C^r_{t,v}}f_0\|_{L^1(\R)}.
		\end{equation*}     	
	\end{lem}
	
For completeness, the existence of a unique strong solution to  the problem \eqref{eq:continuous-model}-\eqref{contburger} 
is discussed in  the Appendix \ref{exststrong}.

	\section{Semi-discrete scheme}\label{dg}
	
	This section deals with a semi-discrete scheme to approximate solutions to \eqref{eq:continuous-model}-\eqref{contburger} and with some properties of the said discrete system. Note that, for a compactly supported initial datum $f_0$, the solution $f(t,x,v)$ has compact support. Therefore, without loss of generality, we assume that there is $M > 0$ such that for $v \in [-M,M] =: J$ and $t \in (0,T],\,\, \mbox{supp}f(t,x,v) \subset \Omega = I \times J$. 
	
	Let $I_h = \{I_i\}_{i=1}^{N_x}$ and $J_h = \{J_j\}_{j=1}^{N_v}$ be the partitions of intervals $I$ and $J$, respectively.
	Let $\mathcal{T}_{h}$ be defined as the Cartesian product of these two partitions, i.e.
	\begin{align*}
		\mathcal{T}_{h} = \left\{ T_{ij} = I_i \times J_j : 1 \leq i \leq N_x, 1 \leq j \leq N_v  \right\},
	\end{align*}
	with
	\[
	I_i = [x_{i - \frac{1}{2}},x_{i+\frac{1}{2}}] \quad \forall \,\, 1 \leq i \leq N_x; \quad J_j = [v_{j - \frac{1}{2}}, v_{j + \frac{1}{2}}] \, \quad\forall \,\, 1 \leq j \leq N_v.
	\]
	The mesh sizes $h_x$ and $h_v$ relative to the above partition are defined as follows: 
   \[
   0 < h_x = \max_{1 \leq i \leq N_x} h_i^x, \quad \mbox{where} \quad
   h_i^x = x_{i+\frac{1}{2}} - x_{i-\frac{1}{2}},
   \]
   and
	\[
	 0 < h_v = \max_{1 \leq j \leq N_v}h_j^v, \quad \mbox{where} \quad h_j^v = v_{j+\frac{1}{2}} - v_{j - \frac{1}{2}}.
	\]
	The mesh size of the partition $\mathcal{T}_h$ is defined as $h := \max(h_x,h_v)$. Here the mesh is assumed to be regular and quasi-uniform in the sense that there exist positive constants $c_1, c_2, c_1', c_2'$ such that the ratio of maximal and minimal mesh sizes stay bounded during mesh refinement in the following sense: 
	\[
	c_1 \leq \frac{h_x}{h_i^x} \leq c_2 \quad \mbox{and} \quad c_1' \leq \frac{h_v}{h_j^v} \leq c_2', \qquad \forall \quad 1 \leq i \leq N_x, 1 \leq j \leq N_v.
	\] 
	The set of all vertical edges  and all horizontal edges are denoted by $\Gamma_x$ and $\Gamma_v$, respectively,
	\[
	\Gamma_x := \bigcup\limits_{i,j}\left\{\{x_{i-\frac{1}{2}}\}\times J_j\right\}\quad \mbox{and} \quad \Gamma_v := \bigcup\limits_{i,j}\left\{I_i \times \{v_{j-\frac{1}{2}}\}\right\}.
	\]
	We denote the collection of all edges by $\Gamma_h := \Gamma_x \cup \Gamma_v$.
	We define the discontinuous finite element spaces for approximating $(u,f)$ as follows:
	\begin{equation*}
		\begin{aligned}
			X_h &:= \{\psi \in L^2(I) : \psi \in \mathbb{P}^k(I_i), \,\,\,\,i = 1,\dots,N_x\},
		\\
		V_h &:= \{\psi \in L^2(J) : \psi \in \mathbb{P}^k(J_j),\, \, \,\,j = 1,\dots,N_v\},
	\\
		\mathcal{Z}_h &:= \left\{ \phi \in L^2(\Omega) : \phi \in \mathbb{Q}^k(T_{ij}),\,  \, i = 1,\dots, N_x;\,\,j = 1, \dots, N_v \right\},
	\end{aligned}
\end{equation*}
	where $\mathbb{P}^k$ is the space of scalar polynomials of degree at most $k$ and $\mathbb{Q}^k(T_{ij})$ is the space of tensor product of polynomials of degrees at most $k$ in each variable.
	
	Below, we define the jump and average values of a function at nodal points. Let $\left(\phi_h\right)^+_{i+\frac{1}{2},v}$ and $\left(\phi_h\right)^-_{i+\frac{1}{2},v}$ be the values of $\phi_h$ at $\left(x_{i+\frac{1}{2}},v\right)$ from the right cell $I_{i+1}\times J_j$ and from the left cell $I_i\times J_j$, respectively. More precisely
	\[
	\left(\phi\right)^+_{i+\frac{1}{2},v} = \lim_{\eth \to 0^+} \phi_h\left(x_{i+\frac{1}{2}}+\eth,v\right) \quad \mbox{and} \quad \left(\phi\right)^-_{i+\frac{1}{2},v} = \lim_{\eth \to 0^+} \phi_h\left(x_{i+\frac{1}{2}}-\eth,v\right).
	\]
	Similarly, we set $\left(\phi_h\right)^+_{x,j+\frac{1}{2}}$ and $\left(\phi_h\right)^-_{x,j-\frac{1}{2}}$.
	  The jump $[\![ \cdot ]\!]$ and average $\{\cdot\}$ of $\phi_h$ at $(x_{i+\frac{1}{2}},v),\,\,\forall\, v \in J_j$ are defined by
	\begin{align*}
		[\![\phi_h]\!]_{i+\frac{1}{2},v} &:= \left(\phi_h\right)^+_{i+\frac{1}{2},v} - \left(\phi_h\right)^-_{i+\frac{1}{2},v} \quad\forall \, \phi_h \in \mathcal{Z}_h,
		\\
		\{\phi_h\}_{i+\frac{1}{2},v} &:= \frac{1}{2}\left(\left(\phi_h\right)^+_{i+\frac{1}{2},v} + \left(\phi_h\right)^-_{i+\frac{1}{2},v}\right) \quad\forall \, \phi_h \in \mathcal{Z}_h.
	\end{align*}
Similarly, one can define jump and average at $(x,v_{j+1/2}),\,\,\forall\, x \in I_i$.

	\textbf{Discrete norm:} We define the following discrete semi-norms and norms: 
	\[
	|w|_{m,\mathcal{T}_h} = \left(\sum_{R \in \mathcal{T}_h}|w|^2_{m,R}\right)^\frac{1}{2}, \,\|w\|_{m,\mathcal{T}_h} = \left(\sum_{R \in \mathcal{T}_h}\|w\|^2_{m,R}\right)^\frac{1}{2} \quad \forall\,w \in H^m(\mathcal{T}_h),\,\, m \geq 0,
	\]
	\[
	\|w\|_{\infty,\mathcal{T}_h} = \sup_{R \in \mathcal{T}_h}\|w\|_{L^{\infty}(R)},\quad \|w\|_{L^p(\mathcal{T}_h)} = \left(\sum_{R \in \mathcal{T}_h}\|w\|^p_{L^p(R)}\right)^\frac{1}{p},\,\,\forall\,w \in L^p(\mathcal{T}_h), 
	\]
	for all $1 \leq p < \infty$.
	
	Next, we recall some standard estimates which are frequently used in our analysis:\\
	\textbf{Inverse inequality:} (see \cite[Lemma 1.44, p. 26]{2}) If $w_h \in \mathbb{P}^k(I_i)$, then
	\begin{equation}\label{inverseeqn}
		\|\partial_xw_h\|_{0,I_i} \leq Ch_x^{-1}\|w_h\|_{0,I_i}.
	\end{equation} 
	\textbf{Trace inequality:} (see \cite[Lemma 1.46, p. 27]{2}) For $w_h \in \mathbb{P}^k(I_i)$,
	\begin{equation}\label{traceeqn}
		\|w_h\|_{0,\partial I_i} \leq Ch_x^{-\frac{1}{2}}\|w_h\|_{0,I_i}.
	\end{equation}
\textbf{Norm comparison:} (see \cite[Lemma 1.50, p. 29]{2}) Let $1\leq p,q \leq \infty$ and $w_h \in \mathbb{P}^k(I_i)$. Then,
\begin{equation}\label{normcompeqn}
	\|w_h\|_{L^p(I_i)} \leq Ch_x^{\frac{1}{p}-\frac{1}{q}}\|w_h\|_{L^q(I_i)}.
\end{equation}
	
	\subsection{LDG formulation}
	
	We rewrite equation \eqref{contburger} by introducing an auxiliary variable $w = \sqrt{\epsilon}\,\partial_x u$ as follows:
	\begin{equation}\label{ctsbur}
		\partial_t u + \frac{1}{2}\partial_x u^2 - \sqrt{\epsilon}\,\partial_x w + \rho u = \rho V,
	\end{equation} 
	\begin{equation}\label{ctsburg}
		w - \sqrt{\epsilon}\, \partial_x u = 0.
	\end{equation}
	This  helps to devise  LDG scheme for the Burgers' equation. We denote by $\left(u_h(t),w_h(t)\right) \in X_h \times X_h$, a discrete approximation for $\left(u(t),w(t)\right)$ for all $t \in [0,T]$ and we denote by $f_h(t) \in \mathcal{Z}_h$, a discrete approximation for $f(t)$ for all $t \in [0,T]$. As in the continuum setting, we set discrete local density and discrete momentum by 
	\begin{equation}\label{rh}
		\rho_h = \sum_{j = 1}^{N_v}\int_{J_j}f_h\,{\rm d}v \qquad \mbox{and} \qquad
		\rho_hV_h = \sum_{j = 1}^{N_v}\int_{J_j}vf_h\,{\rm d}v,
	\end{equation}
respectively. Our discrete problem is to seek
 $(u_h(t),w_h(t),f_h(t)) \in X_h \times X_h \times \mathcal{Z}_h$, for $t \in [0,T]$ such  that
	\begin{equation}{\label{bh}}
		\left(\frac{\partial f_h}{\partial t},\psi_h\right) + \mathcal{B}_{h}(u_h ; f_{h},\psi_{h}) = 0  \hspace{2mm} \forall \,\,\psi_{h} \in \mathcal{Z}_h,
	\end{equation}
\begin{equation}\label{ch}
	\begin{aligned}
		\left(\frac{\partial u_h}{\partial t}, \phi_h\right) + a_h(u_h,\phi_h) + \sqrt{\epsilon} \,b_h(w_h,\phi_h) &+ \left(\rho_h u_h, \phi_h\right) 
		\\
		&= \left(\rho_h V_h, \phi_h\right) \quad \forall\,\, \phi_h \in X_h,
	\end{aligned}
\end{equation}
\begin{equation}\label{ch1}
	\left(w_h,q_h\right) + \sqrt{\epsilon}\,b_h(u_h,q_h) = 0 \quad \forall \,\, q_h \in X_h,
\end{equation}
	with $f_h(0) = f_{0h} \in \mathcal{Z}_h$ and $u_h(0) = u_{0h} \in X_h$ to be defined later.	
	In \eqref{bh}, 
	\begin{equation}
		\mathcal{B}_h(u_h;f_h,\psi_h) := \sum_{i = 1}^{N_x}\sum_{j=1}^{N_v}\mathcal{B}_{ij}^h(u_h;f_h,\psi_h),
	\end{equation}
	with
	\begin{equation}\label{bhdef}
		\begin{aligned}
			\mathcal{B}_{ij}^h(u_h;f_{h},\psi_{h}) & :=  - \int_{T_{ij}}vf_h\partial_x\psi_{h} \, {\rm d}v\, {\rm d}x - \int_{T_{ij}}f_h(u_h - v)\partial_v\psi_h \, {\rm d}v\, {\rm d}x 
			\\
			& \quad + \int_{J_j}\left[\left(\widehat{vf_h}\psi_h^-\right)_{i+1/2,v} - \left(\widehat{vf_h}\psi_h^+\right)_{i-1/2,v}\right]\, {\rm d}v
			\\
			& \quad +\int_{I_i}\left[\left(\reallywidehat{(u_h - v)f_h}\psi_h^-\right)_{x,j+1/2} - \left(\reallywidehat{(u_h - v)f_h}\psi_h^+\right)_{x,j-1/2}\right]\, {\rm d}x
		\end{aligned}
	\end{equation}
	wherein the generalized numerical fluxes are
   \begin{equation}\label{flux}
   	  \left\{
   	\begin{aligned}
   	\widehat{vf_h}\big\vert_{\{x_{i-1/2}\}\times J_j} &:= \left\{
   	\begin{aligned}
   		& \left( 1 - \lambda_1\right)vf_h^+ + \lambda_1vf_h^- \quad \mbox{if}\quad v \geq 0 
   		\\
   		& \left( 1 - \lambda_1\right)vf_h^- + \lambda_1vf_h^+ \quad \mbox{if} \quad v < 0 
   	\end{aligned}
   \right.
   \\
   &= \left\{vf_h\right\} + \left(\frac{1 - 2\lambda_1}{2}\right)\left\vert v\right\vert[\![f_h]\!] \hspace{3mm} \text{on}\hspace{2mm} \{x_{i-1/2}\}\times J_j ,
   \\
   \reallywidehat{(u_h - v)f_h}\big\vert_{I_i \times \{v_{j-1/2}\}} &:= \left\{
   \begin{aligned}
   	 &\left(1 - \lambda_2\right)\left(u_h - v\right)f_h^+ + \lambda_2\left(u_h - v\right)f_h^- \quad\mbox{if}\quad \left(u_h - v\right) \geq 0
   	 \\
   	 &\left(1 - \lambda_2\right)\left(u_h - v\right)f_h^- + \lambda_2\left(u_h - v\right)f_h^+ \quad\mbox{if}\quad \left(u_h - v\right) < 0
   \end{aligned}
\right.
\\
&= \left\{\left(u_h - v\right)f_h\right\} + \left(\frac{1 - 2\lambda_2}{2}\right)\left\vert u_h - v\right\vert[\![f_h]\!] \hspace{3mm} \text{on}\hspace{2mm} I_i\times \{v_{j-1/2}\},
\end{aligned}
\right.
   \end{equation}
with $\lambda_1,\lambda_2 > 1/2$. We define the numerical fluxes on the boundary $\partial\Omega$ by
\[
\left(\widehat{vf_h}\right)_{1/2,v} = \left(\widehat{vf_h}\right)_{N_x+1/2,v}, \qquad \left(\reallywidehat{(u_h - v)f_h}\right)_{x,1/2} = \left(\reallywidehat{(u_h - v )f_h}\right)_{x,N_v+1/2} = 0, 
\]
for all $ (x,v) \in I_i \times J_j$ for $1 \leq i \leq N_x, 1 \leq j \leq N_v$.
 \begin{rem}
		Even though classical purely upwind fluxes are employed in DG schemes for linear hyperbolic equations, it is not easy to define such fluxes in the presence of variable coefficients. Lately, generalized numerical fluxes similar to \eqref{flux} have been used in such scenarios, thanks to the simplicity in their definition \cite{liu2020optimal}. Furthermore, such generalized fluxes provide more flexibility in dealing with small viscosity coefficient in the viscous Burgers' equation, see, for  some comments in (ii)  of the   Observations in the  Section 5.
\end{rem}
  In \eqref{ch}-\eqref{ch1},
	\begin{equation}\label{ah}
		a_h(u_h,\phi_h) := -\sum_{i = 1}^{N_x}\int_{I_i}\frac{u_h^2}{2}\partial_x\phi_h\,{\rm d}x - \sum_{i = 0}^{N_x-1}\left(\frac{\widehat{u_h^2}[\![\phi_h]\!]}{2}\right)_{i+1/2}
	\end{equation}
and
	\begin{equation}\label{b_h}
		b_h(w_h,\phi_h) := \sum_{i = 1}^{N_x}\int_{I_i}w_h\partial_x \phi_h \,{\rm d}x + \sum_{i = 0}^{N_x-1}\left(\widehat{w_h}[\![\phi_h]\!]\right)_{i+1/2}
	\end{equation}
	with the numerical fluxes 
	\begin{equation}\label{fu2}
		\widehat{u_h^2} := \frac{1}{3}\left(\left(u_h^+\right)^2 + u_h^+u_h^- + \left(u_h^-\right)^2\right),
	\end{equation}
	\begin{equation}\label{fw}
		\widehat{w_h} := \left(1 - \lambda\right)w_h^- + \lambda w_h^+ = \{w_h\} + \left(\frac{2\lambda - 1}{2}\right)[\![w_h]\!] ,
	\end{equation}
and
	\begin{equation}\label{fu}
		\widehat{u_h} := \lambda u_h^- + \left(1 - \lambda\right)u_h^+ = \{u_h\} + \left(\frac{1 - 2\lambda}{2}\right)[\![u_h]\!],
	\end{equation}
	where $\lambda \geq 1/2$. Note that the numerical fluxes given in \eqref{fw}-\eqref{fu} are in a generalized sense. Note further that the numerical fluxes at the end points of $I$ are defined by $\left(u_h\right)^+_{\frac{1}{2}} = \left(u_h\right)^+_{N_x+\frac{1}{2}}, \left(u_h\right)^-_{\frac{1}{2}} = \left(u_h\right)^-_{N_x + \frac{1}{2}}, \left(w_h\right)^+_{\frac{1}{2}} = \left(w_h\right)^+_{N_x + \frac{1}{2}}$ and $\left(w_h\right)^-_{\frac{1}{2}} = \left(w_h\right)^-_{N_x + \frac{1}{2}}$. The $\widehat{u_h^2}$ is not a generalized numerical flux and it is referred as central flux which was introduced by Liu et al. in \cite{liu2015}.   

	\begin{rem}
		The author of \cite{ploymaklam2019local} has run a large time simulation to understand the stability of a numerical scheme stemming out of the central flux \eqref{fu2} while comparing it with the stability of a numerical scheme associated with the following Lax-Friedrich flux: 
		\[
		\widehat{u_h^2} = \frac{1}{2}[(u_h^+)^2 + (u_h^-)^2 - 2\max|u| (u_h^+ - u_h^-)],
		\]   
		see, section $3$ (Example $3.3$) of \cite{ploymaklam2019local} for more details. The author observes that the scheme defined by using the central fluxes is more stable than the scheme that uses the Lax-Friedrich flux.
	\end{rem}

We observe below, certain properties of the bilinear form $b_h(w_h,\phi_h)$:
	\begin{itemize}
		\item For $w_h \in X_h$,
		\begin{equation}
			\begin{aligned}
				b_h(w_h,w_h) &= \sum_{i = 1}^{N_x}\int_{I_i}w_h\partial_x w_h \,{\rm d}x + \sum_{i = 0}^{N_x-1}\left(\widehat{w_h}[\![w_h]\!]\right)_{i+1/2} 
				\\
				& = \sum_{i = 0}^{N_x-1}\left(\widehat{w_h}[\![w_h]\!] - \frac{1}{2}[\![w_h^2]\!]\right)_{i+1/2}
				= \left(\lambda - \frac{1}{2}\right)\sum_{i = 0}^{N_x-1}[\![w_h]\!]^2_{i+1/2}.
			\end{aligned}
		\end{equation}
		\item For $w_h, \phi_h \in X_h$, there holds
		\begin{equation}\label{bhp}
			\begin{aligned}
				b_h(w_h,\phi_h) &= \sum_{i = 1}^{N_x}\int_{I_i}w_h\partial_x\phi_h \,{\rm d}x + \sum_{i = 0}^{N_x-1}\left(\widehat{w_h}[\![\phi_h]\!]\right)_{i+1/2}
				\\
				&= -\sum_{i = 1}^{N_x}\int_{I_i}\partial_x w_h \,\phi_h\,{\rm d}x - \sum_{i = 0}^{N_x-1}\left([\![w_h\phi_h]\!] - \widehat{w_h}[\![\phi_h]\!]\right)_{i+1/2}
				\\
				&= -b_h(\phi_h,w_h) - \sum_{i = 0}^{N_x-1}\left([\![w_h\phi_h]\!] - \widehat{w_h}[\![\phi_h]\!] - \widehat{\phi_h}[\![w_h]\!]\right)_{i+1/2}. 
			\end{aligned}
		\end{equation}
	\end{itemize}
		
	Since $X_h\times X_h \times \mathcal{Z}_h$ is finite dimensional, the discrete problem \eqref{bh}-\eqref{ch1} leads to a system of non-linear ODEs coupled with linear algebraic equations which is known as a system of differential algebraic equations. From equation \eqref{ch1} $w_h$ can be written explicitly as a function of $u_h$. On substitution in \eqref{ch} we obtain a system of non-linear ODEs. An application of Picard's theorem shows the existence of a local-in-time unique solution $(u_h,w_h,f_h)$. This solution can be made global-in-time provided we have bounds on the solution in appropriate norms. We shall be commenting on this towards the end of the paper.

	\subsection{Some properties of the discrete solution}\label{dgp}
	
	\begin{lem}[Discrete mass conservation]\label{lem:mass}
		Let $k \geq 0$ and let $f_h$ be the DG approximation to $f$ satisfying $\eqref{bh}$ with $f_{h}(0) = \mathcal{P}_hf_0$, where $\mathcal{P}_h$ is the $L^2$-projection onto $\mathcal{Z}_h$. Then, the following discrete mass conservation property holds: 
		\begin{equation*}
			\begin{aligned}
				\int_{\Omega}f_h(t,x,v)\,{\rm d}x\,{\rm d}v  = \int_{\Omega}f_0(x,v)\,{\rm d}x\,{\rm d}v \quad \forall\,\, t > 0.
			\end{aligned}
		\end{equation*}
	\end{lem}
	
	\begin{proof}
		From the definition of the $L^2$-projection, it follows that
		\begin{equation*}{\label{l2}}
			\begin{aligned}
				\sum_{i,j}\int_{T_{ij}} f_h(0)\,{\rm d}x\,{\rm d}v = \sum_{i,j}\int_{T_{ij}}\mathcal{P}_hf_0\,{\rm d}x\,{\rm d}v = \sum_{i,j}\int_{T_{ij}} f_0\,{\rm d}x\,{\rm d}v.
			\end{aligned}
		\end{equation*}
		Let us fix an arbitrary element $T_{ij}$ and let us take a test function $\psi_h$ in $\eqref{bh}$ such that $\psi_h = 1$ in $T_{ij}$ and $\psi_h = 0$ elsewhere. Then, \eqref{bh} reduces to 
		\begin{equation*}
			\begin{aligned}
				\int_{T_{ij}} \frac{\partial f_h}{\partial t}\,{\rm d}x\,{\rm d}v + \mathcal{B}_{ij}^h(u_h;f_h,1) &=0.
			\end{aligned}
		\end{equation*} 
		From the definition of $\mathcal{B}_{ij}^h$, we arrive at
		\begin{equation*}
			\begin{aligned}
				\int_{T_{ij}}\frac{\partial f_h}{\partial t}\,{\rm d}x\,{\rm d}v  
				&+ \int_{J_j}\left[\left(\widehat{vf_h}\right)_{i+1/2,v} - \left(\widehat{vf_h}\right)_{i-1/2,v}\right]\,{\rm d}v		 	
				\\
				&+ \int_{I_i}\left[\left(\reallywidehat{(u_h-v)f_h}\right)_{x,j+1/2} - \left(\reallywidehat{(u_h-v)f_h}\right)_{x,j-1/2}\right]\,{\rm d}x = 0.
			\end{aligned}
		\end{equation*}
		Note that the choice of $T_{ij}$ was done arbitrarily. Furthermore, the second and the third integral terms on the left hand side holds true for all $i,j$. Since boundary conditions in $x$ are periodic and the support in $v$ is compact, taking summation over all $i,j$ and integration in time yields the desired result. 
	\end{proof}
	As a consequence of the above lemma, for any given non-negative initial data, we have
	\[
	\int_\Omega f_h(t,x,v){\rm d}x\,{\rm d}v \geq 0 \quad\mbox{and}\quad \int_{I}\rho_h(t,x){\rm d}x \geq 0.
	\]
	

	\begin{lem}[Discrete total momentum conservation]\label{lem:momentum}
		Let $(f_h,u_h) \in \mathcal{C}^1([0,T];\mathcal{Z}_h\times X_h)$ be the DG-LDG approximation obtained as a solution of \eqref{bh} and \eqref{ch}. Then, for $k \geq 1$,
		\begin{equation*}
			\int_{\Omega}vf_h(t,x,v)\,{\rm d}x\,{\rm d}v + \int_Iu_h(t,x)\,{\rm d}x = \int_{\Omega}vf_0(x,v)\,{\rm d}x\,{\rm d}v + \int_I u_0(x)\,{\rm d}x \quad \forall\,\, t \geq 0.
		\end{equation*}
	\end{lem}
	
	\begin{proof}
		Choose $\phi_h = 1 \in X_h$ in \eqref{ch}, and use the definitions \eqref{rh} to obtain
		\begin{equation*}\label{uu}
			\sum_{i}\int_{I_i}\frac{\partial u_h}{\partial t}\,{\rm d}x = \sum_{i}\int_{I_i}\left(\rho_hV_h - \rho_hu_h\right)\,{\rm d}x = \sum_{i,j}\int_{T_{ij}}\left( v - u_h\right)f_h\,{\rm d}x\,{\rm d}v.
		\end{equation*}
		Now, putting $\psi_h = v \in \mathcal{Z}_h$ in equation \eqref{bh}, we arrive at
		\begin{equation*}
			\sum_{i,j}\int_{T_{ij}}\left(\frac{\partial}{\partial t}(vf_h) - \left(u_h - v\right)f_h\right)\,{\rm d}x\,{\rm d}v = 0.
		\end{equation*}
		Adding the above two expressions followed by an integration in time yields the result.
	\end{proof}


	\begin{lem}[Discrete total energy identity]
		Let $k \geq 2$ and let $(f_h,u_h,w_h) \in \mathcal{C}^1(0,T;\mathcal{Z}_h)\times \mathcal{C}^1(0,T;X_h)\times \mathcal{C}^0(0,T;X_h)$ be the DG-LDG approximate solution of \eqref{bh} and \eqref{ch}-\eqref{ch1}. Then, 
		\begin{equation*}
			\begin{aligned}
				&\frac{1}{2}\left(\sum_{i,j}\int_{T_{ij}} v^2\,f_h\,{\rm d}x\,{\rm d}v + \sum_{i=1}^{N_x}\int_{I_i}\,u_h^2\,{\rm d}x \right) + \sum_{i=1}^{N_x}\int_{0}^{t}\int_{I_i}w^2_h\,{\rm d}x 
				\\
				&+ \sum_{i,j}\int_{0}^{t}\int_{T_{ij}}\left(u_h - v\right)^2f_h\,{\rm d}x\,{\rm d}v 
				= \frac{1}{2}\left(\sum_{i,j}\int_{T_{ij}}v^2\,f_h(0)\,{\rm d}x\,{\rm d}v + \sum_{i=1}^{N_x}\int_{I_i}u_h^2(0)\,{\rm d}x\right) \,\, \forall \,\, t\geq 0.
			\end{aligned}
		\end{equation*}
	\end{lem}
	
	\begin{proof}
		A choice of \textcolor{blue}{$\psi_h = \frac{v^2}{2} \in \mathcal{Z}_h$} in equation \eqref{bh} yields
		\begin{equation}\label{n}
			\sum_{i,j}\int_{T_{ij}}\left(\frac{v^2}{2}\frac{\partial f_h}{\partial t} + v^2f_h - u_hvf_h\right)\,{\rm d}x\,{\rm d}v = 0.
		\end{equation}
		Choose $\phi_h = u_h$ and $q_h = w_h$ in equation \eqref{ch}-\eqref{ch1} to obtain
		\begin{equation}\label{n1}
			\begin{aligned}
				\frac{1}{2}\frac{{\rm d}}{{\rm d}t}\|u_h\|^2_{0,I_h} + \|w_h\|^2_{0,I_h} &+ a_h(u_h,u_h) + \sqrt{\epsilon}\,b_h(w_h,u_h) 
				\\
				&+ \sqrt{\epsilon}\,b_h(u_h,w_h) + \left(\left(\rho_h u_h - \rho_h V_h\right), u_h\right) = 0.
			\end{aligned}
		\end{equation}
		From equation \eqref{ah}, we obtain
		\begin{equation*}\label{n2}
			\begin{aligned}
				a_h(u_h,u_h) = \sum_{i = 0}^{N_x-1}\left(\frac{[\![u_h^3]\!]}{6} - \frac{\widehat{u_h^2}[\![u_h]\!]}{2}\right)_{i+1/2} = 0,
			\end{aligned}
		\end{equation*}
		thanks to the flux defined in \eqref{fu2} and the periodic boundary condition.\\
		From the fluxes defined in \eqref{fw} and \eqref{fu}, we deduce that
		\begin{equation}\label{bzero}
			b_h(u_h,w_h) + b_h(w_h,u_h) = 0.
		\end{equation}
		Hence, summing the equations \eqref{n} and \eqref{n1} yields
		\begin{equation*}
			\begin{aligned}
				\frac{1}{2}\frac{{\rm d}}{{\rm d}t}\left(\sum_{i,j}\int_{T_{ij}} v^2\,f_h\,{\rm d}x\,{\rm d}v + \sum_{i=1}^{N_x}\int_{I_i}u_h^2\,{\rm d}x\right) &+ \sum_{i=1}^{N_x}\int_{I_i}w_h^2\,{\rm d}x 
				\\
				&+ \sum_{i,j}\int_{T_{ij}}\left(u_h - v\right)^2f_h\,{\rm d}x\,{\rm d}v = 0.
			\end{aligned}
		\end{equation*}
		An integration in time leads to the desired identity.
	\end{proof}
	
	In the discrete case, it is difficult to prove that the total energy dissipates as in \eqref{energy}, because it is hard to show the non-negativity of $f_h$. In the continuum case, this dissipation property plays a crucial role in the proof of well-posedness of the system \eqref{eq:continuous-model}-\eqref{contburger}.


	\begin{lem}\label{fhbd}
		Let $k \geq 0$ and let $f_h$ be the DG approximation to $f,$ satisfying $\eqref{bh}.$ Then, 
		\begin{equation*}
			\max_{t \in [0,T]}\|f_h(t,\cdot)\|_{0,\mathcal{T}_h} \leq e^{\frac{T}{2}}\|f_0\|_{0,\mathcal{T}_h}.
		\end{equation*}
	\end{lem}
	
	\begin{proof}
		Choosing $\psi_h = f_h$ in \eqref{bh}, we obtain	
		\begin{equation*}
			\begin{aligned}
				&\frac{1}{2}\frac{{\rm d}}{{\rm d}t}\|f_h\|^2_{0,\mathcal{T}_h} - \sum_{i,j}\left( \frac{1}{2}\int_{T_{ij}}\left(v\partial_xf_h^2 + (u_h - v)\partial_vf_h^2\right)\,{\rm d}x\,{\rm d}v\right) 
				\\
				&-\sum_{i,j}\int_{J_j}\left(\widehat{v f_h}[\![f_h]\!]\right)_{i-1/2,v}\,{\rm d}v - \sum_{i,j}\int_{I_i}\left(\reallywidehat{(u_h - v)f_h}[\![f_h]\!]\right)_{x,j-1/2}\,{\rm d}x = 0.
			\end{aligned}
		\end{equation*}
		After applying integration by parts in second and third term and using $[\![f_h^2]\!] = 2\{f_h\}[\![f_h]\!]$, it follows that
		\begin{equation}\label{aa}
			\begin{aligned}
				&\frac{1}{2}\frac{{\rm d}}{{\rm d}t}\|f_h\|^2_{0,\mathcal{T}_h} - \frac{1}{2}\|f_h\|^2_{0,\mathcal{T}_h} + \sum_{i,j}\int_{J_j}\left(\frac{2\lambda_1 - 1}{2}\right)\left\vert v\right\vert[\![f_h]\!]_{i-1/2,v}^2\,{\rm d}v 
				\\
				&+ \sum_{i,j}\int_{I_i}\left(\frac{2\lambda_2 - 1}{2}\right)\left(\left\vert u_h - v\right\vert[\![f_h]\!]^2\right)_{x,j-1/2}\,{\rm d}x = 0.
			\end{aligned}
		\end{equation}
		Observe that for $\lambda_1,\lambda_2 > 1/2$, last two terms on the left hand side are non-negative and hence, can be dropped. An integration in time yields the desired result.
	\end{proof}


	For our subsequent use, we shall need the following lemma.
	\begin{lem}\label{rhoh2}
		Let $f$ and $f_h$ be the continuum and the discrete solutions of the Vlasov - viscous Burgers' system, respectively. Further, let $\rho$ be the local density associated with $f$ and $\rho_h$ be the discrete local density defined as in \eqref{rh}. Then, 
		\begin{equation*}
			\|\rho - \rho_h\|_{0,I_h} \leq (2M)^\frac{1}{2}\|f - f_h\|_{0,\mathcal{T}_h}
		\quad \mbox{and} \quad
			\|\rho - \rho_h\|_{\infty,I_h} \leq 2M\|f - f_h\|_{\infty,\mathcal{T}_h}.
		\end{equation*}
		Moreover, 		
		\begin{equation*}
			\|\rho V - \rho_h V_h\|_{0,I_h} \leq 2M\|f - f_h\|_{0,\mathcal{T}_h}.
		\end{equation*}		
	\end{lem}
	
	\begin{proof}
		An application of the H\"older inequality yields
		\begin{equation*}
			\begin{aligned}
				\|\rho - \rho_h\|^2_{0,I_h} &= \sum_{i = 1}^{N_x}\int_{I_i}\left(\sum_{j = 1}^{N_v}\int_{J_j}\left(f - f_h\right){\rm d}v\right)^2{\rm d}x
				\\
				&\leq  2M\sum_{i,j}\int_{T_{ij}}\left(f - f_h\right)^2{\rm d}v\,{\rm d}x
			\end{aligned}
		\end{equation*} 
		By definition, it follows that
		\begin{equation*}
			\begin{aligned}
				\rho - \rho_h = \int_{J}(f - f_h)\,{\rm d}v \leq 2M\|f - f_h\|_{L^\infty(J)}
			\end{aligned}
		\end{equation*}
		and hence, our second result. An application of the H\"older inequality shows
		\begin{equation*}
			\begin{aligned}
				\|\rho V - \rho_h V_h\|^2_{0,I_h} &= \sum_{i = 1}^{N_x}\int_{I_i}\left(\sum_{j = 1}^{N_v}\int_{J_j}\left(f - f_h\right)v\,{\rm d}v\right)^2{\rm d}x
				\\
				&\leq  4M^2\sum_{i,j}\int_{T_{ij}}\left(f - f_h\right)^2{\rm d}v\,{\rm d}x.
			\end{aligned}
		\end{equation*} 
		This concludes the proof.
	\end{proof}
	
	As a consequence of Lemma \ref{fhbd}, it follows that
	\begin{equation}\label{rho1}
		\|\rho_h\|_{0,I_h} \leq C(T)M\|f_0\|_{0,\mathcal{T}_h}.
	\end{equation}

	\section{A priori Estimates}\label{err}
This section discusses some a priori error estimates for the discrete solution. In order to derive an optimal order of convergence, we adopt the following strategy:
	\begin{itemize}
		\item Using a projection operator $Q_\lambda^x$ introduced by Liu et al. in \cite{liu2015} and its approximation property, we obtain a bound on $\|Q_\lambda^xu - u_h\|_{0,I_h}$, (See, Corollary \ref{cor:uL2}).
		\item Inspired by a projection operator introduced by Liu et. al. in \cite{liu2020optimal}, we define a new projection $\Pi_{\lambda_1,\lambda_2}$ which in turn helps us to obtain a bound on $\|\Pi_{\lambda_1,\lambda_2}f - f_h\|_{0,\mathcal{T}_h}$ thanks to its approximation property, (See, Lemma \ref{fl2}).    
		\item The aforementioned bounds are such that the bound on $\|Q_\lambda^xu - u_h\|_{0,I_h}$ depends on $\|f - f_h\|_{0,\mathcal{T}_h}$ (See equations \eqref{uL2}-\eqref{uL21}).
		Furthermore, the bound on $\|\Pi_{\lambda_1,\lambda_2}f - f_h\|_{0,\mathcal{T}_h}$ depends on $\|u - u_h\|_{\infty,I_h}$, (See, Lemma \ref{K2estimate}).
		\item In Lemma \ref{L:uinf},  using inverse hypothesis, we obtain 
		\begin{equation*}
			\|u - u_h\|_{\infty,I_h} \lesssim h + h^{-\frac{1}{2}}\|u - u_h\|_{0,I_h}.
		\end{equation*}
		\item Then, an application of a non-linear version of Gr\"onwall type inequality yields optimal order of convergence, (See, Theorem \ref{mmma1}).
	\end{itemize}	
	By optimality, we mean optimality with respect to approximation properties of the projection operators employed in the proof.
	\subsection{Error estimates for viscous Burgers' system}
	
	This subsection deals with error estimates for the viscous Burgers' system.\\
	Let $\mathcal{P}_x : L^2(I_h) \to X_h$ be the standard $L^2$-projection.

	\textbf{Global projection.} Let $s \geq k+1$ and $\lambda \geq \frac{1}{2}$. Consider the projection $Q^x_\lambda : H^s(I_h) \rightarrow X_h$ defined by
	\begin{equation}\label{gp}
		\int_{I_i}\left((Q^x_\lambda\Upsilon) - \Upsilon\right)z\,{\rm d}x = 0, \quad \forall \,\, z \in \mathbb{P}^{k-1}(I_i), \,\,1 \leq i \leq N_x,
	\end{equation}
together with the flux relation
	\begin{equation}\label{gp1}
		\left(\widehat{Q^x_\lambda\Upsilon}\right)_{i+1/2} = \lambda\Upsilon^-_{i+1/2} + \left(1 - \lambda\right)\Upsilon^+_{i+1/2},\,\,  1 \leq i \leq N_x-1.
	\end{equation}
	The above projection $Q^x_\lambda$ is uniquely defined when $\lambda > 1/2$ (see \cite[Lemma 4.1]{liu2015}). Even when $\lambda = 1/2$, the projection $Q^x_\lambda$ is uniquely defined provided $k$ is even and $N_x$ is odd (again see \cite[Lemma 4.1]{liu2015} for the proof).
	
	Below, we recall some properties of the above defined projection (for proof refer to \cite[Lemmas 4.2-4.3, p. 331]{liu2015}.
	
	\textbf{Approximation properties of the global projection.}  
	For $\Upsilon\mid_{I_i} \in H^{k+1}(I_i)$ for $i = 1,2,\dots,N_x$, there exists a positive constant $C = C(k,\lambda)$, independent of $\Upsilon$ such that for $k \geq 0$,
	\begin{equation}\label{uprojection}
		\left\{
		\begin{aligned}
			\|Q^x_\lambda\Upsilon - \Upsilon\|_{0,I_h} &\leq C h_x^{k+1}|\Upsilon|_{k+1,I_h},
			\\
			\|Q^x_\lambda\Upsilon - \Upsilon\|_{0,\Gamma_x} &\leq Ch_x^{k+1/2}|\Upsilon|_{k+1,I_h}. 
		\end{aligned}
		\right.
	\end{equation}


	\begin{lem}\label{L:uinf}
		Let $u$ be the solution of viscous Burgers' problem \eqref{contburger}. Let $u_h \in X_h$ be its finite element approximation. Assume that $u \in W^{1,\infty}(I) \cap H^{k+1}(I)$. Then,
		\begin{equation*}\label{inftybound}
			\|u - u_h\|_{\infty,I_h} \lesssim h\|u\|_{W^{1,\infty}(I)} + h^{k+\frac{1}{2}}\|u\|_{k+1,I_h} + h^{-\frac{1}{2}}\|u - u_h\|_{0,I_h}.
		\end{equation*}
	\end{lem}
	
	\begin{proof}
		Observe that
		\begin{align*}
			\|u - u_h\|_{\infty,I_h} &\leq \|u - \mathcal{P}_x u\|_{\infty,I_h} + \|\mathcal{P}_x u - u_h\|_{\infty,I_h}
			\\
			&\lesssim h\|u\|_{W^{1,\infty}(I)} + h^{-\frac{1}{2}}\|\mathcal{P}_x u - u_h\|_{0,I_h}
			\\
			&\lesssim h\|u\|_{W^{1,\infty}(I)} + h^{-\frac{1}{2}}\|u - \mathcal{P}_x u\|_{0,I_h} + h^{-\frac{1}{2}}\|u - u_h\|_{0,I_h}
			\\
			&\lesssim h\|u\|_{W^{1,\infty}(I)} + h^{k+\frac{1}{2}}\|u\|_{k+1,I_h} + h^{-\frac{1}{2}}\|u - u_h\|_{0,I_h}.
		\end{align*}
		Here, in the first step we have employed the triangle inequality. In the second step, the first term is a consequence of the projection estimate \cite{brenner2007mathematical} and the second term is a consequence of the norm comparison inequality \eqref{normcompeqn}. Triangle inequality is applied again in the third step. Finally in the fourth step, we use projection estimate for the second term.
	\end{proof}


	\textbf{Error equation for the viscous Burgers' system:} Since the scheme with fluxes \eqref{fu2}-\eqref{fu} is consistent, \eqref{ch}-\eqref{ch1} also hold for solution $(u,w)$. Hence, taking the difference, we obtain the error equation
	\begin{equation}\label{erru}
		\begin{aligned}
			\left(\frac{\partial e_u}{\partial t},\phi_h\right) + a_h(u,\phi_h) &- a_h(u_h,\phi_h) + \sqrt{\epsilon}\,b_h(e_w,\phi_h) 
			\\
			&+ \left(\rho u - \rho_h u_h, \phi_h\right) = \left(\rho V - \rho_h V_h, \phi_h\right) \quad \forall \,\, \phi_h \in X_h,
		\end{aligned}
	\end{equation}
	\begin{equation}\label{errw}
		\left(e_w, q_h\right) + \sqrt{\epsilon}\,b_h(e_u,q_h) = 0 \quad \forall \,\,q_h \in X_h,
	\end{equation}
	where 
	\[
	e_u = u - u_h, \quad e_w = w - w_h.
	\]
	Using the projection operator, we rewrite
	\begin{equation}\label{p1}
		\begin{aligned}
			e_u &:= u - u_h = (u - Q^x_\lambda u) + (Q^x_\lambda u - u_h) = \theta_u - \eta_u,
			\\
			e_w &:= w - w_h  = (w - Q^x_{1 - \lambda} w) + (Q^x_{1 - \lambda}w - w_h) = \theta_w - \eta_w.
		\end{aligned}
	\end{equation}
	Here
	\begin{equation*}
		\begin{aligned}
			Q^x_\lambda u - u =: \eta_u,  \quad Q^x_\lambda u - u_h =: \theta_u, \quad	
			Q^x_{1 - \lambda}w - w =: \eta_w, \quad Q^x_{1 - \lambda}w - w_h =: \theta_w .
		\end{aligned}
	\end{equation*}
Using \eqref{p1} and the definition of the projection $Q^x_\lambda$ given by \eqref{gp}-\eqref{gp1}, we can rewrite the error equation \eqref{erru} and \eqref{errw} as
\begin{equation}\label{erru1}
	\begin{aligned}
		\left(\partial_t\theta_u,\phi_h\right) + \sqrt{\epsilon}b_h(\theta_w,\phi_h) &= \left(\rho V - \rho_hV_h,\phi_h\right) + \left(\partial_t\eta_u,\phi_h\right) - \left(\rho\theta_u.\phi_h\right) + \left(\rho\eta_u,\phi_h\right)
		\\
		& + \left(\left(\rho - \rho_h\right)\theta_u,\phi_h\right) - \left(\left(\rho - \rho_h\right)\eta_u,\phi_h\right) - \left(\left(\rho - \rho_h\right)u,\phi_h\right) 
		\\
		&- a_h(u,\phi_h) + a_h(u_h,\phi_h), \quad \forall\,\, \phi_h \in X_h,
	\end{aligned}
\end{equation}
\begin{equation}\label{errw1}
	\left(\theta_w,q_h\right) + \sqrt{\epsilon} b_h(\theta_u,q_h) = \left(\eta_w,q_h\right), \quad \forall \,\, q_h \in X_h.
\end{equation}

\begin{rem}
	Note that \eqref{gp1} yields
	\[
	\left(\widehat{\eta_u}\right)_{i+1/2} = \{\eta_u\}_{i+1/2} + \left(\frac{1 - 2\lambda}{2}\right)[\![\eta_u]\!]_{i+1/2} = 0,
	\]
	which implies
	\begin{equation}\label{lambdadep}
		\{\eta_u\}_{i+1/2} = \left(\frac{2\lambda - 1}{2}\right)[\![\eta_u]\!]_{i+1/2}.
	\end{equation}
	Therefore, if $\lambda = 1/2$ then $\{\eta_u\} = 0$ and if  $\lambda > 1/2$ then $ \{\eta_u\} \neq 0$.
\end{rem}

If $\lambda > 1/2$, then the following result holds (for detailed proof, see \cite[Lemma 2.3, p. 2085]{wang2019implicit}).

\begin{lem}\label{conj}
	Let $\lambda > 1/2$. Suppose $\left(\theta_u,\theta_w\right) \in X_h \times X_h$ satisfy \eqref{errw1}. Then, there is a positive constant $C$, independent $h$ and $\epsilon$, such that
	\begin{equation}\label{conjecture}
		\|\partial_x\theta_u\|_{0,I_h} + \sum_{i = 0}^{N_x-1}h^{-\frac{1}{2}}[\![\theta_u]\!]_{i+1/2} \leq C\epsilon^{-\frac{1}{2}}\|\theta_w\|_{0,I_h}.
	\end{equation}
\end{lem} 
	
	
	\begin{lem}\label{ul2}
		Let $(u,w)$ be the solution of the viscous Burgers' equation \eqref{ctsbur}-\eqref{ctsburg}. Let $(u_h,w_h) \in X_h \times X_h$ solve \eqref{ch}-\eqref{ch1}. Let $u \in L^\infty(0,T;H^{k+1}(I))$. Then, there exists a positive constant $C$ independent of $h$ such that for all $t \in (0,T]$,
		\begin{equation}\label{4.11}
			\begin{aligned}
				\frac{1}{2}\frac{{\rm d}}{{\rm d}t}\|\theta_u\|^2_{0,I_h} &+ \frac{1}{2}\|\theta_w\|^2_{0,I_h} \leq C\left(h^{2k+2} + \|\theta_u\|^2_{0,I_h} + \|f - f_h\|^2_{0,\mathcal{T}_h}\right) + Ch^{-\frac{3}{2}}\|\theta_u\|^3_{0,I_h} 
				\\
				&+ h^{-1}\|\theta_u\|^4_{0,I_h} + \left(\frac{2\lambda - 1}{2}\right)C\left(h^{k+1} + h^{k+\frac{1}{2}}\|\theta_u\|_{0,I_h}\right)\left(\sum_{i=0}^{N_x-1}\,h^{-1}[\![\theta_u]\!]^2_{i+1/2}\right)^\frac{1}{2}. 
			\end{aligned}
		\end{equation}
	\end{lem}

	\begin{proof}
			After choosing $\phi_h = \theta_u$ and $q_h = \theta_w$ in equations \eqref{erru1} and \eqref{errw1}, respectively and then add up to obtain
			\begin{equation*}
					\begin{aligned}
							\frac{1}{2}\frac{{\rm d}}{{\rm d}t}\|\theta_u\|&^2_{0,I_h} + \|\theta_w\|^2_{0,I_h} + \sqrt{\epsilon}\,b_h(\theta_w,\theta_u) + \sqrt{\epsilon}\,b_h(\theta_u,\theta_w) + \|\rho^{\frac{1}{2}}\theta_u\|^2_{0,I_h} = \left((\eta_u)_t,\theta_u\right) 
							\\
							& + \left(\rho\eta_u,\theta_u\right)  + \frac{1}{2}\sum_{i = 1}^{N_x}\int_{I_i}\left(u^2 - u_h^2\right)\partial_x\theta_u\,{\rm d}x + \frac{1}{2}\sum_{i = 0}^{N_x-1}\left(u^2 - \{u_h\}^2\right)[\![\theta_u]\!]_{i+1/2} 
							\\
							&+ \frac{1}{2}\sum_{i = 0}^{N_x-1}\left(\{u_h\}^2 - \widehat{u_h^2}\right)[\![\theta_u]\!]_{i+1/2} + \left(\eta_w,\theta_w\right) - \left(\left(\rho - \rho_h\right)\eta_u,\theta_u\right) 
							\\
							&- \left(\left(\rho - \rho_h\right)u,\theta_u\right) + \left(\rho V - \rho_h V_h,\theta_u\right) + \left(\left(\rho - \rho_h\right)\theta_u,\theta_u\right)
						\end{aligned}
				\end{equation*}
			Note that the third and the fourth terms on the left hand side of the above equality sum to zero, thanks to \eqref{bzero} and the definition of the projection in \eqref{gp}-\eqref{gp1}. Furthermore, the fifth term on the left hand side is non-negative. Hence can be dropped while estimating. An application of the H\"older  inequality with the Young's inequality, the estimate from Lemma \ref{rhoh2} and an application of the identity $\frac{a^2}{2} - \frac{b^2}{2} = a(a - b) - \frac{(a-b)^2}{2}$ results in
			\begin{equation*}
					\begin{aligned}
							\frac{1}{2}\frac{{\rm d}}{{\rm d}t}\|\theta_u\|&^2_{0,I_h} + \|\theta_w\|^2_{0,I_h} \lesssim \|(\eta_u)_t\|^2_{0,I_h} + \|\rho\|^2_{L^\infty(I)}\|\eta_u\|^2_{0,I_h} + \|\eta_w\|^2_{0,I_h} + \|f - f_h\|^2_{0,\mathcal{T}_h} 
							\\
							&+  h^{-1}\|f - f_h\|^2_{0,\mathcal{T}_h}\|\eta_u\|^2_{0,I_h} + \|f - f_h\|^2_{0,\mathcal{T}_h}\|u\|^2_{L^\infty(I)} + \|\theta_u\|^2_{0,I_h} + \|\theta_u\|^4_{L^4(I_h)}
							\\
							&+ \sum_{i = 1}^{N_x}\int_{I_i}u\left(u - u_h\right)\partial_x\theta_u\,{\rm d}x - \frac{1}{2}\sum_{i = 1}^{N_x}\int_{I_i}\left(u - u_h\right)^2\partial_x\theta_u\,{\rm d}x
							\\
							&+ \sum_{i = 0}^{N_x-1}u\left\{u - u_h\right\}[\![\theta_u]\!]_{i+1/2} - \frac{1}{2}\sum_{i = 0}^{N_x-1}\left(\left\{u - u_h\right\}\right)^2[\![\theta_u]\!]_{i+1/2} 
							\\
							&+ \frac{1}{2}\sum_{i = 0}^{N_x-1}\left(\{u_h\}^2 - \widehat{u_h^2}\right)[\![\theta_u]\!]_{i+1/2} + \frac{1}{2}\|\theta_w\|^2_{0,I_h}.
						\end{aligned}
				\end{equation*}
			A use of the norm comparison inequality \eqref{normcompeqn} with the approximation property \eqref{uprojection}, the identity \eqref{lambdadep} and employing the substitution $ u - \{u_h\} = \{u - u_h\} = \{\theta_u\} - \{\eta_u\}$, in the above equation shows
			\begin{equation}\label{uuuest}
					\begin{aligned}
							\frac{1}{2}\frac{{\rm d}}{{\rm d}t}\|\theta_u\|^2_{0,I_h} + \frac{1}{2}\|\theta_w\|^2_{0,I_h} \leq Ch^{2k+2} &+ C\|f - f_h\|^2_{0,\mathcal{T}_h} + \|\theta_u\|^2_{0,I_h} + h^{-1}\|\theta_u\|^4_{0,I_h}
							\\
							& + \kappa_1 + \kappa_2 + \kappa_3 + \kappa_4 + \kappa_5.
						\end{aligned}
				\end{equation}
			Here
			\begin{equation*}
				\begin{aligned}
					\kappa_1 &:= \frac{1}{2}\sum_{i = 1}^{N_x}\int_{I_i}u\,\partial_x\theta_u^2\,{\rm d}x + \sum_{i = 0}^{N_x - 1}u\,\{\theta_u\}\,[\![\theta_u]\!]_{i+1/2}
					\\
					\kappa_2 &:= -\sum_{i = 1}^{N_x}\int_{I_i}u\,\eta_u\,\partial_x\theta_u\,{\rm d}x - \left(\frac{2\lambda - 1}{2}\right)\sum_{i = 0}^{N_x - 1}\,u\,[\![\eta_u]\!]_{i+1/2}[\![\theta_u]\!]_{i+1/2}
			      	 \end{aligned}
				\end{equation*}
			\begin{equation*}
		\begin{aligned}
			\kappa_3 &:= -\frac{1}{6}\sum_{i = 1}^{N_x}\int_{I_i}\,\partial_x\theta_u^3\,{\rm d}x - \frac{1}{2}\sum_{i = 0}^{N_x-1}\,\{\theta_u\}^2\,[\![\theta_u]\!]_{i+1/2}
			\\
					\kappa_4 &:= \sum_{i=1}^{N_x}\int_{I_i}\,\left(\theta_u - \frac{1}{2}\eta_u\right)\eta_u\,\partial_x\theta_u\,{\rm d}x + \left(\frac{2\lambda - 1}{2}\right)\sum_{i = 0}^{N_x-1}\,\left(\{\theta_u\} - \frac{1}{2}\{\eta_u\}\right)[\![\eta_u]\!]_{i+1/2}[\![\theta_u]\!]_{i+1/2}
					\\
					\kappa_5 &:= \frac{1}{2}\sum_{i = 0}^{N_x-1}\left(\{u_h\}^2 - \widehat{u_h^2}\right)[\![\theta_u]\!]_{i+1/2}.
				\end{aligned}
			\end{equation*}
		Using integration by parts and the identity $[\![ab]\!] = \{a\}[\![b]\!] + [\![a]\!]\{b\}$, we obtain 
		\begin{equation*}
			\kappa_1 \leq C\|u_x\|_{L^\infty(I)}\|\theta_u\|^2_{0,I_h}.
		\end{equation*}
 Taylor's theorem says $u(x) = u(x_i) + \left(x - x_i\right)\partial_xu(x_i^*)$ for all $x \in I_i$ where $x_i^* \in (x,x_i)$. Therefore,
	\begin{equation*}
		\begin{aligned}
			\kappa_2 &= -\sum_{i=1}^{N_x}\,u(x_i)\int_{I_i}\eta_u\,\partial_x\theta_u\,{\rm d}x - \sum_{i=1}^{N_x}\partial_xu(x_i^*)\int_{I_i}\left(x - x_i\right)\,\eta_u\,\partial_x\theta_u\,{\rm d}x 
			\\
			&\quad- \left(\frac{2\lambda - 1}{2}\right)\sum_{i = 0}^{N_x-1}\,u\,[\![\eta_u]\!]_{i+1/2}[\![\theta_u]\!]_{i+1/2}
			\\
			&\leq h^{k+1}C\|\partial_xu\|_{L^\infty(I)}\|\theta_u\|_{0,I_h} + \|u\|_{L^\infty(I)}\left(\frac{2\lambda - 1}{2}\right)\sum_{i = 0}^{N_x-1}h^{\frac{1}{2}}[\![\eta_u]\!]_{i+1/2}h^{-\frac{1}{2}}[\![\theta_u]\!]_{i+1/2}
			\\
			&\leq C\left(h^{2k+2} + \|\theta_u\|_{0,I_h}^2\right) + \left(\frac{2\lambda - 1}{2}\right)C\|\eta_u\|_{0,I_h}\left(\sum_{i=0}^{N_x-1}\,h^{-1}[\![\theta_u]\!]_{i+1/2}^2\right)^{\frac{1}{2}}.
		\end{aligned}
	\end{equation*}
Note that the first term in the above expression of $\kappa_2$ vanishes, thanks to the definition of the global projection \eqref{gp}. To bound the second term, we have employed the approximation property \eqref{uprojection} and the inverse inequality \eqref{inverseeqn}. The last step is a consequence of the Young's inequality, the trace inequality \eqref{traceeqn} and the Cauchy-Schwarz inequality.\\
For $\kappa_3$, an integration by parts yields
\begin{equation*}
	\kappa_3 = \frac{1}{24}\sum_{i = 0}^{N_x-1}\,[\![\theta_u]\!]^3_{i+1/2} \lesssim \|\theta_u\|_{\infty, I_h}\|\theta_u\|^2_{0,\partial I_h} \leq Ch^{-\frac{3}{2}}\|\theta_u\|_{0,I_h}^3,
\end{equation*}
where in the last step we have used the norm comparison estimate \eqref{normcompeqn} and the trace inequality \eqref{traceeqn}.\\
To estimate $\kappa_4$, a use of the projection estimate with \eqref{inverseeqn}-\eqref{traceeqn}, the H\"older and the Young's inequalities yield
\begin{equation*}
	\begin{aligned}
		\kappa_4 &\leq \|\theta_u\|_{0,I_h}\|\eta_u\|_{0,I_h}\|\partial_x\theta_u\|_{\infty,I_h} + C\|\partial_x\theta_u\|_{\infty,I_h}\|\eta_u\|^2_{0,I_h} 
		\\
		&\quad + \left(\frac{2\lambda - 1}{2}\right)C\left(\|\theta_u\|_{0,\partial I_h} + \|\eta_u\|_{0,\partial I_h}\right)\|\eta_u\|_{0,\partial I_h}\left(\sum_{i=0}^{N_x-1}h^{-1}[\![\theta_u]\!]_{i+1/2}^2\right)^{\frac{1}{2}}
		\\
		&\leq C\left(h^{2k+2} + \|\theta_u\|^2_{0,I_h}\right) + \left(\frac{2\lambda - 1}{2}\right)Ch^{k+\frac{1}{2}}\|\theta_u\|_{0,I_h}\left(\sum_{i=0}^{N_x-1}h^{-1}[\![\theta_u]\!]_{i+1/2}^2\right)^{\frac{1}{2}}.
	\end{aligned}
\end{equation*}
Finally, using the fact that $\frac{\{u_h\}^2}{2} - \frac{\widehat{u_h^2}}{2} = -\frac{[\![u_h]\!]^2}{24}, [\![u_h]\!] = [\![\theta_u]\!] - [\![\eta_u]\!]$, the approximation property \eqref{uprojection}, trace inequality \eqref{traceeqn} and inverse inequality \eqref{inverseeqn} with the H\"older inequality, we arrive at
\begin{equation*}
	\begin{aligned}
		\kappa_5 &= -\frac{1}{24}\sum_{i=0}^{N_x-1}\left([\![\theta_u]\!]^2_{i+1/2} - 2[\![\eta_u]\!]_{i+1/2}[\![\theta_u]\!]_{i+1/2} + [\![\eta_u]\!]^2_{i+1/2}\right)[\![\theta_u]\!]_{i+1/2}
		\\
		&\leq C\|\theta_u\|_{\infty,I_h}\left(\|\theta_u\|_{0,\partial I_h}^2 + \|\theta_u\|_{0,\partial I_h}\|\eta_u\|_{0,\partial I_h} + \|\eta_u\|^2_{0,\partial I_h}\right)
		\\
		&\leq h^{-\frac{1}{2}}\|\theta_u\|_{0,I_h}\left(h^{-1}\|\theta_u\|^2_{0,I_h} + Ch^k\|\theta_u\|_{0,I_h} + Ch^{2k+1}\right)
		\\
		&\leq C\left(h^{2k+2} + \|\theta_u\|^2_{0,I_h} + h^{-\frac{3}{2}}\|\theta_u\|^3_{0,I_h}\right).
	\end{aligned}
\end{equation*}
Substituting all the above estimates in equation \eqref{uuuest}, we obtain our desired result.
\end{proof}

	\begin{rem}
		In arriving at the estimate \eqref{4.11} for $\theta_u$, we tackled the non-linear term by employing the algebraic identity $\frac{a^2}{2} - \frac{b^2}{2} = a(a-b) - \frac{(a-b)^2}{2}$, the Taylor's theorem and the definition \eqref{fu2} of the central flux $\widehat{u_h^2}$. 
	\end{rem}


	\begin{cor}\label{cor:uL2}
		If $\lambda= 1/2$ with $k$ even and $N_x$ odd, then 
		\begin{equation}\label{uL2}
			\begin{aligned}
				\frac{1}{2}\frac{{\rm d}}{{\rm d}t}\|\theta_u\|^2_{0,I_h} + \|\theta_w\|^2_{0,I_h} &\leq C\left(h^{2k+2} + \|f - f_h\|^2_{0,\mathcal{T}_h} + \|\theta_u\|^2_{0,I_h} \right.
				\\
				&\left. \quad\qquad\qquad\qquad +\,\, h^{-\frac{3}{2}}\|\theta_u\|^3_{0,I_h} + h^{-1}\|\theta_u\|^4_{0,I_h}\right).
			\end{aligned}
		\end{equation}
	If $\lambda > 1/2$, then 
	\begin{equation}\label{uL21}
		\begin{aligned}
			\frac{1}{2}\frac{{\rm d}}{{\rm d}t}\|\theta_u\|^2_{0,I_h} + \|\theta_w\|^2_{0,I_h} &\leq C\left(\left(1 + \epsilon^{-1} \right)h^{2k+2} + \|f - f_h\|^2_{0,\mathcal{T}_h} + \left(1 + \epsilon^{-1}h^{2k+1}\right)\|\theta_u\|^2_{0,I_h} \right.
			\\
			&\left. \quad +\,\, h^{-\frac{3}{2}}\|\theta_u\|^3_{0,I_h} + h^{-1}\|\theta_u\|^4_{0,I_h}\right).
		\end{aligned}
	\end{equation}
	\end{cor}
\begin{proof}
	For $\lambda = 1/2$ with $k$ even and $N_x$ odd, the last term in \eqref{4.11} is vanish. 
	
	Now, if $\lambda > 1/2$ then to estimate the last term of \eqref{4.11}, we use \eqref{conjecture}.
	\begin{equation*}
		\begin{aligned}
			&\left(\frac{2\lambda - 1}{2}\right)\left(h^{k+1} + h^{k+\frac{1}{2}}\|\theta_u\|_{0,I_h}\right)\left(\sum_{i=0}^{N_x-1}h^{-1}[\![\theta_u]\!]^2_{i+1/2}\right)^\frac{1}{2} 
			\\
			&\hspace{5cm}\leq C\epsilon^{-\frac{1}{2}}\left(h^{k+1} + h^{k+\frac{1}{2}}\|\theta_u\|_{0,I_h}\right)\|\theta_w\|_{0,I_h}
			\\
			&\hspace{5cm}\leq C(\delta)\epsilon^{-1}\left(h^{2k+2} + h^{2k+1}\|\theta_u\|^2_{0,I_h}\right) + \delta\|\theta_w\|^2_{0,I_h}.
		\end{aligned}
	\end{equation*}
For a sufficiently small $\delta > 0$, the desired result follows.
\end{proof}


	\subsection{Error estimates for Vlasov equation }
	
	Since the scheme given by \eqref{bh} with fluxes \eqref{flux} is consistent, \eqref{bh} holds for solution $(u,f)$ as well. Hence,
	by taking the difference, we obtain the error equation:
	\begin{equation*}\label{erro}
		\begin{aligned}
			\left(\frac{\partial }{\partial t}\left(f - f_h\right),\psi_h\right) + \mathcal{B}(u;f,\psi_h) - \mathcal{B}_{h}(u_h;f_h,\psi_h) = 0
			\quad\forall\, \psi_h \in \mathcal{Z}_h.  
		\end{aligned}
	\end{equation*}
	Setting $e_f := f - f_h$, we rewrite the error equation as  
	\begin{equation}\label{error}
		\begin{aligned}
			\left(\frac{\partial e_f}{\partial t},\psi_h\right) + a_h^0(e_f,\psi_h) + \mathcal{N}(u;f,\psi_h) - \mathcal{N}^h(u_h;f_h,\psi_h) = 0
			\quad\forall\, \psi_h \in \mathcal{Z}_h , 
		\end{aligned}
	\end{equation}
	where
	\begin{equation}\label{ae}
		a_h^0(e_f,\psi_h) := -\sum_{i,j}\int_{T_{ij}}e_f v\,\partial_x\psi_h\,{\rm d}v\,{\rm d}x
		-\sum_{j = 1}^{N_v}\sum_{i = 0}^{N_x-1}\int_{J_j}\left(\widehat{v e_f}[\![\psi_h]\!]\right)_{i+1/2,v}\,{\rm d}v,
	\end{equation}	
	\begin{equation}\label{N}
		\mathcal{N}(u;f,\psi_h) := -\sum_{i,j}\int_{T_{ij}} f (u - v)\,\partial_v\psi_h\,{\rm d}v\,{\rm d}x \,\, - \sum_{i = 1}^{N_x}\sum_{j=0}^{N_v-1}\int_{I_i}\left((u - v) f[\![\psi_h]\!]\right)_{x,j+1/2}\,{\rm d}x \, ,
	\end{equation}
	and
	\begin{equation}\label{Nh}
		\begin{aligned}
			\mathcal{N}^h(u_h;f_h,\psi_h) := -\sum_{i,j}\int_{T_{ij}} &f_h (u_h - v)\,\partial_v\psi_h\,{\rm d}v\,{\rm d}x 
			\\
			&- \sum_{i = 1}^{N_x}\sum_{j = 0}^{N_v-1}\int_{I_i}\left(\reallywidehat{(u_h - v) f_h}[\![\psi_h]\!]\right)_{x,j+1/2}\,{\rm d}x \, .
		\end{aligned}
	\end{equation}

\subsection{2D projection}
Inspired by \cite{liu2020optimal}, we now introduce a $2$D global projection $\Pi_{\lambda_1,\lambda_2}: \mathcal{C}^0(\mathcal{T}_h) \rightarrow \mathcal{Z}_h$  for  $\lambda_1,\lambda_2 > \frac{1}{2}$ 
as
\begin{equation}\label{4.21}
	\Pi_{\lambda_1,\lambda_2}g = \left(\Pi_{\lambda_1}\otimes\Pi_{\lambda_2}\right)g, \quad \forall \quad g \in \mathcal{C}^0(\overline{T_{ij}}), 1 \leq i \leq N_x, 1 \leq j \leq N_v.
\end{equation}
where
\begin{equation*}
	\Pi_{\lambda_1}g :=
	\left\{ \,
	\begin{aligned}
		 & Q_{\lambda_1}^xg \qquad \mbox{if} \quad v > 0,
		 \\
		 & Q_{1-\lambda_1}^xg \quad \mbox{if} \quad v < 0,
		 \\
		 & Q_1^xg \qquad \mbox{if}  \quad v = 0,
	\end{aligned}
	\right.
\end{equation*}
and 
\begin{equation*}
	\Pi_{\lambda_2}g :=
	\left\{ \,
	\begin{aligned}
		& Q_{\lambda_2}^vg \qquad \mbox{if} \quad (u(x) - v) > 0,
		\\
		& Q_{1-\lambda_2}^vg \quad \mbox{if} \quad (u(x) - v) < 0,
		\\
		& Q_1^vg \qquad \mbox{if}  \quad (u(x) - v) = 0.
	\end{aligned}
	\right.
\end{equation*}

More precisely, the above defined projection $\Pi_{\lambda_1,\lambda_2}$ can be explicitly defined as follows:
Let $\gamma_1, \Bbbk_1^+$ and $\Bbbk_1^-$ be the set of all indices $j$ for which $v$ changes sign, remains positive and remains negative, respectively inside $J_j$. Let $\gamma_2, \Bbbk_2^+$ and $\Bbbk_2^-$ be the set of all indices $i$ for which $\left(u(x) - v\right)$ changes sign, remains positive and remains negative, respectively inside $I_i$ for a fixed $v$.

The projection $\Pi_{\lambda_1,\lambda_2}g$ satisfies the following equality involving volume integrals: 
\begin{equation}
	\begin{aligned}
		\int_{T_{ij}}\Pi_{\lambda_1,\lambda_2}gz_h\,{\rm d}x\,{\rm d}v = \int_{T_{ij}}gz_h\,{\rm d}x\,{\rm d}v, \quad \forall\quad z_h \in \mathbb{Q}^{k-1}(T_{ij})
	\end{aligned}
\end{equation}
The projection $\Pi_{\lambda_1,\lambda_2}g$ satisfies the following equalities involving vertical boundary integrals: 
\begin{equation}
	\left\{
	\begin{aligned}
	\int_{J_j}\left(\Pi_{\lambda_1,\lambda_2}g\right)^-_{i+\frac{1}{2},v}(z_h)^-_{i+\frac{1}{2},v}\,{\rm d}v &= \int_{J_j}g^-_{i+\frac{1}{2},v}(z_h)^-_{i+\frac{1}{2},v}\,{\rm d}v, \quad \mbox{if} \quad j \in \gamma_1
    \\
	\int_{J_j}\left(\Pi_{\lambda_1,\lambda_2}g\right)^{\lambda_1}_{i+\frac{1}{2},v}(z_h)^-_{i+\frac{1}{2},v}\,{\rm d}v &= \int_{J_j}g^{\lambda_1}_{i+\frac{1}{2},v}(z_h)^-_{i+\frac{1}{2},v}\,{\rm d}v, \quad\mbox{if} \quad j \in \Bbbk_1^+
    \\	\int_{J_j}\left(\Pi_{\lambda_1,\lambda_2}g\right)^{1-\lambda_1}_{i-\frac{1}{2},v}(z_h)^+_{i-\frac{1}{2},v}\,{\rm d}v &= \int_{J_j}g^{1-\lambda_1}_{i-\frac{1}{2},v}(z_h)^+_{i-\frac{1}{2},v}\,{\rm d}v, \quad \mbox{if} \quad j \in \Bbbk_1^-
   \end{aligned}
   \right.
\end{equation}
for all $z_h \in \mathbb{Q}^{k-1}(T_{ij})$.
 The projection $\Pi_{\lambda_1,\lambda_2}g$ satisfies the following equalities involving horizontal boundary integrals:
\begin{equation}
	\left\{
	\begin{aligned}
	\int_{I_i}\left(\Pi_{\lambda_1,\lambda_2}g\right)^-_{x,j+\frac{1}{2}}(z_h)^-_{x,j+\frac{1}{2}}\,{\rm d}x &= \int_{I_i}g^-_{x,j+\frac{1}{2}}(z_h)^-_{x,j+\frac{1}{2}}\,{\rm d}x, \,\,\mbox{if}\,\, i \in \gamma_2
    \\
	\int_{I_i}\left(\Pi_{\lambda_1,\lambda_2}g\right)^{\lambda_2}_{x,j+\frac{1}{2}}(z_h)^-_{x,j+\frac{1}{2}}\,{\rm d}x &= \int_{I_i}g^{\lambda_2}_{x,j+\frac{1}{2}}(z_h)^-_{x,j+\frac{1}{2}}\,{\rm d}x, \,\,\mbox{if}\,\, i \in \Bbbk_2^+
    \\
	\int_{I_i}\left(\Pi_{\lambda_1,\lambda_2}g\right)^{1-\lambda_2}_{x,j-\frac{1}{2}}(z_h)^-_{x,j-\frac{1}{2}}\,{\rm d}x &= \int_{I_i}g^{1-\lambda_2}_{x,j-\frac{1}{2}}(z_h)^-_{x,j-\frac{1}{2}}\,{\rm d}x, \,\,\mbox{if}\,\, i \in \Bbbk_2^-
    \end{aligned}
    \right.
\end{equation}
for all $z_h \in \mathbb{Q}^{k-1}(T_{ij})$.
The projection $\Pi_{\lambda_1,\lambda_2}g$ satisfies the following equalities involving the boundary points:
\begin{equation}
	\left(\Pi_{\lambda_1,\lambda_2}g\right)^{-,-}_{i+\frac{1}{2},j+\frac{1}{2}} = g^{-,-}_{i+\frac{1}{2},j+\frac{1}{2}}, \quad \mbox{if} \quad \left(i,j\right) \in \left(\,\gamma_2,\gamma_1\,\right)
\end{equation}
\begin{equation}
	\left\{
	\begin{aligned}
	\left(\Pi_{\lambda_1,\lambda_2}g\right)^{\lambda_1,-}_{i+\frac{1}{2},j+\frac{1}{2}} &= g^{\lambda_1,-}_{i+\frac{1}{2},j+\frac{1}{2}}, \quad \mbox{if} \quad \left(i,j\right) \in (\gamma_2,\Bbbk_1^+)
    \\
	\left(\Pi_{\lambda_1,\lambda_2}g\right)^{1-\lambda_1,-}_{i-\frac{1}{2},j+\frac{1}{2}} &= g^{1-\lambda_1,-}_{i-\frac{1}{2},j+\frac{1}{2}}, \quad \mbox{if} \quad \left(i,j\right) \in (\gamma_2,\Bbbk_1^-)
    \end{aligned}
    \right.
\end{equation}
\begin{equation}
	\left\{
	\begin{aligned}
	\left(\Pi_{\lambda_1,\lambda_2}g\right)^{-,\lambda_2}_{i+\frac{1}{2},j+\frac{1}{2}} &= g^{-,\lambda_2}_{i+\frac{1}{2},j+\frac{1}{2}}, \quad \mbox{if} \quad \left(i,j\right) \in (\Bbbk_2^+,\gamma_1)
    \\
	\left(\Pi_{\lambda_1,\lambda_2}g\right)^{-,1-\lambda_2}_{i+\frac{1}{2},j-\frac{1}{2}} &= g^{-,1-\lambda_2}_{i+\frac{1}{2},j-\frac{1}{2}}, \quad \mbox{if} \quad \left(i,j\right) \in (\Bbbk_2^-,\gamma_1)
    \end{aligned}
    \right.
\end{equation}
\begin{equation}
	\left(\Pi_{\lambda_1,\lambda_2}g\right)^{\lambda_1,\lambda_2}_{i+\frac{1}{2},j+\frac{1}{2}} = g^{\lambda_1,\lambda_2}_{i+\frac{1}{2},j+\frac{1}{2}}, \quad \mbox{if} \quad \left(i,j\right) \in (\Bbbk_2^+,\Bbbk_1^+)
\end{equation}
\begin{equation}
	\left(\Pi_{\lambda_1,\lambda_2}g\right)^{\lambda_1,1-\lambda_2}_{i+\frac{1}{2},j-\frac{1}{2}} = g^{\lambda_1,1-\lambda_2}_{i+\frac{1}{2},j-\frac{1}{2}}, \quad \mbox{if} \quad \left(i,j\right) \in (\Bbbk_2^-,\Bbbk_1^+)
\end{equation}
\begin{equation}
	\left(\Pi_{\lambda_1,\lambda_2}g\right)^{1-\lambda_1,\lambda_2}_{i-\frac{1}{2},j+\frac{1}{2}} = g^{1-\lambda_1,\lambda_2}_{i-\frac{1}{2},j+\frac{1}{2}}, \quad \mbox{if} \quad \left(i,j\right) \in (\Bbbk_2^+,\Bbbk_1^-)
\end{equation}
\begin{equation}
	\left(\Pi_{\lambda_1,\lambda_2}g\right)^{1-\lambda_1,1-\lambda_2}_{i-\frac{1}{2},j-\frac{1}{2}} = g^{1-\lambda_1,1-\lambda_2}_{i-\frac{1}{2},j-\frac{1}{2}}, \quad \mbox{if} \quad \left(i,j\right) \in (\Bbbk_2^-,\Bbbk_1^-).
\end{equation}
Here, we have used the following notations:
\begin{equation*}
	\begin{aligned}
		g^{\varkappa,\varsigma}_{i+\frac{1}{2},j+\frac{1}{2}} &:= \varkappa\varsigma g^{-,-}_{i+\frac{1}{2},j+\frac{1}{2}} + \varkappa\left(1-\varsigma\right)g^{-,+}_{i+\frac{1}{2},j+\frac{1}{2}} 
		\\
		&\quad+ \left(1-\varkappa\right)\varsigma g^{+,-}_{i+\frac{1}{2},j+\frac{1}{2}} + \left(1-\varkappa\right)\left(1-\varsigma\right)g^{+,+}_{i+\frac{1}{2},j+\frac{1}{2}},
	\end{aligned}
\end{equation*}
\begin{equation*}
	\begin{aligned}
		g^{\varkappa}_{i-1/2,v} &:= \left(1-\varkappa\right)\left(g\right)^+_{i-1/2,v} + \varkappa\left(g\right)^-_{i-1/2,v},
		\\
		g^{\varsigma}_{x,j-1/2} &:= \left(1-\varsigma\right)\left(g\right)^+_{x,j-1/2} + \varsigma\left(g\right)^-_{x,j-1/2},
		\\
		g^{\varkappa,-}_{i-1/2,j-1/2} &:= \left(1-\varkappa\right)g\left(x^+_{i-1/2},v^-_{j-1/2}\right) + \varkappa g\left(x^-_{i-1/2},v^-_{j-1/2}\right),
		\\
		g^{-,\varkappa}_{i-1/2,j-1/2} &:= \left(1-\varkappa\right)g\left(x^-_{i-1/2},v^+_{j-1/2}\right) + \varkappa g\left(x^-_{i-1/2},v^-_{j-1/2}\right),
		\\
		g^{\pm,\pm}_{i-1/2,j-1/2} &:= g\left(x^\pm_{i-1/2},v^\pm_{j-1/2}\right),
	\end{aligned}
\end{equation*}
with $x^\pm_{i-1/2}$ denoting $\lim_{\varrho \to 0^+}\left(x_{i-1/2}\pm \varrho\right)$ and $\{v^\pm_{j-1/2}\}$ denoting $\lim_{\varrho \to 0^+}\left(v_{j-1/2}\pm \varrho\right)$. In the above notations the parenthesis $\varkappa$ takes the value $\lambda_1$ and $1-\lambda_1$ whereas $\varsigma$ takes the value $\lambda_2$ and $1-\lambda_2$.

In the following lemma, we state the approximation property of the global projection $\Pi_{\lambda_1,\lambda_2}$ whose proof can be found in \cite[Lemma 3.1, p. 8]{liu2020optimal}.
\begin{lem}\label{westimate}
	For $f \in H^{k+1}(\Omega)$, there exists a unique $\Pi_{\lambda_1,\lambda_2}f$ defined by \eqref{4.21} and satisfying the following approximation property:
	\begin{equation}\label{wprojection}
		\|f - \Pi_{\lambda_1,\lambda_2}f\|_{0,\mathcal{T}_h} + h^\frac{1}{2}\|f - \Pi_{\lambda_1,\lambda_2}f\|_{0,\Gamma_h} \leq C\,h^{k+1}\|f\|_{k+1,\Omega},
	\end{equation}
where $\|\cdot\|_{0,\Gamma_h}^2 = \|\cdot\|^2_{0,\Gamma_x} + \|\cdot\|^2_{0,\Gamma_v}$.
\end{lem} 


	Using the projection, split $f - f_h$ as
	\begin{equation}\label{p}
		e_f := f - f_h := \left(\Pi_{\lambda_1,\lambda_2}f - f_h\right) - \left(\Pi_{\lambda_1,\lambda_2}f - f\right) =: \theta_f - \eta_f,
	\end{equation}
	where
	\[
	\theta_f = \Pi_{\lambda_1,\lambda_2}f - f_h \quad \mbox{and} \quad \eta_f = \Pi_{\lambda_1,\lambda_2}f - f.
	\]

	\begin{lem}\label{lemn}
		Let $u \in C^0(I), f \in C^0(\Omega)$ and $f_h \in \mathcal{Z}_h$ with $k \geq 0$. Then, the following identity holds  true:
		\begin{align*}
			\mathcal{N}(u;f,\theta_f) &- \mathcal{N}^h(u_h;f_h,\theta_f) = \sum_{i,j}\int_{I_i}\left(\frac{2\lambda_2 - 1}{2}\right)\left\vert(u_h - v)\right\vert[\![\theta_f]\!]^2_{x,j-1/2}\,{\rm d}x
			\\
			& + \sum_{i,j}\int_{T_{ij}} \left((u-u_h)\,\partial_v f\,\theta_f - \frac{1}{2}\theta^2_f\right)\,{\rm d}v\,{\rm d}x + \mathcal{K}^2(u_h - v,f,\theta_f),
		\end{align*}
		where
		\begin{equation}\label{K2}
			\mathcal{K}^2(u_h - v,f,\theta_f) = \sum_{i,j}\int_{T_{ij}} \eta_f (u_h - v)\,\partial_v \theta_f\,{\rm d}v\,{\rm d}x + \sum_{i,j}\int_{I_i}\left(\reallywidehat{\left(u_h - v\right)\eta_f}[\![\theta_f]\!]\right)_{x,j-1/2}\,{\rm d}x\, .
		\end{equation}
	\end{lem}
	
	\begin{proof}
		Proof is similar to the \cite[Lemma 4.8]{ayuso2009discontinuous} with some changes in boundary terms. So, we are providing a short proof.
	Subtracting the non-linear terms \eqref{N} and \eqref{Nh}, we arrive at 
	\begin{equation}\label{sasa}
		\begin{aligned}
			&\mathcal{N}\left(u;f,\theta_f\right) - \mathcal{N}^h\left(u_h;f_h,\theta_f\right) = -\sum_{i,j}\int_{T_{ij}}\left[f\left(u - v\right) - f_h\left(u_h - v\right)\right]\partial_v\theta_f \,{\rm d}x\,{\rm d}v 
			\\
			&\qquad- \sum_{i,j}\int_{I_i}\left(\left(\left(u - v\right)f - \reallywidehat{\left(u_h - v\right) f_h}\right)[\![\theta_f]\!]\right)_{x,j-1/2}\,{\rm d}x 
			= T_1 + T_2 + T_3.
		\end{aligned}
	\end{equation}
	After adding and subtracting $\left(u_h - v\right)f$ in the volume term, we obtain 
	\begin{equation*}
		\begin{aligned}
			T_1 &= -\sum_{i,j}\int_{T_{ij}}\left(f\left(u - u_h\right) \right)\partial_v\theta_f \,{\rm d}x\,{\rm d}v,
			\\
			T_2 &= -\sum_{i,j}\int_{T_{ij}}\left(\left(u_h - v\right)\left( f - f_h\right)\right)\partial_v\theta_f \,{\rm d}x\,{\rm d}v,
			\\
			T_3 &= - \sum_{i,j}\int_{I_i}\left(\left(\left(u - v\right)f - \reallywidehat{\left(u_h - v\right) f_h}\right)[\![\theta_f]\!]\right)_{x,j-1/2}\,{\rm d}x.
		\end{aligned}
	\end{equation*}
	An integration by parts with respect to the $v$ variable yields
	\begin{equation*}
		\begin{aligned}
			T_1 &= \sum_{i,j}\int_{T_{ij}}\left(u - u_h\right)\,\partial_v f \theta_f\,{\rm d}x\,{\rm d}v + \sum_{i,j}\int_{I_i}\left(u - u_h\right)  f [\![\theta_f]\!]_{x,j-1/2}\,{\rm d}x
			\\
			& =: T_{11} + T_{12}.
		\end{aligned}
	\end{equation*}
	We rewrite the term $T_2$ as
	\begin{equation*}
		\begin{aligned}
			T_2 &= - \sum_{i,j}\int_{T_{ij}} \left(f - f_h\right)\left(u_h -  v\right)\,\partial_v\theta_f\,{\rm d}x\,{\rm d}v 
			\\
			& =  \sum_{i,j}\int_{T_{ij}} \eta_f\left(u_h -  v\right)\,\partial_v\theta_f\,{\rm d}x\,{\rm d}v - \frac{1}{2}\sum_{i,j}\int_{T_{ij}} \left(u_h -  v\right)\,\partial_v\theta_f^2\,{\rm d}x\,{\rm d}v  
			\\
			& =: T_{21} + T_{22}.
		\end{aligned}
	\end{equation*}
	An integration by parts with respect to the $v$ variable in the term $T_{22}$ yields
	\begin{equation*}
		\begin{aligned}
			T_{22}  = - \frac{1}{2}\sum_{i,j}\int_{T_{ij}}\theta_f^2\,{\rm d}x\,{\rm d}v + \sum_{i,j}\int_{I_i}\left(\frac{\left(u_h - v\right)}{2}[\![\theta_f^2]\!]\right)_{x,j-1/2}\,{\rm d}x
			=: T_{221} + T_{222}.
		\end{aligned}
	\end{equation*}
	We finally deal with the boundary term $T_3$ as follows:
	\begin{equation*}
		\begin{aligned}
			T_3 &= - \sum_{i,j}\int_{I_i}\left(\left(\left(u - v\right)f - \reallywidehat{\left(u_h - v\right) f_h}\right)[\![\theta_f]\!]\right)_{x,j-1/2}\,{\rm d}x
			\\
			&= - \sum_{i,j}\int_{I_i}\left(\left(\left(u - v\right)f  - \left(u_h - v\right) f \right)[\![\theta_f]\!]\right)_{x,j-1/2}\,{\rm d}x
			\\
			& \quad- \sum_{i,j}\int_{I_i}\left(\left( \reallywidehat{\left(u_h - v\right) f} - \reallywidehat{\left(u_h - v\right) f_h}\right)[\![\theta_f]\!]\right)_{x,j-1/2}\,{\rm d}x
			\\
			&= - \sum_{i,j}\int_{I_i}\left(\left(u - u_h\right)f [\![\theta_f]\!]\right)_{x,j-1/2}\,{\rm d}x 
			- \sum_{i,j}\int_{I_i}\left(\reallywidehat{\left(u_h - v\right) \left(f - f_h\right)} [\![\theta_f]\!]\right)_{x,j-1/2}\,{\rm d}x
			\\
			& = T_{31} + T_{32}.
		\end{aligned}
	\end{equation*}
	From $T_{12},T_{222},T_{31}$ and $T_{32}$ terms, we obtain 
	\begin{equation*}
		\begin{aligned}
			T_{31} + T_{12} = 0.
		\end{aligned}
	\end{equation*}
	\begin{equation*}
		\begin{aligned}
			T_{222} + T_{32} &= \sum_{i,j}\int_{I_i}\left(\left(u_h - v\right)\{\theta_f\}[\![\theta_f]\!] - \left(u_h - v\right)\left\{f - f_h\right\} [\![\theta_f]\!]\right)_{x,j-1/2}\,{\rm d}x
			\\
			&\quad + \sum_{i,j}\int_{I_i}\left(\frac{2\lambda_2 - 1}{2}\right)\left(\left\vert\left(u_h - v\right)\right\vert[\![f - f_h]\!]\,[\![\theta_f]\!]\right)_{x,j-1/2}\,{\rm d}x
			\\
			&= \sum_{i,j}\int_{I_i}\left(\frac{2\lambda_2 - 1}{2}\right)\left(\left\vert\left(u_h - v\right)\right\vert[\![\theta_f]\!]^2 + \reallywidehat{\left(u_h - v\right) \eta_f}[\![\theta_f]\!]\right)_{x,j-1/2}\,{\rm d}x.
		\end{aligned}
	\end{equation*}
	Here, in first step we use $[\![ab]\!] = \{a\}[\![b]\!] + [\![a]\!]\{b\}$. Now, after putting all identities in equation \eqref{sasa}, we conclude the proof.	
\end{proof}

	
	After choosing $\psi_h = \theta_f$ in the error equation \eqref{error} and a use of equation \eqref{p} and Lemma \ref{lemn}, we rewrite equation \eqref{error} in $\theta_f$ as
	\begin{equation}\label{err1}
		\begin{aligned}
			&\frac{1}{2}\frac{{\rm d}}{{\rm d}t}\|\theta_f\|^2_{0,\mathcal{T}_h} + \sum_{i,j}\int_{J_j}\left(\frac{2\lambda_1 - 1}{2}\right)\left\vert v\right\vert[\![\theta_f]\!]^2_{i-1/2,v}\,{\rm d}v  + \sum_{i,j}\int_{I_i}\left(\frac{2\lambda_2 - 1}{2}\right)\left\vert u_h - v\right\vert[\![\theta_f]\!]^2_{x,j-1/2}\,{\rm d}x
			\\
			& = \left((\eta_f)_t, \theta_f\right) - \mathcal{K}^1(v,\eta_f,\theta_f) - \sum_{i,j}\int_{T_{ij}} \left((u-u_h)\,\partial_v f\,\theta_f - \frac{1}{2}\theta^2_f\right)\,{\rm d}v\,{\rm d}x 
			\\
			&\quad- \mathcal{K}^2(u_h - v,f,\theta_f) ,
		\end{aligned}
	\end{equation}	
	where,
	\begin{equation}\label{K1}
		\mathcal{K}^1(v,\eta_f,\theta_f) = \sum_{i,j}\int_{T_{ij}} \eta_f v\,\partial_x \theta_f\,{\rm d}x\,{\rm d}v + \sum_{i,j}\int_{J_j}\left(\widehat{v\,\eta_f}[\![\theta_f]\!]\right)_{i-1/2,v}\,{\rm d}v\, .
	\end{equation}


	To show the estimate on $e_f := \theta_f - \eta_f$, it is enough to show the estimate on $\theta_f$ since from \eqref{wprojection} estimate on $\eta_f$ are known.

	Let $J^{i+1/2} = \{x_{i+1/2}\}\times J$ and $I^{j+1/2} = I \times \{v_{j+1/2}\}$.
	
	\begin{lem}\label{K1estimate}
		Let $k \geq 1$ and let $f \in C^0([0,T]; W^{1,\infty}(\Omega) \cap H^{k+1}(\Omega))$ be the solution of \eqref{eq:continuous-model}. Let $f_h(t) \in \mathcal{Z}_h$ be its approximation satisfying \eqref{bh} and let $\mathcal{K}^1(v,\eta_f,\theta_f)$ be defined as in \eqref{K1}. Assume that the partition $\mathcal{T}_h$ is constructed so that none of the components of $v$ vanish inside any element. Then, the following estimate holds:
		\begin{equation}\label{k1estimate}
			|\mathcal{K}^1(v,\eta_f,\theta_f)| \leq Ch^{k+1}\|f\|_{k+1,\Omega} \|\theta_f\|_{0,\mathcal{T}_h} \quad \forall \quad t \in [0,T].
		\end{equation} 
		
	\end{lem}

\begin{proof}
	Define 
	\begin{equation*}
		\mathcal{K}^1(v,\eta_f,\theta_f) = \sum_{i,j}\mathcal{K}^1_{ij}(v,\eta_f,\theta_f),
	\end{equation*}
where
\begin{equation*}
	K^1_{ij}(v,\eta_f,\theta_f) = \int_{T_{ij}}\eta_fv\,\partial_x\theta_f\,{\rm d}x\,{\rm d}v + \int_{J_j}\left(\reallywidehat{v\eta_f}[\![\theta_f]\!]\right)_{i-1/2,v}\,{\rm d}v.
\end{equation*}
We will estimate $\mathcal{K}^1(v,\eta_f,\theta_f)$ for a single arbitrary element $T_{ij}$ and then we sum over all elements. Let $\bar{v} = \mathcal{P}^0_v(v)$ be the $L^2$-projection onto the piecewise constants of $T_{ij}$ of $v$ then we write
\begin{equation}\label{4.35}
	\mathcal{K}^1_{ij}(v,\eta_f,\theta_f) = \mathcal{K}^1_{ij}(v - \bar{v},\eta_f,\theta_f) + \mathcal{K}^1_{ij}(\bar v,\eta_f,\theta_f).
\end{equation}
The H\"older inequality with the projection estimates and \eqref{inverseeqn}-\eqref{traceeqn} yields
\begin{equation*}
	\begin{aligned}
		\mathcal{K}^1_{ij}(v - \bar{v},\eta_f,\theta_f) &= \|v - \bar{v}\|_{L^{\infty}(J)}\left(\|\eta_f\|_{0,T_{ij}}\|\partial_x\theta_f\|_{0,T_{ij}} + \sum_{m = i+1/2}\|\eta_f\|_{0,J^m}\|\theta_f\|_{0,J^m}\right)
		\\
		& \leq h_vh^k\|f\|_{k+1,T_{ij}}\|\theta_f\|_{0,T_{ij}}.
	\end{aligned}
\end{equation*}
Using \eqref{4.21} and \eqref{flux}, the last term of \eqref{4.35} is zero. 
This completes the proof.
\end{proof}
	

	\begin{lem}\label{K2estimate}
		Let $\mathcal{T}_h$ be a Cartesian mesh of $\Omega$, $k \geq 1$ and let $(u_h,f_h) \in X_h\times \mathcal{Z}_h$ be the solution to \eqref{bh}. Let $(u,f) \in L^\infty([0,T];\,W^{1,\infty}(I)\cap H^{k+1}(I))\times L^\infty([0,T];\,W^{1,\infty}(\Omega)\cap H^{k+1}(\Omega))$ and let $\mathcal{K}^2$ be defined as in \eqref{K2}. Then, the following estimate holds
		\begin{equation}\label{k2estimate}
			\begin{aligned}
				|\mathcal{K}^2(u_h - v,f,\theta_f)| \leq& C\left(h^k\|u_h - u\|_{\infty,I_h} 
				+ h^{k+1}\|u\|_{W^{1,\infty}(I)}\right)\|f\|_{k+1,\Omega} \|\theta_f\|_{0,\mathcal{T}_h}.
			\end{aligned}
		\end{equation}		
	\end{lem}

\begin{proof}
	Let
	\begin{equation*}
		\mathcal{K}^2(u_h - v,f,\theta_f) = \sum_{i,j}\mathcal{K}^2_{ij}(u_h - v,f,\theta_f),
	\end{equation*}
where
\begin{equation*}
	\mathcal{K}^2_{ij}(u_h - v,f,\theta_f) = \int_{T_{ij}}\left(u_h - v \right)\eta_f\,\partial_v\theta_f\,{\rm d}x\,{\rm d}v + \int_{I_i}\left(\reallywidehat{\left(u_h - v\right)\eta_f}[\![\theta_f]\!]\right)_{x,j+1/2}\,{\rm d}x.
\end{equation*}
First, we prove the estimate for a single element $T_{ij}$ then we take sum over all element. By adding and subtracting $\left(u - v\right)$ in $\mathcal{K}^2_{ij}(u_h - v,f,\theta_f)$, we obtain
\begin{equation}\label{c5}
	\mathcal{K}^2_{ij}(u_h - v,f,\theta_f) \leq \mathcal{K}^2_{ij}((u_h - u),f,\theta_f) + \mathcal{K}^2_{ij}((u - v),f,\theta_f).
\end{equation}
Using the H\"older inequality, \eqref{inverseeqn}-\eqref{traceeqn} with the projection estimates from the first term, we obtain
\begin{equation*}\label{c6}
	\begin{aligned}
		\mathcal{K}^2_{ij}(u_h - u),f,\theta_f) &\leq \|u_h - u\|_{\infty,I_i}\left(\|\eta_f\|_{0,T_{ij}}\|\partial_v\theta_f\|_{0,T_{ij}} \right.
		\\
		&\left. + \sum_{m = i+1/2}\|\eta_f\|_{0,I^m}\|\theta_f\|_{0,I^m}\right)
		\\
		&\leq Ch^{k}\|u_h - u\|_{\infty,I_i}\|f\|_{k+1,T_{ij}}\|\theta_f\|_{0,T_{ij}}.
	\end{aligned}
\end{equation*}
Now, to complete the proof we need to estimate the last term of \eqref{c5} for this we make different cases:
\begin{enumerate}
	\item If $\left(u(x) - v\right) \neq 0, \,\, \forall \,\, x \in I_i$.
	\item If $\left(u - v\right)$ vanish inside $T_{ij}$.
\end{enumerate}
In the case $(1)$, without loss of generality assume that $\left(u(x) - v\right) > 0$ for all $x \in I_i$. A use of equation \eqref{4.21} with \eqref{flux} shows:
\begin{equation*}
	\mathcal{K}^2_{ij}\left((u - v),f,\theta_f\right) = 0.
\end{equation*}
In the case $(2)$, the projection operator is $\Pi_{\lambda_1,\lambda_2}$. Then, the H\"older inequality with \eqref{inverseeqn}-\eqref{traceeqn} and projection estimate, yields
\begin{equation*}
	\begin{aligned}
		\mathcal{K}^2_{ij}((u - v),f,\theta_f) &\leq \|u - v\|_{L^\infty(I)}\left(\|\eta_f\|_{0,T_{ij}}\|\partial_v\theta_f\|_{0,T_{ij}} + \sum_{m = j+1/2}\|\eta_f\|_{0,I^m}\|\theta_f\|_{0,I^m}\right)
		\\
		&\leq Ch^k\|u - v\|_{L^\infty(I)}\|f\|_{k+1,T_{ij}}\|\theta_f\|_{0,T_{ij}}.
	\end{aligned}
\end{equation*}
Now, using the fact that, there exist $x^* \in I_i$ such that $\left(u(x^*) - v\right) = 0$ together with the mean value theorem, we obtain
\begin{equation*}
	\|u - v\|_{L^{\infty}(I)} = \max_{x \in I_i}\vert u(x) - u(x^*)\vert \leq C\max_{x \in I_i}\vert x - x^*\vert \,\|u_x\|_{L^\infty(I)} \leq Ch\|u\|_{W^{1,\infty}(I)}.
\end{equation*}
Hence,
\begin{equation*}
	\mathcal{K}^2_{ij}((u - v),f,\theta_f) \leq Ch^{k+1}\|u\|_{W^{1,\infty}(I)}\|f\|_{k+1,T_{ij}}\|\theta_f\|_{0,T_{ij}}.
\end{equation*}
This completes the proof.
\end{proof}


	\begin{lem}
		Let $f(t)$ be the solution of Vlasov equation \eqref{eq:continuous-model}. Let $f_h(t) \in \mathcal{Z}_h$ be the finite approximation. Assume that $f \in L^\infty([0,T]; H^{k+1}(\Omega)\cap W^{1,\infty}(\Omega))$. Then there exist a constant $C > 0$ such that for $t \in (0,T]$
		\begin{equation*}
			\|f(t) - f_h(t)\|_{\infty,\mathcal{T}_h} \leq C\left(h\|f(t)\|_{W^{1,\infty}(\Omega)} + h^{k+\frac{1}{2}}\|f(t)\|_{k+1,\Omega} + h^{-\frac{1}{2}}\|f(t) - f_h(t)\|_{0,\mathcal{T}_h}\right).
		\end{equation*}
	\end{lem}
	
	\begin{proof}
		The proof similar to Lemma \ref{L:uinf} just taking $u = f, u_h = f_h$ and $Q_\lambda u = \Pi_{\lambda_1,\lambda_2}f$. This completes the proof.
	\end{proof}

	
	\begin{lem}\label{fl2}
		Let $k \geq 1$ and let $f \in \mathcal{C}^1([0,T];H^{k+1}(\Omega)\cap W^{1,\infty}(\Omega))$ be the solution of the Vlasov - viscous Burgers' problem \eqref{eq:continuous-model}-\eqref{contburger} and let $u \in \mathcal{C}^0([0,T];H^{k+1}(I)\cap W^{1,\infty}(I))$ be the associated fluid velocity. Let $(u_h,f_h) \in X_h \times \mathcal{Z}_h $ be the approximation of $(u,f)$, respectively, then
		\begin{equation}\label{fL2}
			\frac{{\rm d}}{{\rm d}t}\|\theta_f\|_{0,\mathcal{T}_h}^2 \leq Ch^{2k+2} + C\|\theta_f\|_{0,\mathcal{T}_h}^2 + C\|u - u_h\|_{0,I_h}^2,
		\end{equation}
		where $C$ depends on the polynomial degree $k$, the shape regularity of the partition and it is also depends on $(u,f)$. 
	\end{lem}
	
	\begin{proof}
		From equation \eqref{err1} with triangle inequality, we have
		\begin{equation}\label{estimate}
			\begin{aligned}
				&\frac{1}{2}\frac{{\rm d}}{{\rm d}t}\|\theta_f\|^2_{0,\mathcal{T}_h} + \sum_{i,j}\int_{J_j}\left(\frac{2\lambda_1 - 1}{2}\right)\vert v\vert[\![\theta_f]\!]^2_{i-1/2,v}\,{\rm d}v  
				\\
				&+ \sum_{i,j}\int_{I_i}\left(\frac{2\lambda_2 - 1}{2}\right)\vert u_h - v\vert[\![\theta_f]\!]^2_{x,j-1/2}\,{\rm d}x \leq \left((\eta_f)_t,\theta_f\right) + \frac{1}{2}\|\theta_f\|^2_{0,\mathcal{T}_h} + |\mathcal{K}^1| + |\mathcal{K}^2| + |T_1|.
			\end{aligned}
		\end{equation}
		Here,
		\[
		T_1 = \sum_{i,j}\int_{T_{ij}}(u-u_h)\,\partial_vf\,\theta_f\,{\rm d}v\,{\rm d}x.
		\]
		First estimate $T_1$, 
		\begin{equation}\label{T1estimate}
			\begin{aligned}
				|T_1| & = \left\vert\sum_{i,j}\int_{T_{ij}}(u-u_h)\,\partial_v f\, \theta_f\,{\rm d}v\,{\rm d}x\right\vert \leq \|u - u_h\|_{0,I_h}\|\partial_v f\|_{L^\infty(\Omega)}\|\theta_f\|_{0,\mathcal{T}_h}
				\\
				&\leq \|u - u_h\|^2_{0,I_h}\|\partial_v f\|_{L^\infty(\Omega)} + \|\partial_v f\|_{L^\infty(\Omega)}\|\theta_f\|^2_{0,\mathcal{T}_h},
			\end{aligned}
		\end{equation}
		here, in second step use the H\"older inequality, in third step the Young's inequality.\\
		Now, $\mathcal{K}^1$ is estimated from Lemma \ref{K1estimate} and the arithmetic-geometric inequality,
		\begin{equation}\label{K1Estimate}
			|\mathcal{K}^1| \leq Ch^{2k+2}\|f\|_{k+1,\Omega}^2 + C\|\theta_f\|^2_{0,\mathcal{T}_h}.
		\end{equation}
		To deal with $\mathcal{K}^2$, we observe that the bound \eqref{k2estimate} in Lemma \ref{K2estimate} yields
		\begin{equation}\label{K2Estimate}
			\begin{aligned}
				|\mathcal{K}^2| &\leq C\left(h^k\|u - u_h\|_{\infty,I_h} 
				 + Ch^{k+1}\|u\|_{W^{1,\infty}(I)}\right) \|f\|_{k+1,\Omega}\|\theta_f\|_{0,\mathcal{T}_h}
				\\
				&\leq Ch^k\left(h\|u\|_{W^{1,\infty}(I)} + h^{k+\frac{1}{2}}\|u\|_{k+1,I} + h^{-\frac{1}{2}}\|u - u_h\|_{0,I_h}\right)\|f\|_{k+1,\Omega}\|\theta_f\|_{0,\mathcal{T}_h} 
				\\
				&\quad + Ch^{k+1}\|u\|_{W^{1,\infty}(I)}\|f\|_{k+1,\Omega} \|\theta_f\|_{0,\mathcal{T}_h}
				\\
				&\leq C_{u,f}h^{2k+2} + C\|u - u_h\|_{0,I_h}^2 + C\|\theta_f\|^2_{0,\mathcal{T}_h}. 
			\end{aligned}
		\end{equation}
		here, in second step we use Lemma \ref{L:uinf}, in third step use the Young's inequality.
		
		Now, substituting the estimates \eqref{T1estimate}-\eqref{K2Estimate} into \eqref{estimate} and using the fact that the last two terms on left hand side are non-negative, the H\"older inequality, projection estimate \eqref{wprojection} with the Young's inequality, we obtain
		\begin{equation*}
			\begin{aligned}
				\frac{{\rm d}}{{\rm d}t}\|\theta_f\|^2_{0,\mathcal{T}_h} \leq Ch^{2k+2} + C\|\theta_f\|^2_{0,\mathcal{T}_h} + C\|u - u_h\|_{0,I_h}^2.
			\end{aligned}
		\end{equation*}
		where $C$ is now independent of $h$ and $f_h$, and depends on $t$ and on the solution $(u,f)$ through its norm. This completes the proof.	
	\end{proof}

	\begin{thm}\label{mmma1}
		Let $k \geq 1$ and let $(u,f) \in \mathcal{C}^1(0,T;H^{k+1}(I)\cap W^{1,\infty}(I)) \times \mathcal{C}^1(0,T;H^{k+1}(\Omega)\cap W^{1,\infty}(\Omega))$ be the solution of the Vlasov - viscous Burgers' system. Let $(f_h,u_h,w_h) \in \mathcal{C}^1(0,T;\mathcal{Z}_h) \times \mathcal{C}^1(0,T;X_h) \times \mathcal{C}^0(0,T;X_h)$ be the DG-LDG approximation of \eqref{bh},\eqref{ch}-\eqref{ch1}, then for $\lambda = 1/2$ with $k$ even and $N_x$ odd or $\lambda > \frac{1}{2}$, there exist a positive constant $C$ independent of $h$ such that for all $t \in (0,T]$
		\begin{equation*}
			\|u - u_h\|_{L^\infty(0,T;L^2(I))} + \|w - w_h\|_{L^2([0,T]\times I)} + \|f - f_h\|_{L^\infty(0,T;L^2(\Omega))} \leq C\tilde{C}h^{k+1},
		\end{equation*}
	where, $\tilde{C} = 2$ for $\lambda = 1/2$ and $\tilde{C} = \left(1 + \epsilon^{-1}\right)^\frac{1}{2}\left(1 + \epsilon^{-1}h^{2k+1}\right)^\frac{1}{2}$ for $\lambda > 1/2$.
	\end{thm}
	
	\begin{proof}
		A use of equations \eqref{p1}, \eqref{p} and triangle inequality implies
		\begin{equation*}
			\begin{aligned}
				\|u - u_h\|_{0,I_h} + \|w - w_h\|_{0,I_h} + \|f - f_h\|_{0,\mathcal{T}_h} &\leq \|\eta_u\|_{0,I_h} + \|\theta_u\|_{0,I_h} + \|\eta_w\|_{0,I_h}
				\\
				& + \|\theta_w\|_{0,I_h} + \|\eta_f\|_{0,\mathcal{T}_h} + \|\theta_f\|_{0,\mathcal{T}_h}.
			\end{aligned}
		\end{equation*}
		From equations \eqref{uprojection} and \eqref{wprojection}, we know the estimate of $\|\eta_u\|_{0,I_h}, \|\eta_w\|_{0,I_h}$ and $\|\eta_f\|_{0,\mathcal{T}_h}$, respectively, so it is enough to estimate $\|\theta_u\|_{0,I_h} + \|\theta_w\|_{0,I_h} + \|\theta_f\|_{0,\mathcal{T}_h}$.\\
		Adding \eqref{uL2}-\eqref{uL21} and \eqref{fL2} with the Young's inequality, we obtain
		\begin{equation}\label{etauf}
			\begin{aligned}
				\frac{{\rm d}}{{\rm d}t}\left(\|\theta_u\|^2_{0,I_h} + \|\theta_f\|^2_{0,\mathcal{T}_h}\right) + \|\theta_w\|^2_{0,I_h} &\leq C\left(Ah^{2k+2} + B\|\theta_u\|^2_{0,I_h} + h^{-\frac{3}{2}}\|\theta_u\|^3_{0,I_h}  \right.
				\\
				&\left. \quad +\,\, h^{-1}\|\theta_u\|^4_{0,I_h} + \|\theta_f\|^2_{0,\mathcal{T}_h}\right)
				\\
				& \leq C\left(Ah^{2k+2} + B\|\theta_u\|^2_{0,I_h}   +\,\, h^{-3}\|\theta_u\|^4_{0,I_h} + \|\theta_f\|^2_{0,\mathcal{T}_h}\right) 
			\end{aligned}
		\end{equation} 
	here $A = 1, B = 1$ for $\lambda = 1/2$ and $A = 1 + \epsilon^{-1}, B = 1 + \epsilon^{-1}h^{2k+1}$ for $\lambda > 1/2$.
		Setting
		\begin{equation}\label{vertiii}
			\vertiii{\left(\theta_f,\theta_u,\theta_w\right)}^2 := \|\theta_f\|^2_{0,\mathcal{T}_h} + \|\theta_u\|_{0,I_h}^2 + \int_{0}^{t}\|\theta_w\|^2_{0,I_h}\,{\rm d}s
		\end{equation}
		and a function $\Phi$ as
		\begin{equation}\label{Phi}
			\Phi(t) = Ah^{2k+2} + \int_{0}^{t}\left(B\vertiii{\left(\theta_f,\theta_u,\theta_w\right)}^2 + h^{-3}\vertiii{\left(\theta_f,\theta_u,\theta_w\right)}^4\right)\,{\rm d}s.
		\end{equation}
		An integration of equation \eqref{etauf} with respect to $t$ from $0$ to $t$ and a use of \eqref{vertiii}-\eqref{Phi} shows
		\begin{equation*}
			\vertiii{\left(\theta_f,\theta_u,\theta_w\right)}^2 \leq C\Phi(t).
		\end{equation*}
		Without loss of generality assume that $\vertiii{\left(\theta_f,\theta_u,\theta_w\right)} > 0$, otherwise, we add an arbitrary small quantity say $\delta$ and proceed as in a similar way as describe below and then pass the limit as $\delta \rightarrow 0$. Note that $0 < \Phi(0) \leq \Phi$ and $\Phi$ is differentiable.
		An differentiation of $\Phi$ with respect to $t$, implies
		\begin{equation*}\label{44}
			\begin{aligned}
				\partial_t\Phi(t) &= B\vertiii{\left(\theta_f,\theta_u,\theta_w\right)}^2 + h^{-3}\vertiii{\left(\theta_f,\theta_u,\theta_w\right)}^4
				\\
				&\leq C\left(B\Phi(t) + h^{-3}\left(\Phi(t)\right)^2\right).
			\end{aligned}
		\end{equation*}
Moreover, $\partial_t\Phi(t) > 0$ and hence, $\Phi(t)$ is strictly monotonically increasing function which is also positive. An integration in time $t$ yields
\begin{equation*}\label{Phit}
	\int_{0}^{t}\frac{\partial_s\Phi(s)}{\Phi(s)\left(B + h^{-3}\Phi(s)\right)}\,{\rm d}s \leq \int_{0}^{t}C\,{\rm d}s \leq CT.
\end{equation*}
After evaluating integration on the left hand side exactly by using $\Phi(0) = Ah^{2k+2}$ and taking exponential both side, we obtain
\begin{equation*}
	\Phi(t)\left(B - Ah^{2k-1}(e^{CBT}-1)\right) \leq ABe^{CBT}h^{2k+2}.
\end{equation*}
Now, choose small $h > 0$ such that $\left(B - h^{2k-1}(e^{CBT}-1)\right) > 0$ this gives
\[
\Phi(t) \leq \tilde{C}\,e^{CBT}h^{2k+2},
\]
where, $\tilde{C} = \left(1 + \epsilon^{-1}\right)\left(1 + \epsilon^{-1}h^{2k+1}\right)$. This completes the proof for all $t \in (0,T]$.
	\end{proof}

\begin{rem}
	By Lemma \ref{fhbd} and Theorem \ref{mmma1} along with equation \eqref{3}, we have
	\[
	f_h \in L^\infty(0,T;L^2(\mathcal{T}_h)), \quad u_h \in L^\infty(0,T;L^2(I_h)) \quad \mbox{and} \quad w_h \in L^2([0,T]\times I_h).
	\]
	Using these bounds, we can improved the earlier local-in-time existence result for the discrete problem to global-in-time existence result by extending the interval of existence.
\end{rem}


	\section{Numerical simulation}\label{computation}
	
	In this section, we give some results from our numerical simulations of the Vlasov - viscous Burgers' system:
	
	\begin{equation}\label{87}
		\begin{aligned}
			\partial_tf + v\,\partial_xf + \partial_v\left(\left(u - v\right)f\right) = F(t,x,v), x \in [0,L] = I, v \in [-M,M] = J, t > 0
		\end{aligned}
	\end{equation}
	\begin{equation}\label{88}
		\begin{aligned}
			\partial_tu + u\,\partial_x u - \epsilon \partial^2_x u = \rho V - \rho u + G(x,t), x \in [0,L], t > 0,
		\end{aligned}
	\end{equation}
	with the same boundary and initial conditions as in \eqref{eq:continuous-model}-\eqref{contburger}. Our proposed scheme reduces the problem \eqref{87}-\eqref{88} into a system of ODEs:
	\begin{equation}\label{89}
		\frac{{\rm d}}{{\rm d}t}\vec{\alpha}(t) = \mathcal{L}(\vec{\alpha},t)
	\end{equation}
	and
	\begin{equation}\label{90}
		\frac{{\rm d}}{{\rm d}t}\vec{\beta}(t) = \mathcal{M}(\vec{\beta},t),
	\end{equation}
	where, $\vec{\alpha}(t)$ and $\vec{\beta}(t)$ are the coefficient vectors of $f_h$ and $u_h$, respectively. To further approximate the solutions of the systems \eqref{89}-\eqref{90}, we use the third order TVD RK scheme \cite{gottlieb1998total}. 
	
	In the figures, we use the notation $k_x, k_v$ for the degree of polynomials in $x$ and $v$-variables, respectively. $N_x$ and $N_v$ denote the number of elements taken in $x$ and $v$-domains, respectively. L2f and L2u are errors of $f(t,x,v)$ and $u(t,x)$ in $L^2(\Omega)$ and $L^2(I)$-norms, respectively at the final time $t = T$.
	
	We test the order of convergence of the proposed scheme on the system \eqref{87}-\eqref{88} in two different scenarios (see Example \ref{exm1} and Example \ref{exm3} below). We have run the simulations for both these examples with $I = [0,2\pi], J = [-1,1], T = 0.1$. 
	\begin{exm}\label{exm1}
		Take  
		\[
		F(t,x,v) = \left( v - 1 \right)\cos(x-t)e^{-v^2} - \left( 2v\left(\sin(x-t) - v \right) + 1 \right)\left( 1 + \sin(x-t)\right)e^{-v^2}
		\]
		\[
		G(t,x) = \left(\sin(x - t) - 1 \right)\cos(x - t) + \epsilon\,\sin(x - t) + \left(1.493648265624854\right)\left(1 + \sin(x-t)\right) \sin(x-t)
		\]
		and 
		\[
		f(0,x,v) = \left( 1 + \sin(x) \right)e^{ - v^2}, \quad u(0,x) = \sin(x).
		\]
		The exact solution of the problem is given by
		\begin{equation*}
			\begin{aligned}
				f(t,x,v) &= \left( 1 + \sin(x - t)\right) e^{ - v^2}
				\\
				u(t,x) &= \sin(x - t).
			\end{aligned}
		\end{equation*}
	
	\end{exm}
	
\begin{figure}
	\centering
			\includegraphics[width=6.5cm]{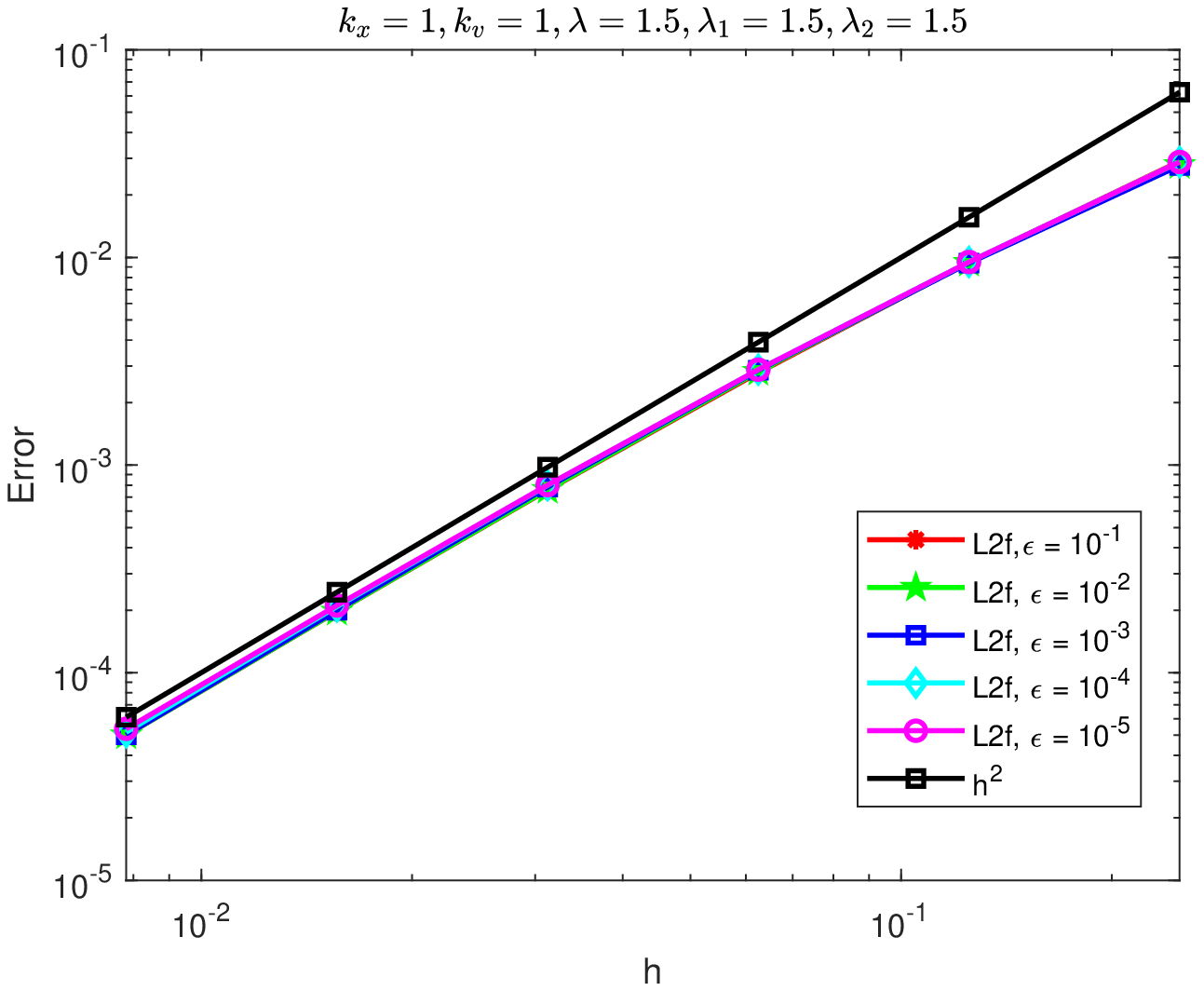}
		\includegraphics[width=6.5cm]{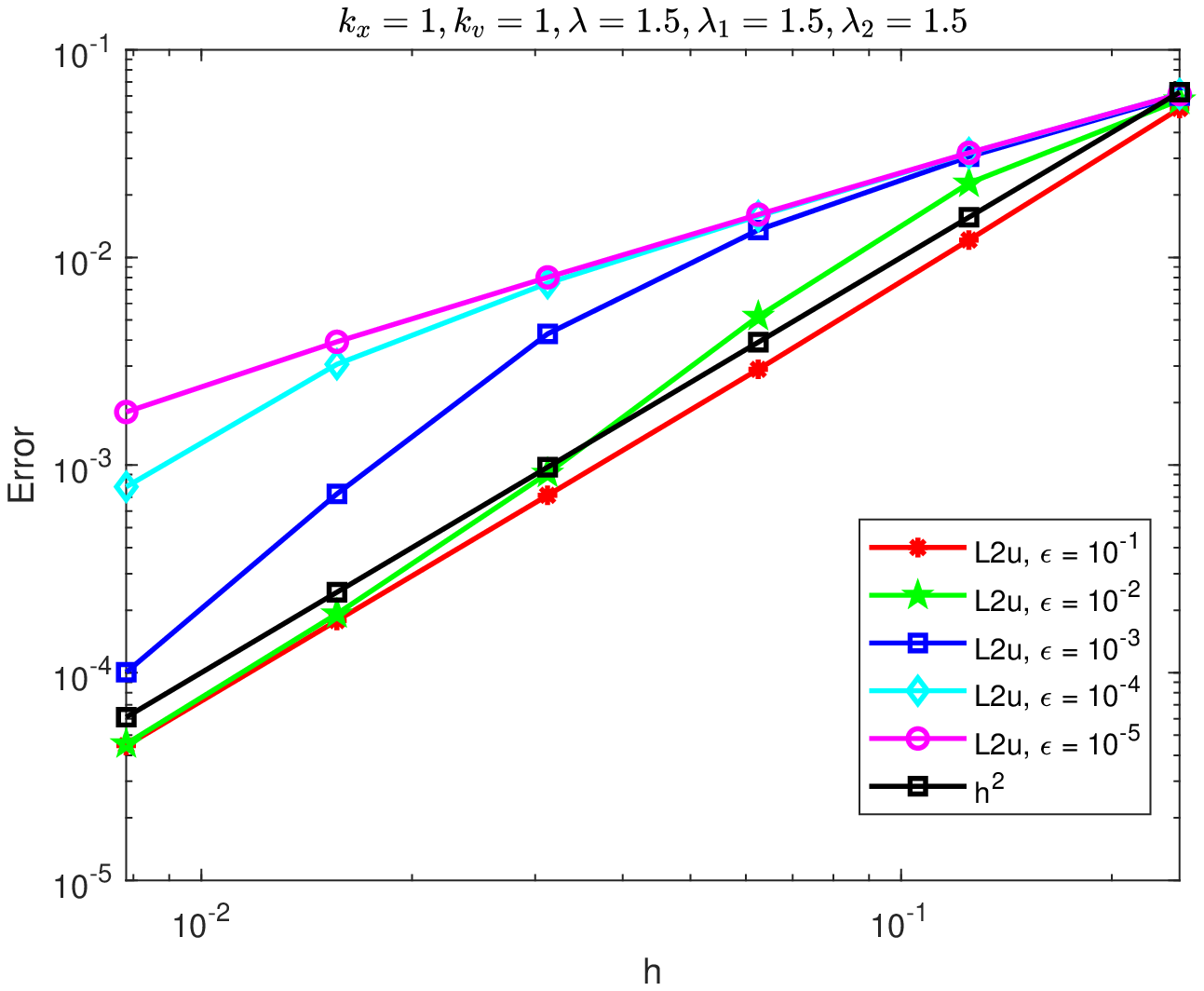}
\caption{Convergence rate for different values of $\epsilon$ of the distribution function $f$(Left) and the velocity $u$(Right) for the Example \ref{exm1}; $k_x = 1, k_v = 1$  }
\label{fig:1-1}
\end{figure}

\begin{figure}
	\centering
		\includegraphics[width=6.5cm]{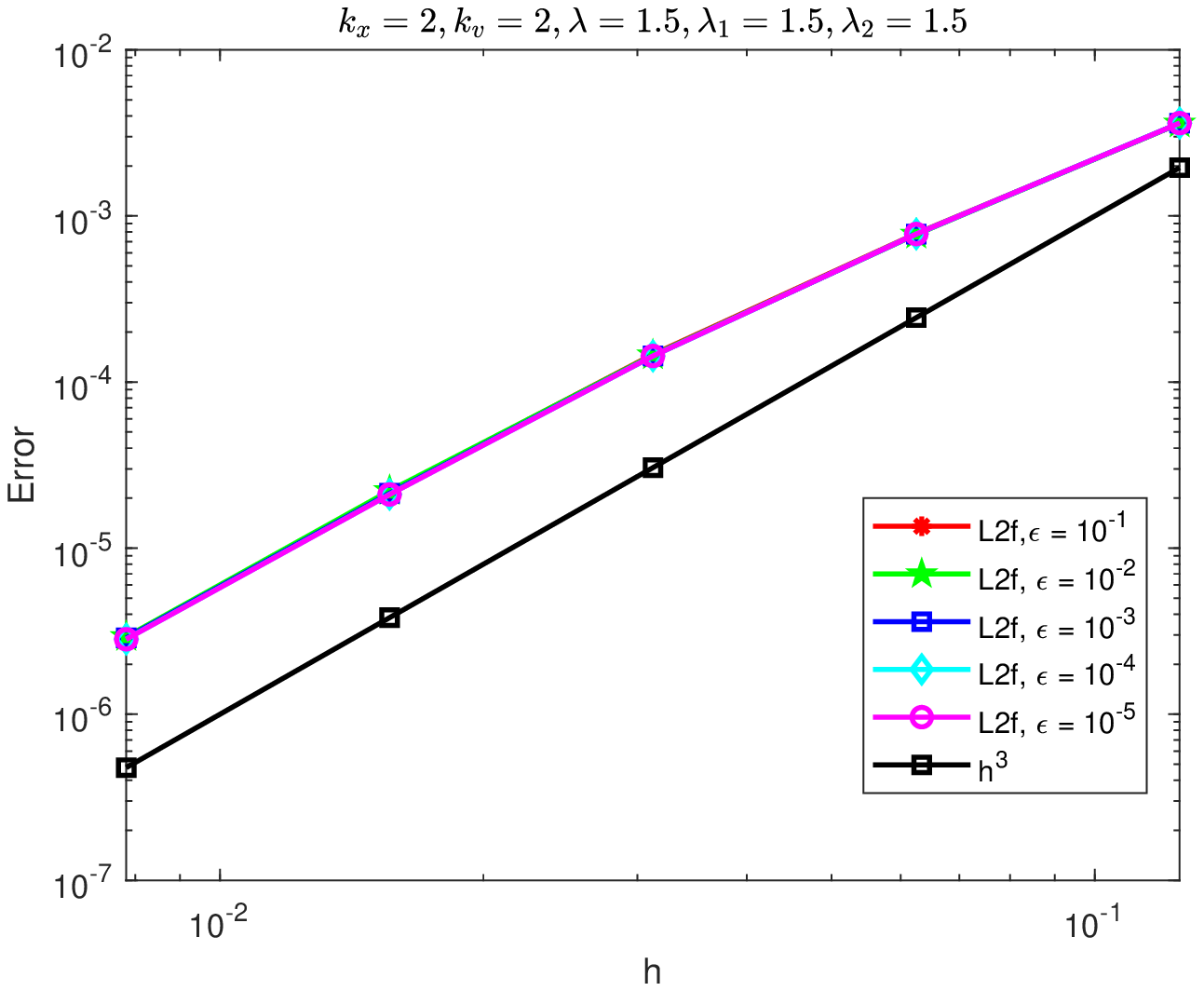}
		\includegraphics[width=6.5cm]{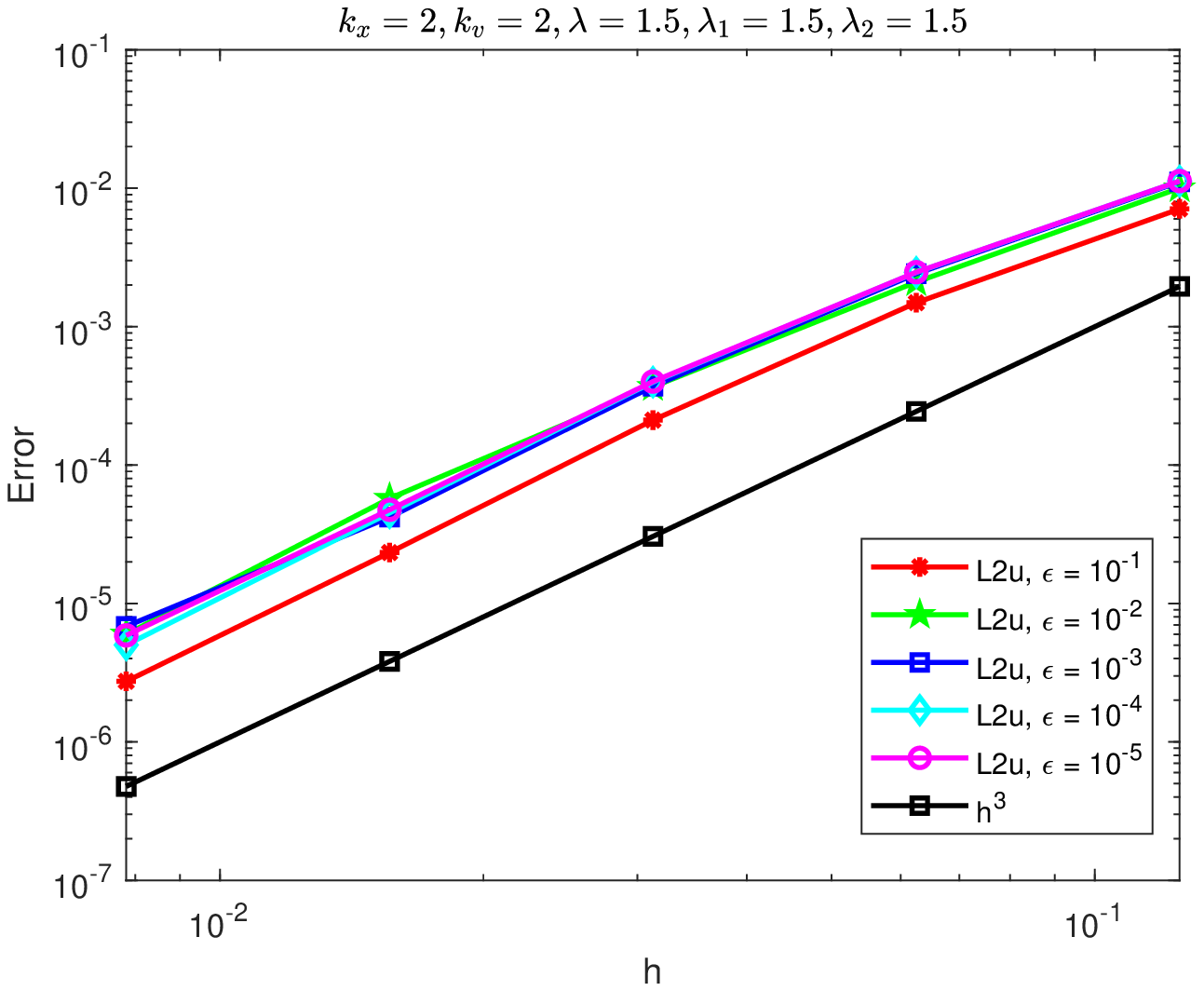}
	\caption{Convergence rate for different values of $\epsilon$ of the distribution function $f$(Left) and the velocity $u$(Right) for the Example \ref{exm1}; $k_x = 2, k_v = 2$  }
	\label{fig:1-2}
\end{figure}

\begin{figure}
	\centering
		\includegraphics[width=6.5cm]{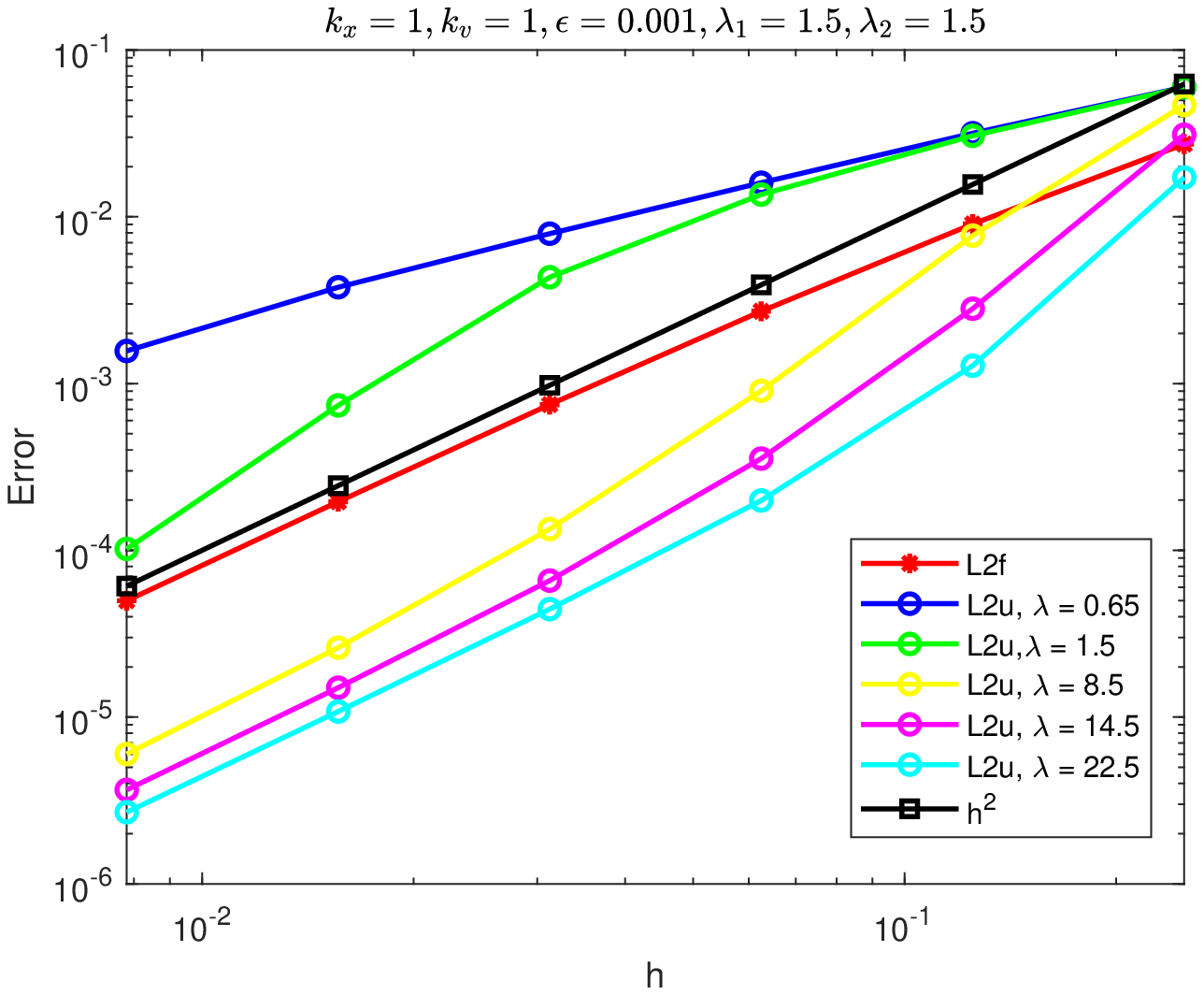}
		\includegraphics[width=6.5cm]{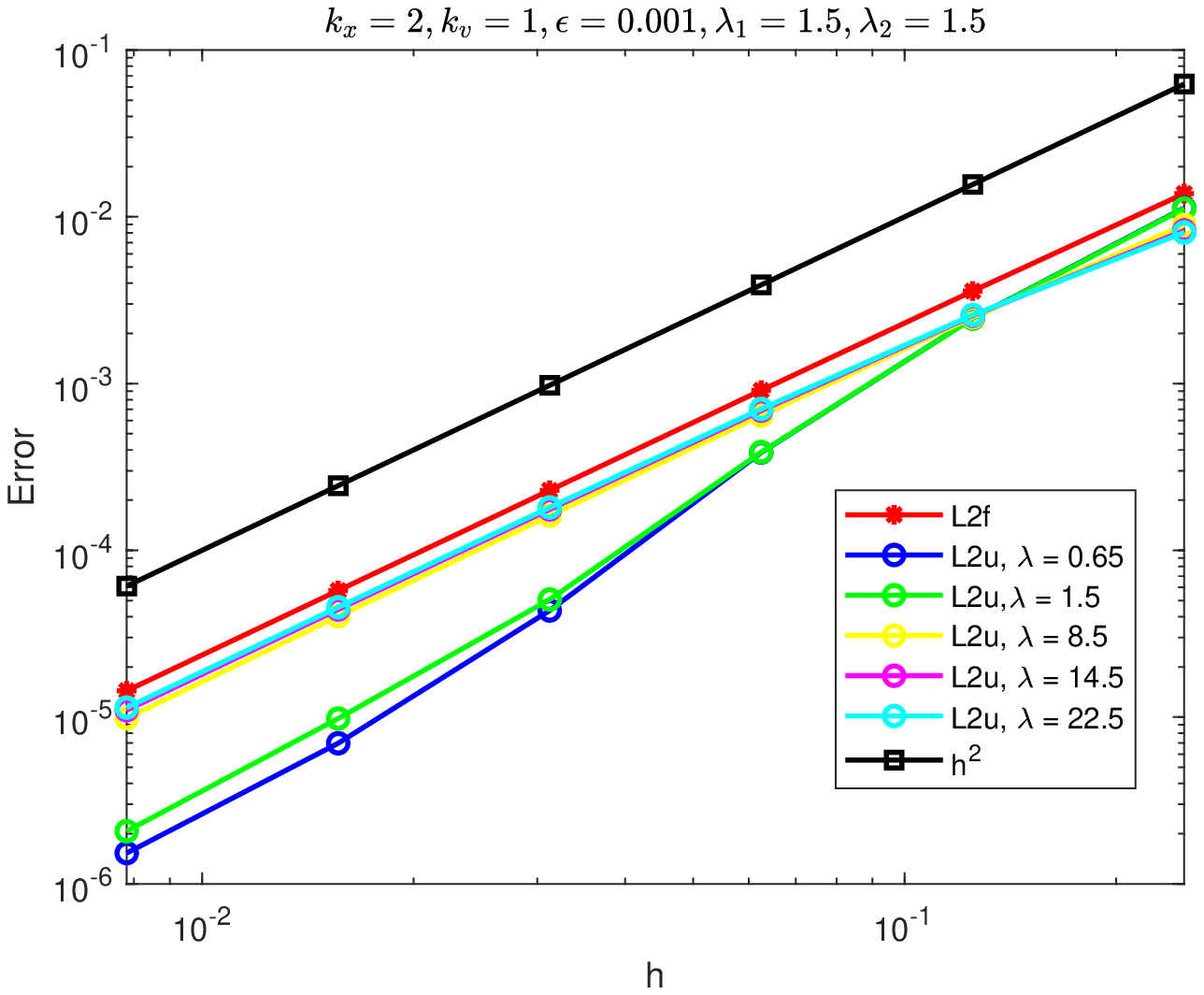}
	\caption{Convergence rate for different values of $\lambda$ of $u$ for the Example \ref{exm1}; $k_x = 1, k_v = 1$ (Left), $k_x = 2, k_v = 1$ (Right) }
	\label{fig:1-3}
\end{figure}

\begin{exm}\label{exm3}
	Take  
	\begin{equation*}
		\begin{aligned}
				F(t,x,v) &= \frac{e^t}{\sqrt{2\pi}}e^{\frac{-v^2}{2}}\left( + \left(e^{-t}\left(cos(x) + \sin(x)\right) - v \right)\left(1 + \cos(x)\right)\left(9v - 5v^3\right) \right.
				\\
				&\qquad \left. -v\,\sin(x)\left(1 + 5v^2\right)\right)
		\end{aligned}
	\end{equation*}
	\begin{equation*}
		\begin{aligned}
			G(t,x) &= \left(-1 + \epsilon \right)\left(\cos(x) + \sin(x)\right)e^{-t} + e^{-2t}\cos(2x) 
			\\
			&\quad+ \frac{1}{\sqrt{2\pi}}\left(4.202186105579451\right) \left( 1 + \cos(x) \right)\left(\cos(x) + \sin(x)\right)
		\end{aligned}
	\end{equation*}
	and 
	\[
	f(0,x,v) = \frac{e^{\frac{-v^2}{2}}}{\sqrt{2\pi}}\left(1 + \cos(x)\right)\left(1 + 5v^2\right), \quad u(0,x) = \cos(x) + \sin(x).
	\]
	The exact solution of the problem is given by
	\begin{equation*}
		\begin{aligned}
			f(t,x,v) &= \frac{e^t}{\sqrt{2\pi}}e^{\frac{-v^2}{2}}\left(1 + \cos(x)\right)\left(1 + 5v^2\right)
			\\
			u(t,x) &= e^{-t}\left(cos(x) + \sin(x)\right).
		\end{aligned}
	\end{equation*} 
\end{exm}

\begin{figure}
	\centering
	\includegraphics[width=6.5cm]{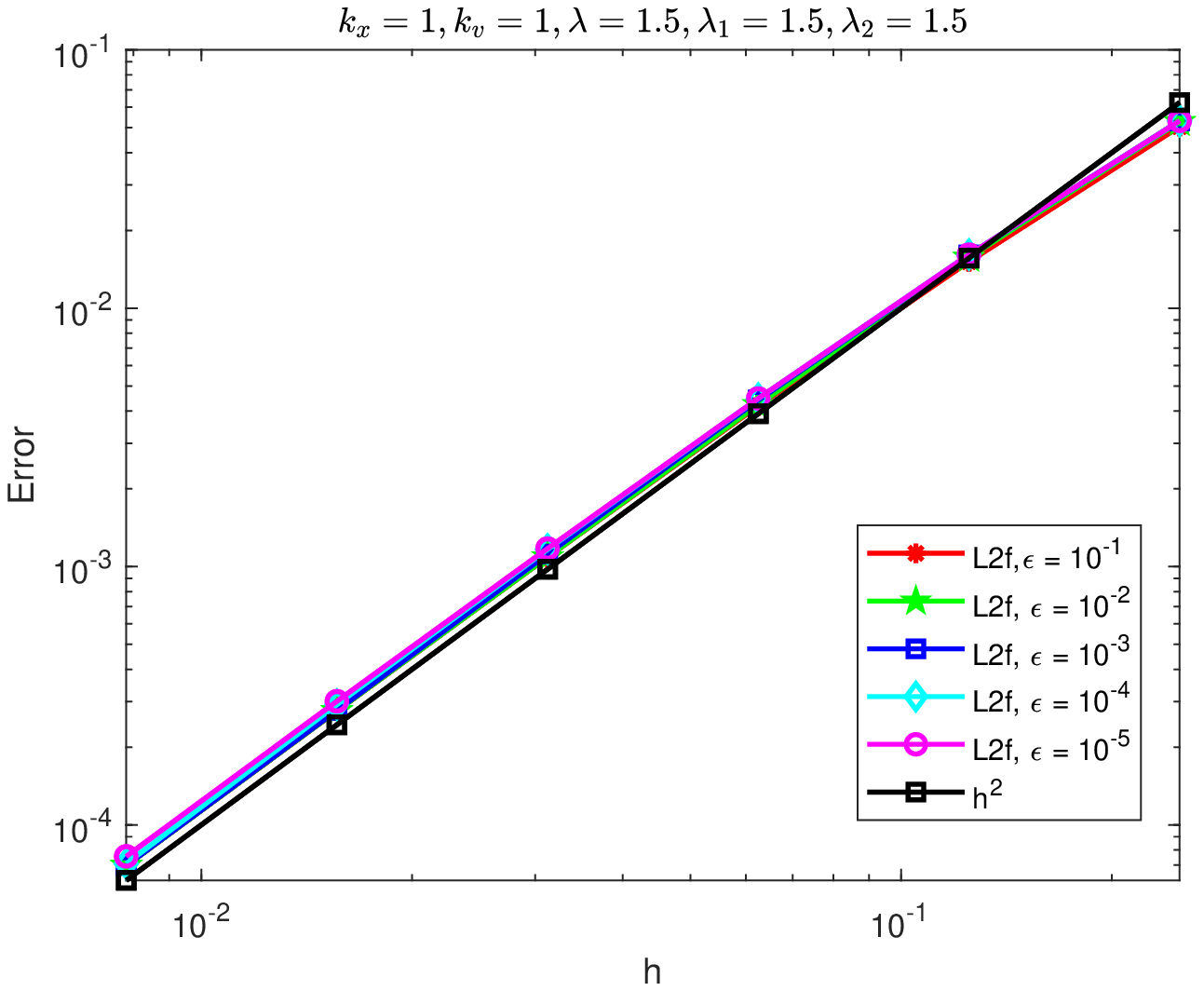}
	\includegraphics[width=6.5cm]{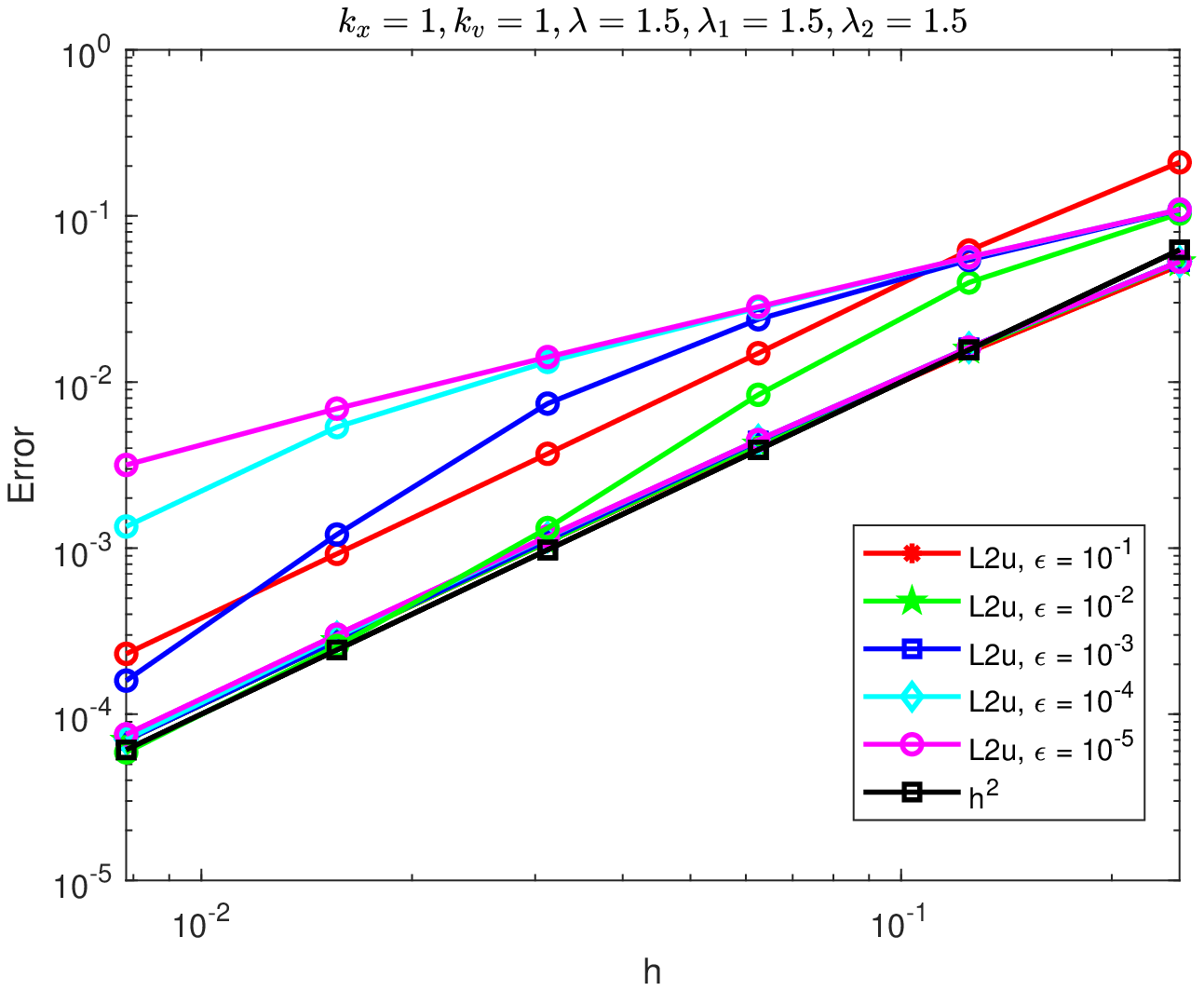}
	\caption{Convergence rate  for different values of $\epsilon$ of the  distribution function $f$(Left) and the velocity $u$(Right) for the Example \ref{exm3}; $k_x = 1, k_v = 1$  }
	\label{fig:2-1}
\end{figure}

\begin{figure}
	\centering
	\includegraphics[width=6.5cm]{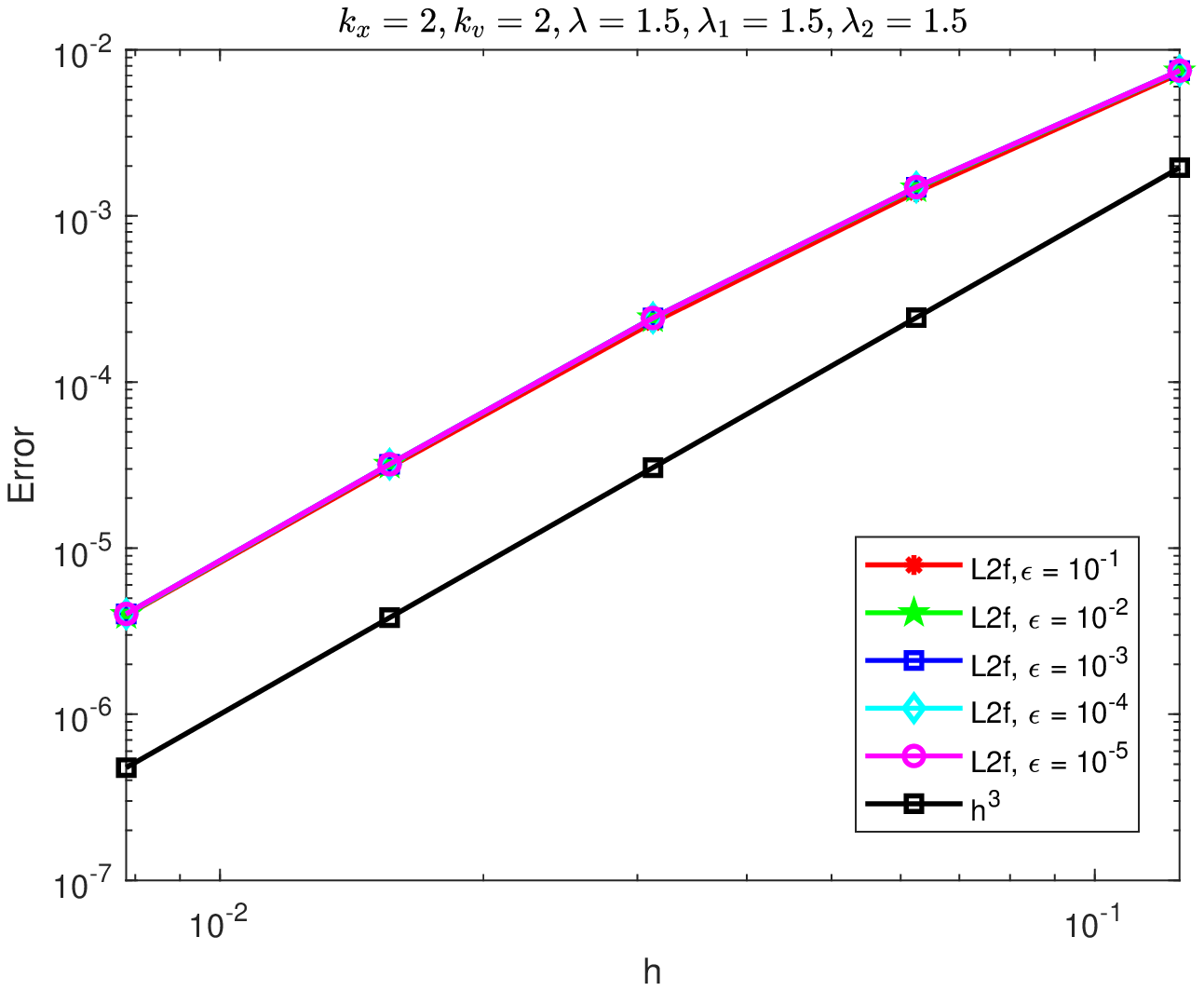}
	\includegraphics[width=6.5cm]{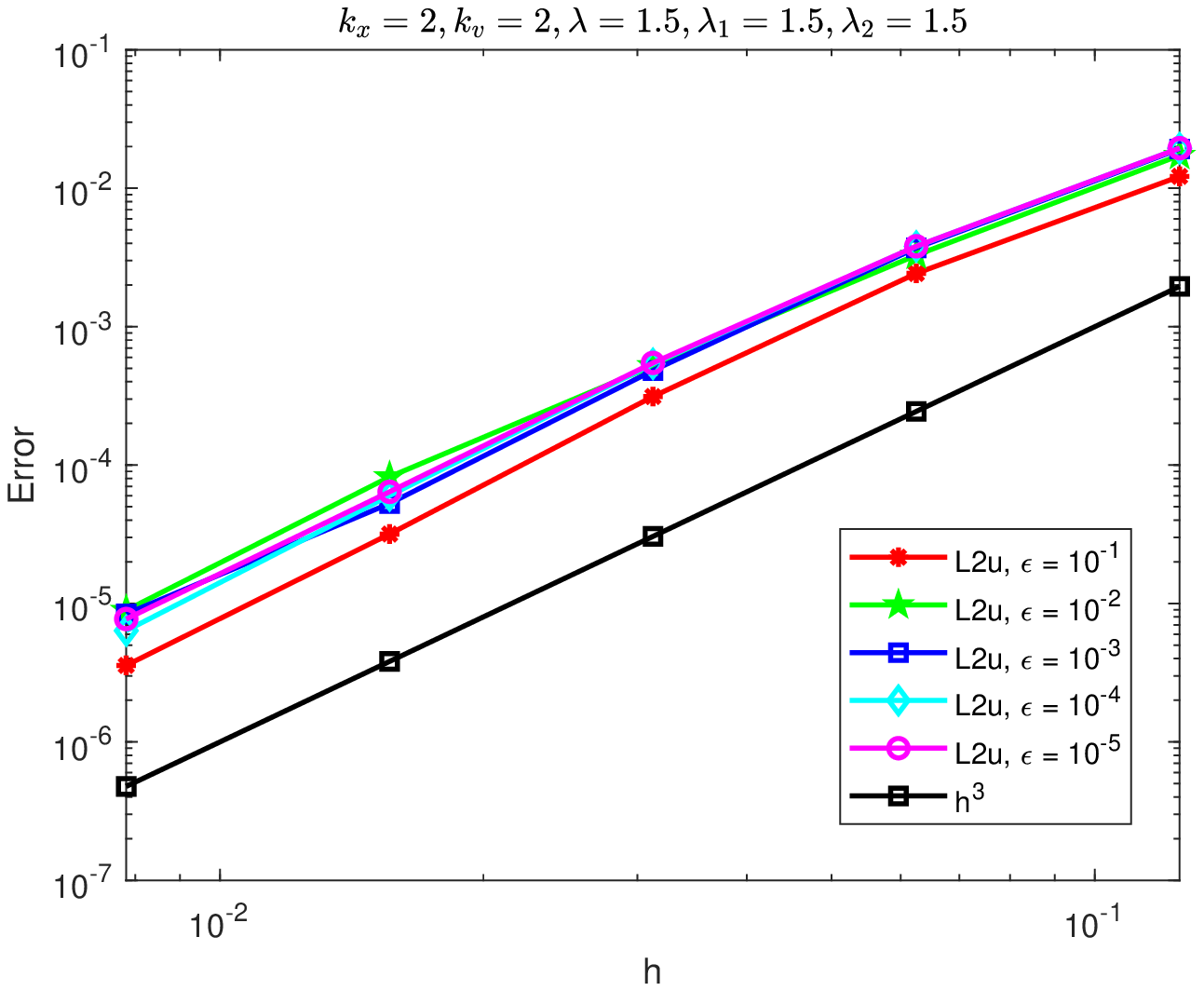}
	\caption{Convergence rate for different values of $\epsilon$ of the distribution function $f$(Left) and the velocity $u$(Right) for the Example \ref{exm3}; $k_x = 2, k_v = 2$  }
	\label{fig:2-2}
\end{figure}

\begin{figure}
	\centering
	\includegraphics[width=6.5cm]{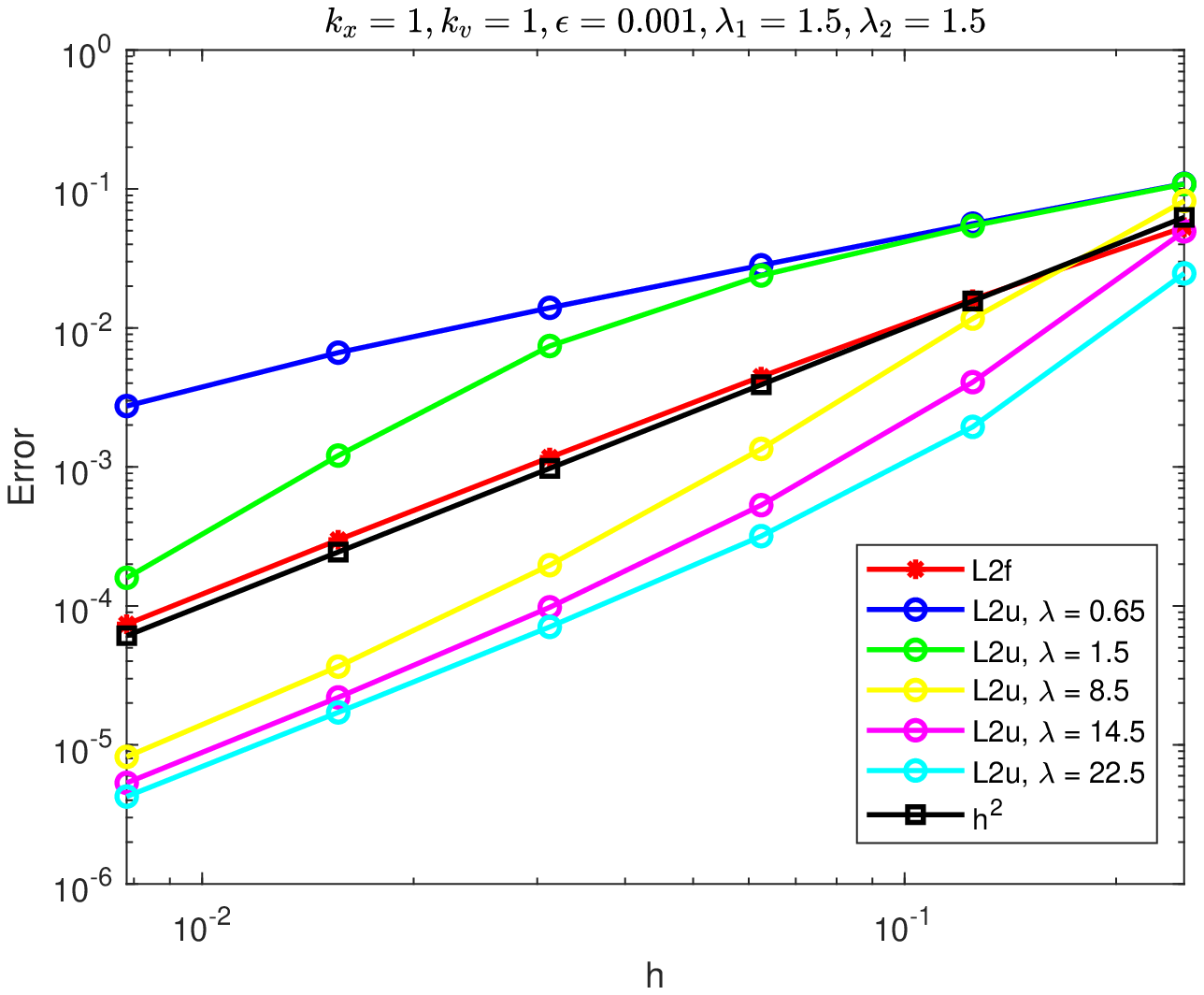}
	\includegraphics[width=6.5cm]{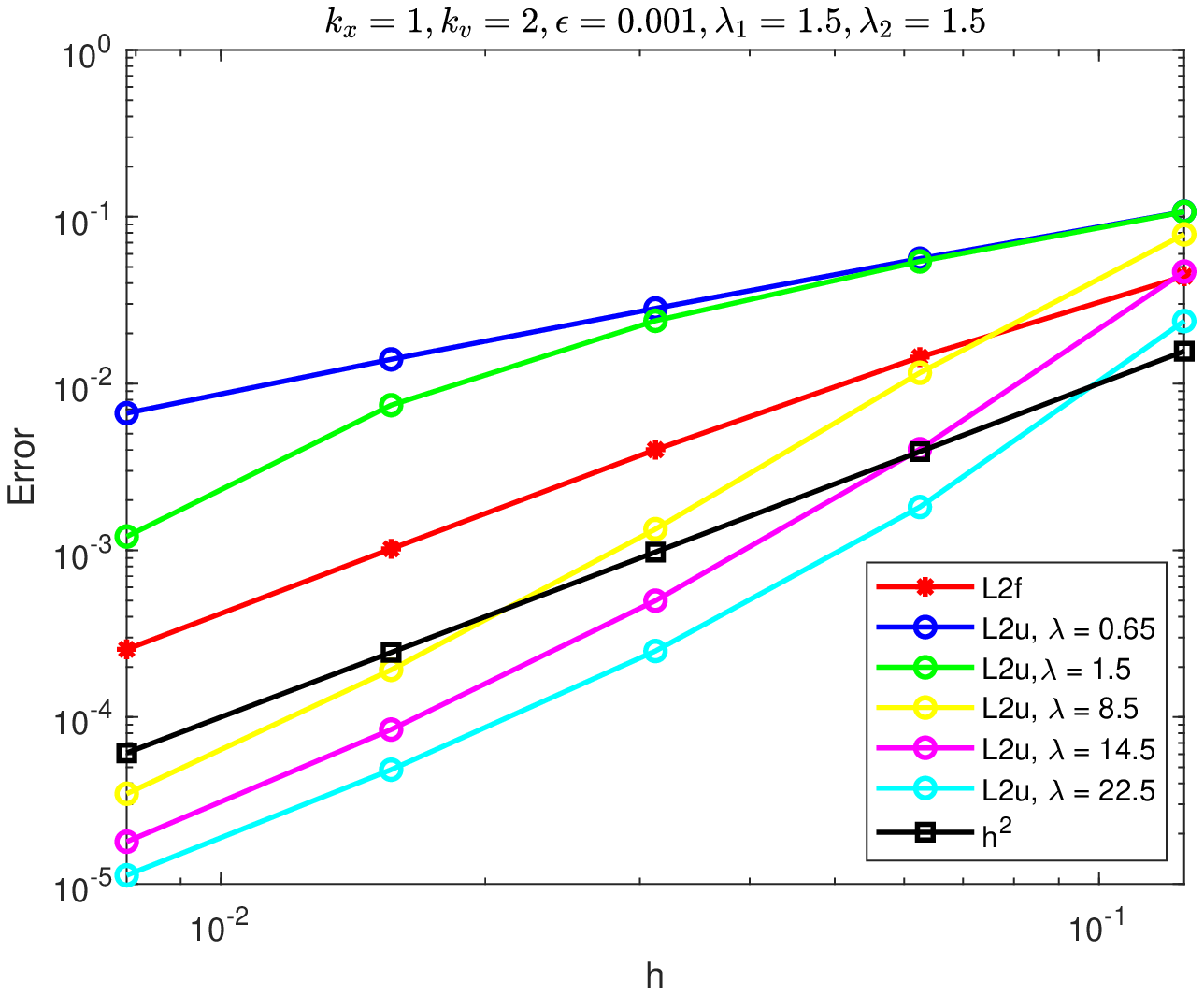}
	\caption{Convergence rate for different values of $\lambda$  of $u$ for the Example \ref{exm3}; $k_x = 1, k_v = 1$ (Left), $k_x = 1, k_v = 2$ (Right) }
	\label{fig:2-3}
\end{figure}	

\textbf{Observations:} We now make, below, several observations on numerical results of Examples \ref{exm1} and  \ref{exm3}.
\begin{itemize}
	\item [(i)] In Figures \ref{fig:1-1}--\ref{fig:1-2} and  Figures \ref{fig:2-1} --Figure \ref{fig:2-2}, $h$ varies like $2^{-m}, m=1,\cdots,7,$ while $\epsilon$ varies like $10^{-\ell}, \ell=1,\cdots,5$ with $\epsilon < h.$ From Theorem ~\ref{mmma1}, it is noted that the convergence for both $f$ and $u$ is  of order $O(\epsilon^{-1/2} h^{k+1}).$  In Figures \ref{fig:1-1} and  \ref{fig:2-1},
piecewise polynomial of degree $k=1$ is used and the computational order for $f$ matches with theoretical order of convergence, that is order $2$, but for $u$ there seems to be some deviation in order showing the effect of smaller $\epsilon.$
Figures \ref{fig:1-2} and \ref{fig:2-2} uses piecewise polynomial of degree $k=2$ and  computational order of convergence confirms the theoretical order of convergence for both $f$ and $u$, which seems to be uniform in $\epsilon.$
\item [(ii)] In Figure \ref{fig:1-3} and Figure \ref{fig:2-3}, we explore the dependence of the order of convergence on the magnitude of the parameter $\lambda$ appearing in the generalized fluxes for the Burgers' part. This experiment suggests a trade-off between the magnitude of the parameter $\lambda$ and that of the viscosity parameter $\epsilon$. More precisely, for smaller values of $\epsilon$, taking larger value of $\lambda$ helps stabilize the convergence rate.
\end{itemize}

In the next couple of experiments, we test the conservation properties of our proposed numerical scheme.
		
\begin{exm}\label{exm5.3}
	In \eqref{87}-\eqref{88}, let $I = [0,2\pi]$ and $J = [-5,5]$. Further, take $F = G = 0$. Consider the initial data: 
	\begin{equation*}
		\begin{aligned}
			f(0,x,v) &= \left\{
			\begin{aligned}
				&(1 + \sin(x))e^{-v^2} \quad \mbox{if}\quad v \in [-1,1]
				\\
				& \qquad 0 \qquad\qquad\qquad elsewhere.
			\end{aligned}
			\right. 
			\\
			u(0,x) &= sin(x).
		\end{aligned}
	\end{equation*}
	\end{exm}
	 
Note that we don't have access to an explicit representation of the exact solution. Now, we  take degree of polynomials $k_x = 1, k_v = 2$ and number of sub-intervals $N_x = 128, N_v = 128$ with $\epsilon = 0.1, \lambda = 1.5, \lambda_1 = 1.5, \lambda_2 = 1.5$. In Figure \ref{figm1}, we check the discrete mass and discrete momentum conservation properties of our numerical scheme. It confirms our findings in Lemmas \ref{lem:mass} and \ref{lem:momentum} on discrete mass and momentum conservation, respectively. Figure \ref{figm1} also shows that our numerical scheme dissipates the discrete energy. This hints at a result that such a discrete energy dissipation property should hold in general, but unlike in continuous case, we do not have any theoretical justification to substantiate this numerical evidence.

\begin{figure}
	\centering
	\includegraphics[width=6.5cm]{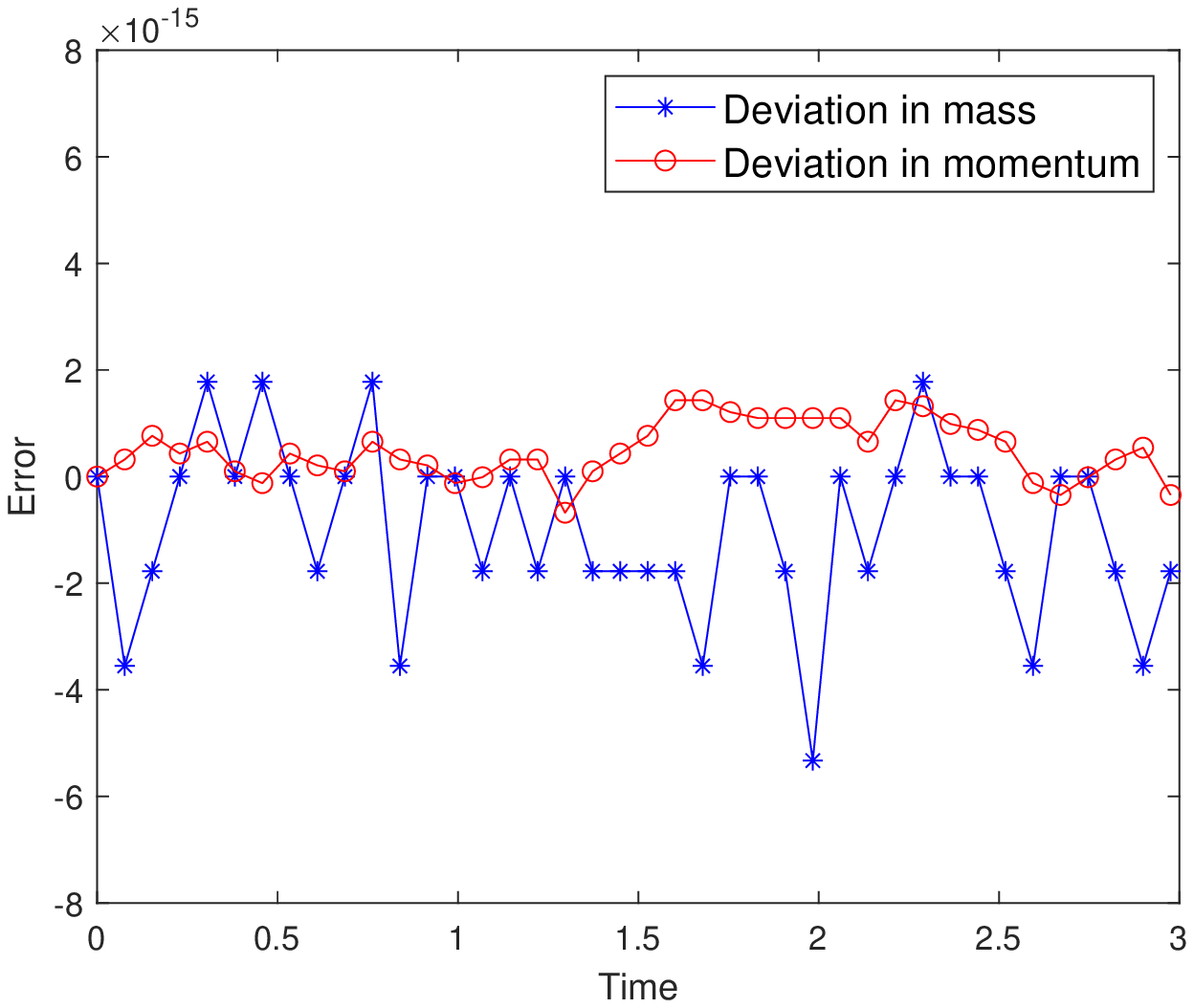} 
	\includegraphics[width=6.5cm]{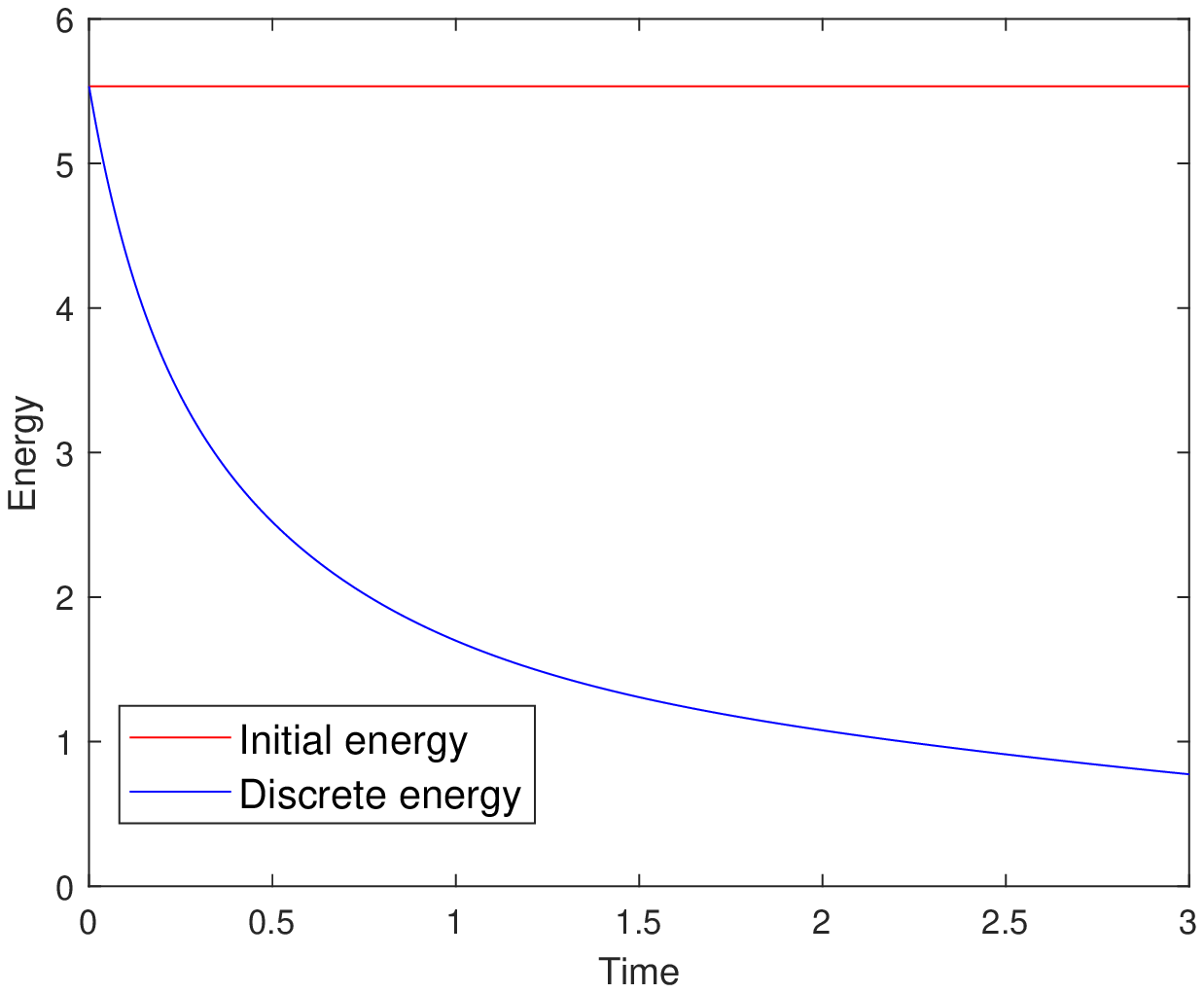} 
	\caption{Deviation in mass and momentum (Left) and energy dissipation (Right) for the Example \ref{exm5.3}}
	\label{figm1}
\end{figure}		

The following example repeats the above experiment for a different set of initial data.
\begin{exm}\label{exm5.4}
	In \eqref{87}-\eqref{88}, let $I = [0,2\pi]$ and $J = [-5,5]$. Further, take $F = G = 0$. Consider the initial data:
	\begin{equation*}
		\begin{aligned}
			f(0,x,v) &= \left\{
			\begin{aligned}
				&\frac{e^{\frac{-v^2}{2}}}{\sqrt{2\pi}}\left(1 + \cos(x)\right)\left(1 + 5v^2\right) \quad \mbox{if}\quad v \in [-1,1]
				\\
				&\qquad 0 \qquad\qquad\qquad\qquad elsewhere.
			\end{aligned}
			\right. 
			\\
			u(0,x) &= sin(x) + cos(x).
		\end{aligned}
	\end{equation*}
	\end{exm}
	Take degree of polynomials $k_x = 1, k_v = 2$ and number of sub-intervals $N_x = 128, N_v = 128$ with $\epsilon = 0.1, \lambda = 1.5, \lambda_1 = 1.5, \lambda_2 = 1.5$. In Figure \ref{figm2}, we note that deviation are up to the machine error and these results are in agreement with the theoretical findings. 
	
\begin{figure}
	\centering
	\includegraphics[width=6.5cm]{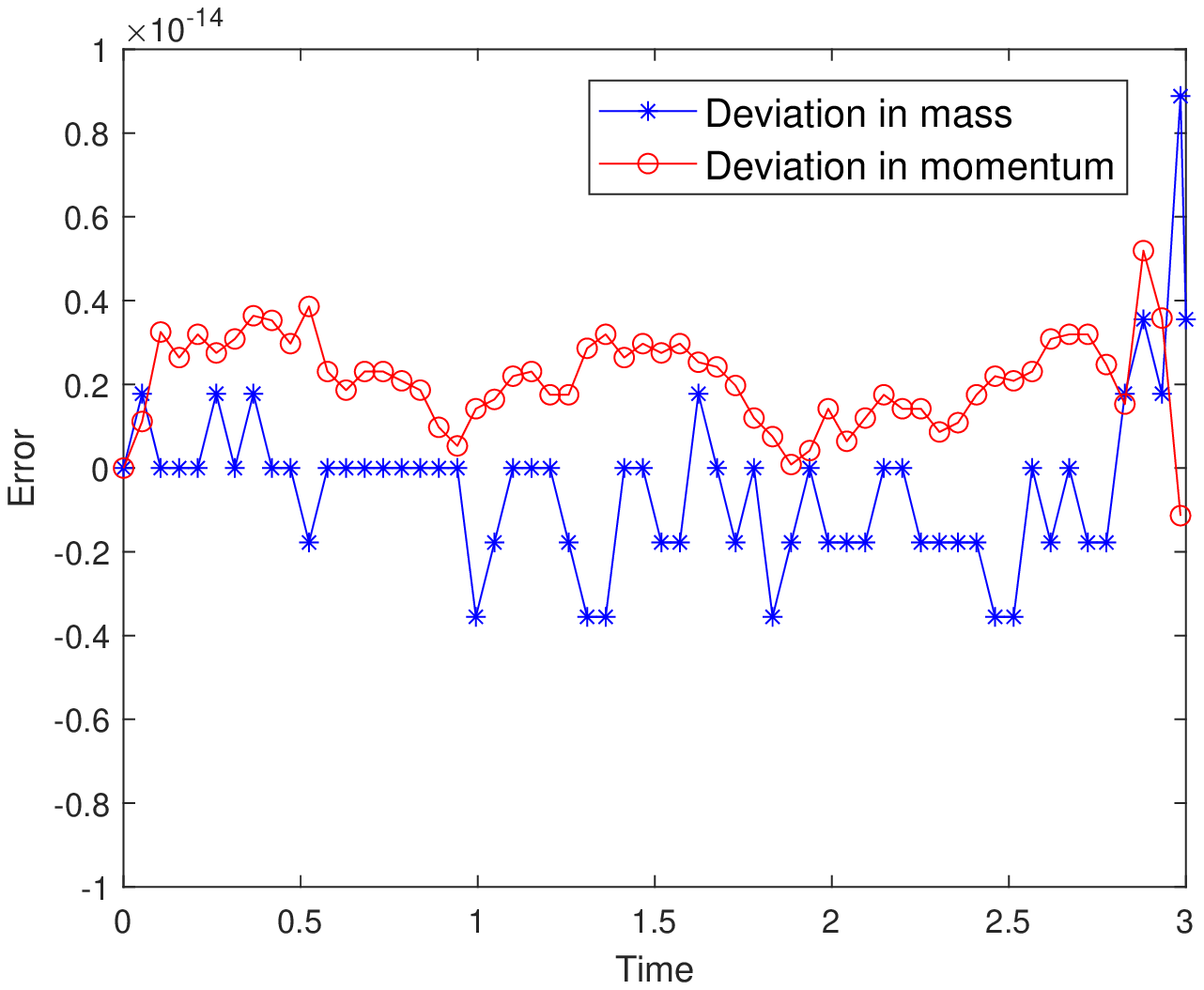} 
	\includegraphics[width=6.5cm]{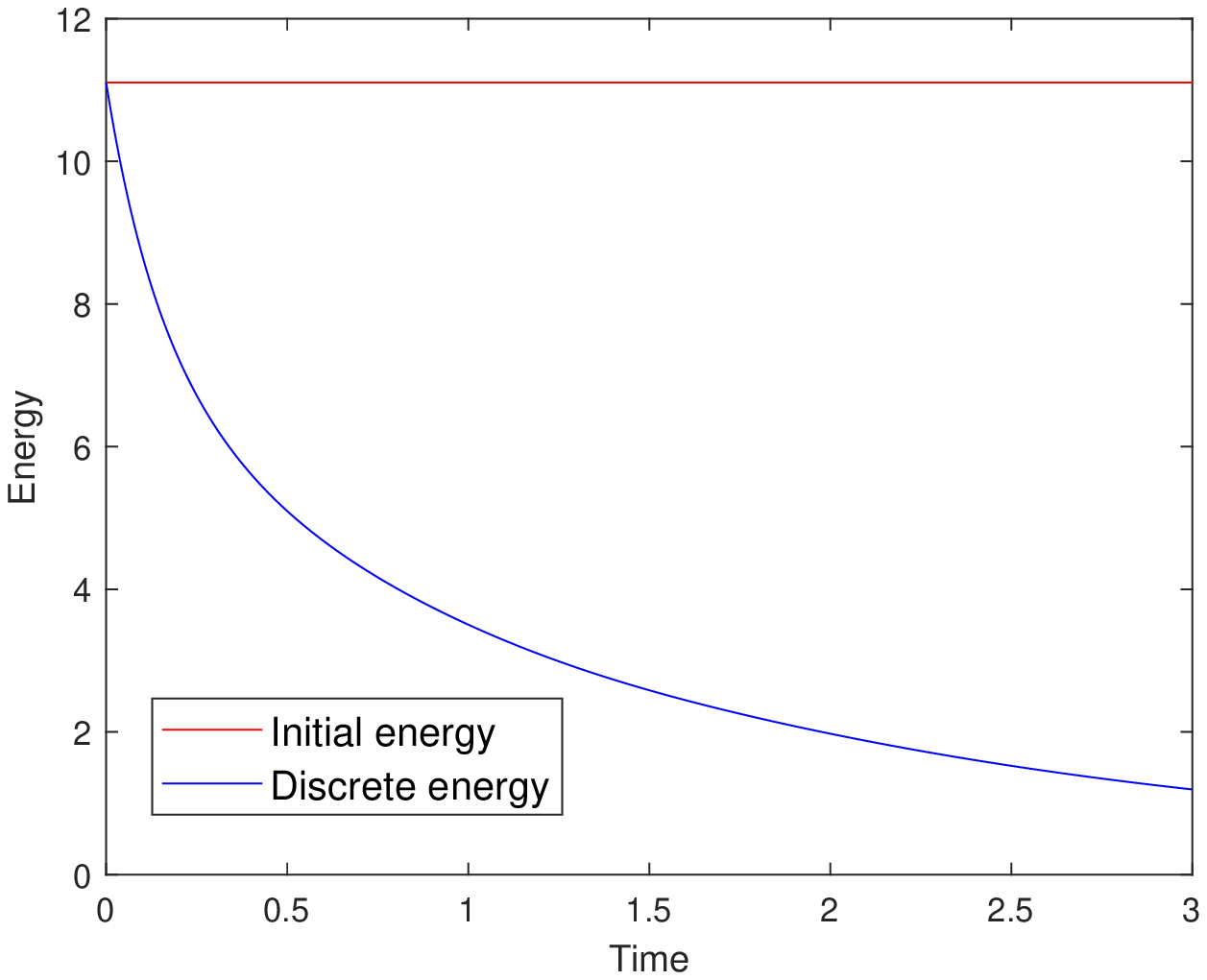} 
	\caption{Deviation in mass and momentum (Left) and energy dissipation (Right) for the Example \ref{exm5.4}}
	\label{figm2}
\end{figure}

	\section{Conclusion}\label{sec:conclude}
	
	In this article, a semi-discrete numerical method for the Vlasov-viscous Burgers' equation is introduced and analyzed. This is a DG method for the Vlasov and LDG method for the viscous Burgers' equations in phase space both with generalized numerical fluxes. The discrete scheme is mass and momentum preserving.
The optimal rate of convergence for $\lambda = 1/2$ and even degree of polynomial for $x$ and odd number of elements for space domain is derived. Further, optimal rates of convergence for $\lambda > 1/2$ and for both even and odd degree of polynomial for $x$ and $v$ are established, but  now the constant in error estimates depends on $\epsilon^{-\frac{1}{2}}.$ 
The main tools used for error estimates are the introduction of generalized Gauss-Radau projection and some application of a variant of non-linear Gr\"onwall's lemma. Finally, computational results confirm our theoretical findings.
	
\textbf{Acknowledgements.} Authors are grateful to anonymous  referees  for their valuable comments and suggestions which help to improve the revised manuscript. K.K. and H.H. thank Laurent Desvillettes for introducing them to the fluid-kinetic equations modelling the thin sprays during the Junior Trimester Program on Kinetic Theory organised at the Hausdorff Research Institute for Mathematics, Bonn. K.K. and H.H. thank the Hausdroff Institute of Mathematics, Bonn, for hosting them during the Junior Trimester program on Kinetic theory (Summer of 2019) where this work was initiated. K.K. further acknowledges the financial support of the University Grants Commission (UGC), Government of India. 
	
\vspace{2em}
\noindent
\textbf{Statements and Declarations}\\
\noindent
\textbf{Funding:}
The second author acknowledges the financial support of the University Grants Commission (UGC), Government of India. 
the first and second authors  thank the Hausdroff Institute of Mathematics, Bonn, for hosting them during the Junior Trimester program on Kinetic theory (Summer of 2019) where this work was initiated. 

\noindent
\textbf{Conflict of Interest:}
The authors declare that they have no conflict of interest.

\noindent
\textbf{Author Contributions:}
All authors contributed equally to prepare this manuscript. All authors read and approved the final manuscript.

\noindent
\textbf{Data Availability:}
The codes during the current study are available from the corresponding author on reasonable request.

	\bibliography{references}

\newcommand{\etalchar}[1]{$^{#1}$}
\begin{thebibliography}{HKMMM20}

\bibitem[BS07]{brenner2007mathematical}
S.~Brenner and R.~Scott.
\newblock {\em The mathematical theory of finite element methods}, volume~15.
\newblock Springer Science \& Business Media, 2007.

\bibitem[CCLX17]{cheng2017}
Y.~Cheng, C.-S. Chou, F.~Li, and Y.~Xing.
\newblock {$L^2$} stable discontinuous {G}alerkin methods for one-dimensional
  two-way wave equations.
\newblock {\em Math. Comp.}, 86(303):121--155, 2017.

\bibitem[CMZ17]{cheng2017application}
Yao Cheng, Xiong Meng, and Qiang Zhang.
\newblock Application of generalized {G}auss--{R}adau projections for the local
  discontinuous {G}alerkin method for linear convection-diffusion equations.
\newblock {\em Math. Comp.}, 86(305):1233--1267, 2017.

\bibitem[CS89]{cockburn1989tvb}
B.~Cockburn and C.-W. Shu.
\newblock Tvb {R}unge-{K}utta local projection discontinuous {G}alerkin finite
  element method for conservation laws. ii. {G}eneral framework.
\newblock {\em Math. Comp.}, 52(186):411--435, 1989.

\bibitem[CS98]{cockburn1998local}
B.~Cockburn and C.-W. Shu.
\newblock The local discontinuous {G}alerkin method for time-dependent
  convection-diffusion systems.
\newblock {\em SIAM J. Numer. Anal.}, 35(6):2440--2463, 1998.

\bibitem[CSZ17]{cao2017superconvergence}
W.~Cao, C.-W. Shu, and Z.~Zhang.
\newblock Superconvergence of discontinuous {G}alerkin methods for 1-{D} linear
  hyperbolic equations with degenerate variable coefficients.
\newblock {\em ESAIM Math. Model. Numer. Anal.}, 51(6):2213--2235, 2017.

\bibitem[DDCS09]{ayuso2009discontinuous}
B.~Ayuso De~Dios, J.~A. Carrillo, and Chi-Wang Shu.
\newblock Discontinuous {G}alerkin methods for the one-dimensional
  {V}lasov-{P}oisson system.
\newblock {\em Kinet. Relat. Models}, 4(4):955--989, 2009.

\bibitem[DL89]{diperna1989ordinary}
R.~J. DiPerna and P.-L. Lions.
\newblock Ordinary differential equations, transport theory and sobolev spaces.
\newblock {\em Invent. Math.}, 98(3):511--547, 1989.

\bibitem[DPE12]{2}
Daniele~Antonio Di~Pietro and Alexandre Ern.
\newblock {\em Mathematical aspects of discontinuous {G}alerkin methods},
  volume~69 of {\em Math\'{e}matiques \& Applications (Berlin) [Mathematics \&
  Applications]}.
\newblock Springer, Heidelberg, 2012.

\bibitem[DR99]{domelevo1999existence}
K.~Domelevo and J.-M. Roquejoffre.
\newblock Existence and stability of travelling wave solutions in a kinetic
  model of two-phase flows.
\newblock {\em Comm. partial differential equations}, 24(1-2):61--108, 1999.

\bibitem[Gou01]{goudon2001asymptotic}
T.~Goudon.
\newblock Asymptotic problems for a kinetic model of two-phase flow.
\newblock {\em Proc. Roy. Soc. Edinburgh Sect. A}, 131(6):1371--1384, 2001.

\bibitem[GS98]{gottlieb1998total}
S.~Gottlieb and C.-W. Shu.
\newblock Total variation diminishing {R}unge-{K}utta schemes.
\newblock {\em Math. Comp.}, 67(221):73--85, 1998.

\bibitem[Ham98]{Hamdache_1998}
K.~Hamdache.
\newblock Global existence and large time behaviour of solutions for the
  {V}lasov-{S}tokes equations.
\newblock 15(1):51--74, 1998.

\bibitem[HKMMM20]{han2019uniqueness}
D.~Han-Kwan, {\'E}.~Miot, A.~Moussa, and I.~Moyano.
\newblock Uniqueness of the solution to the 2{D} {V}lasov--{N}avier--{S}tokes
  system.
\newblock {\em Rev. Mat. Iberoam.}, 36(1):37--60, 2020.

\bibitem[Hop50]{MR47234}
E.~Hopf.
\newblock The partial differential equation {$u_t+uu_x=\mu u_{xx}$}.
\newblock {\em Comm. Pure Appl. Math.}, 3:201--230, 1950.

\bibitem[LP15]{liu2015}
H.~Liu and N.~Ploymaklam.
\newblock A local discontinuous {G}alerkin method for the {B}urgers--{P}oisson
  equation.
\newblock {\em Numer. Math.}, 129(2):321--351, 2015.

\bibitem[LWM20]{liu2020optimal}
M.~Liu, B.~Wu, and X.~Meng.
\newblock Optimal error estimates of the discontinuous {G}alerkin method with
  upwind-biased fluxes for 2{D} linear variable coefficients hyperbolic
  equations.
\newblock {\em J. Sci. Comput.}, 83(1):1--19, 2020.

\bibitem[LZM{\etalchar{+}}20]{li2020discontinuous}
J.~Li, D.~Zhang, X.~Meng, B.~Wu, and Q.~Zhang.
\newblock Discontinuous {G}alerkin methods for nonlinear scalar conservation
  laws$:$ {G}eneralized local {L}ax--{F}riedrichs numerical fluxes.
\newblock {\em SIAM J. Numer. Anal.}, 58(1):1--20, 2020.

\bibitem[LZMW19]{li2019analysis}
J.~Li, D.~Zhang, X.~Meng, and B.~Wu.
\newblock Analysis of discontinuous {G}alerkin methods with upwind-biased
  fluxes for one dimensional linear hyperbolic equations with degenerate
  variable coefficients.
\newblock {\em J. Sci. Comput.}, 78(3):1305--1328, 2019.

\bibitem[MSW16]{meng2016optimal}
X.~Meng, C.-W. Shu, and B.~Wu.
\newblock Optimal error estimates for discontinuous {G}alerkin methods based on
  upwind-biased fluxes for linear hyperbolic equations.
\newblock {\em Math. Comp.}, 85(299):1225--1261, 2016.

\bibitem[MSZW12]{meng2012superconvergence}
X.~Meng, C.-W. Shu, Q.~Zhang, and B.~Wu.
\newblock Superconvergence of discontinuous {G}alerkin methods for scalar
  nonlinear conservation laws in one space dimension.
\newblock {\em SIAM J. Numer. Anal.}, 50(5):2336--2356, 2012.

\bibitem[O'R81]{o1981collective}
P.~J. O'ROURKE.
\newblock {\em Collective drop effects on vaporizing liquid sprays}.
\newblock PhD thesis, Princeton University, 1981.

\bibitem[PKP17]{ploymaklam2017priori}
N~Ploymaklam, P~M Kumbhar, and A~K Pani.
\newblock A priori error analysis of the local discontinuous {G}alerkin method
  for the viscous {B}urgers-{P}oisson system.
\newblock {\em Int. J. Numer. Anal. Model}, 14(4-5):784--807, 2017.

\bibitem[Plo19]{ploymaklam2019local}
N.~Ploymaklam.
\newblock A local discontinuous {G}alerkin method for the reduced
  {B}urgers-{P}oisson equation.
\newblock {\em Thai J. Math.}, 17(2):515--525, 2019.

\bibitem[RH73]{reed1973triangular}
W.~H. Reed and T.~R. Hill.
\newblock Triangular mesh methods for the neutron transport equation.
\newblock Technical report, Los Alamos Scientific Lab., N. Mex.(USA), 1973.

\bibitem[Wil85]{williams2018combustion}
F.~A. Williams.
\newblock {\em Combustion theory}.
\newblock CRC Press, 1985.

\bibitem[WZS19]{wang2019implicit}
H.~Wang, Q.~Zhang, and C.-W. Shu.
\newblock Implicit--explicit local discontinuous {G}alerkin methods with
  generalized alternating numerical fluxes for convection--diffusion problems.
\newblock {\em Journal of Scientific Computing}, 81(3):2080--2114, 2019.

\end{thebibliography}
	\bibliographystyle{amsalpha}
	
	\appendix
	\section{}\label{exststrong}
	
	The following lemma yields estimate for $u_x$ in $L^\infty$-norm which helps us to prove the regularity result for viscous Burgers' equation.
	
	\begin{lem}\label{uxinf}
		For a periodic function $u$ with $u \in C^2(I)$, there holds
		\begin{equation*}
			\|u_x(t)\|_{L^\infty(I)} \leq \sqrt{2}\,\|u_x(t)\|^\frac{1}{2}_{L^2(I)}\|u_{xx}(t)\|_{L^2(I)}^\frac{1}{2}.
		\end{equation*}
	\end{lem}
	
	\begin{proof}
		As $u$ is periodic. Hence, $\int_I u_x\,{\rm d}x = 0$ and there is some $x_0 \in I$ such that $u_x(x_0) = 0$
		\begin{equation*}
			\begin{aligned}
				|u_x|^2 = \left\vert\int_{x_0}^{x}\partial_x(u_x)^2\,{\rm d}x\right\vert = 2\left\vert\int_{x_0}^{x}u_x\,u_{xx}\,{\rm d}x\right\vert \leq 2\,\|u_x\|_{L^2(I)}\,\|u_{xx}\|_{L^2(I)}.
			\end{aligned}
		\end{equation*}
		Here, in first step use fundamental theorem of Calculus and in last step use the Cauchy-Schwarz inequality. After taking square root on both side and supremum over $x \in I$ completes the proof.
	\end{proof}

	The following lemma shows the regularity result for the solution of viscous Burgers' equation.
	
	\begin{lem}\label{uregularity}
		Let $\int_{\R}\int_I\,|v|^pf_0\,{\rm d}x\,{\rm d}v < \infty$ for $0 \leq p \leq 3$ and $u_0 \in H^1(I)$. Then
		\begin{equation}
			u \in L^2(0,T;H^2(I)) \cap L^\infty(0,T;H^1(I)) \cap H^1(0,T;L^2(I)).
		\end{equation}
	\end{lem}
	
	\begin{proof}
		Multiplying equation \eqref{contburger} by $- u_{xx}$ and integrate with respect to $x$ to obtain
		\begin{equation}\label{1}
			\frac{1}{2}\partial_t\|u_x\|^2_{L^2(I)} + \epsilon\,\|u_{xx}\|^2_{L^2(I)} = -\left(\rho V, u_{xx}\right) + \left(\rho u, u_{xx}\right) + \left(uu_x,u_{xx}\right).
		\end{equation} 
		A use of the H\"older inequality with Lemma \ref{uxinf} and the Young's inequality yields
		\begin{equation}\label{2}
			\begin{aligned}
				\left(uu_x,u_{xx}\right) &\leq \|u\|_{L^2(I)}\|u_x\|_{L^\infty(I)}\|u_{xx}\|_{L^2(I)}
				\\
				&\leq \sqrt{2}\,\|u\|_{L^2(I)}\|u_x\|_{L^2(I)}^\frac{1}{2}\|u_{xx}\|^\frac{3}{2}_{L^2(I)}
				\\
				&\leq \epsilon^{-3}\,\|u\|^4_{L^2(I)}\|u_x\|^2_{L^2(I)} + \frac{3\epsilon}{4}\|u_{xx}\|^2_{L^2(I)}.		     
			\end{aligned}
		\end{equation}
		From equation \eqref{1} after a use of the H\"older inequality, the Young's inequality and equation \eqref{2}  with kickback argument, we obtain
		\begin{equation*}
			\begin{aligned}
				\partial_t\|u_x\|^2_{L^2(I)} + \frac{2\epsilon}{5}\,\|u_{xx}\|^2_{L^2(I)} \lesssim \epsilon^{-3}\left(\|\rho V\|^2_{L^2(I)} + \|\rho\|^2_{L^2(I)}\|u\|^2_{L^\infty(I)} + \|u\|^4_{L^2(I)}\|u_x\|^2_{L^2(I)}\right).
			\end{aligned}
		\end{equation*}
		A use of Sobolev inequality and an integration in time from $0$ to $t$ leads to 
		\begin{equation*}
			\begin{aligned}
				\|u_x\|^2_{L^2(I)} + \frac{2\epsilon}{5}\int_{0}^{t}\|u_{xx}\|^2_{L^2(I)}\,{\rm d}s &\lesssim  \epsilon^{-3}\,\|\rho V\|^2_{L^2([0,T]\times I)} 
				\\
				& \quad+ \epsilon^{-3}\,\int_{0}^{t}\left(\|\rho\|^2_{L^2(I)}+ \|u\|^4_{L^2(I)}\right)\|u_x\|^2_{L^2(I)}\,{\rm d}s
				\\
				&\lesssim  \epsilon^{-3}\,\|\rho V\|^2_{L^2([0,T]\times I)} + \epsilon^{-3}\,\left(\|\rho\|^2_{L^\infty(0,T;L^2(I))} \right.
				\\
				&\quad \left. + \|u\|^4_{L^\infty(0,T;L^2(I))}\right)\|u_x\|^2_{L^2([0,T]\times I)}.
			\end{aligned}
		\end{equation*}
		A use of \eqref{3} and \eqref{rhos} shows
		\begin{equation*}
			\begin{aligned}
				\|u_x\|^2_{L^2(I)} + \frac{2\epsilon}{5}\int_{0}^{t}\|u_{xx}\|^2_{L^2(I)}\,{\rm d}s \leq C\,\epsilon^{-3}.
			\end{aligned}
		\end{equation*}
		After taking supremum over $t \in [0,T]$, we obtain the first result. 
		
		Now, multiply equation \eqref{contburger} by $u_t$ and integration with respect to $x$ follows
		\begin{equation*}
			\|u_t\|^2_{L^2(I)} + \frac{\epsilon}{2}\partial_t\|u_x\|^2_{L^2(I)} = \left(\rho V,u_t\right) + \left(\rho u, u_t\right) - \left(u\,u_x, u_t\right). 
		\end{equation*} 
		A use of the H\"older inequality with the Young's inequality, kickback argument and an integration in time completes the rest of the proof.
	\end{proof}
	
	As a consequence of above lemma, it follows that
	\begin{equation*}
		\begin{aligned}
			u \in L^2(0,T;H^2(I)) \subset L^2(0,T;W^{1,\infty}(I)) \subset L^1(0,T;W^{1,\infty}(I)).
		\end{aligned}
	\end{equation*}
Here, in second step use Sobolev inequality and in third step use the H\"older inequality.
	
	
	The following result is on the propagation of velocity moments which is crucial for the proof of existence and uniqueness of strong solution.

	\begin{lem}\label{momentbddgrad}
		Let $u \in L^1(0,T;W^{1,\infty}(I))$ and let $f_0 \geq 0$ be such that 
		\[
		\int_{\R}\int_I\,\vert v\vert^k\left\{f_0 + \vert \partial_xf_0\vert^2 + \vert \partial_vf_0\vert^2\right\}\,{\rm d}x\,{\rm d}v \leq C,
		\]
		for $k \geq 0$. Then, the solution $f$ of the Vlasov equation satisfies
		\[
		\int_{\R}\int_I\,\vert v\vert^k\left\{f + \vert \partial_xf\vert^2 + \vert \partial_vf\vert^2\right\}\,{\rm d}x\,{\rm d}v \leq C,
		\]
		for $k \geq 0$ and for all $t>0$.
	\end{lem}
	
	\begin{proof}
		Consider the equation for $\partial_xf$
		\begin{equation*}
			\partial_t\partial_xf + v\partial_x^2f + \partial_v\left(\partial_xu\, f\right) + \partial_v\left(\left(u - v\right)\partial_xf\right) = 0.
		\end{equation*}
		Multiplying the above equation by $\left(1 + |v|^k\right)\partial_xf$ and integrating with respect to $x,v$ yields
		\begin{equation*}
			\frac{1}{2}\frac{{\rm d}}{{\rm d}t}\int_{\R}\int_I \left(1 + \vert v\vert^k\right)\left\vert\partial_xf\right\vert^2\,{\rm d}x\,{\rm d}v = I_1 + I_2 + I_3
		\end{equation*}
		where
		\[
				I_1 = -\int_{\R}\int_I\left(1 + |v|^k\right)\partial_xu\,\partial_vf\,\partial_xf\,{\rm d}x\,{\rm d}v,
				\quad
				I_2 = \int_{\R}\int_I\left(1 + |v|^k\right)\left\vert\partial_xf\right\vert^2\,{\rm d}x\,{\rm d}v
				\]
				and
				\[
				I_3 = -\frac{1}{2}\int_{\R}\int_I\left(1 + |v|^k\right)\left(u - v\right)\partial_v\left\vert\partial_xf\right\vert^2\,{\rm d}x\,{\rm d}v.
			\]
		After using the Young's inequality in $I_1$, we obtain
		\begin{equation*}
			I_1 \leq \|u_x\|_{L^\infty(I)}\int_{\R}\int_I\left(1 + |v|^k\right)\left( \left\vert\partial_xf\right\vert^2 + \left\vert\partial_vf\right\vert^2\right)\,{\rm d}x\,{\rm d}v.
		\end{equation*}
		An integration by parts yields
		\begin{equation*}
			I_3 = -\frac{1}{2}\int_{\R}\int_I\left(1 + |v|^k\right)\left\vert\partial_xf\right\vert^2\,{\rm d}x\,{\rm d}v + I_4
		\end{equation*}
		with 
		\begin{equation*}
			I_4 = \frac{k}{2}\int_{\R}\int_I\,\vert v\vert^{k-2}v\left(u - v\right)\left\vert\partial_xf\right\vert^2\,{\rm d}x\,{\rm d}v.
		\end{equation*}
		A use of the Young's inequality shows
		\begin{equation*}
			\begin{aligned}
				I_4 &\leq \frac{k}{2}\|u\|_{L^\infty(I)}\int_{\R}\int_I\vert v\vert^{k-1}\left\vert\partial_xf\right\vert^2\,{\rm d}x\,{\rm d}v + \frac{k}{2}\int_{\R}\int_I\,\vert v\vert^k\left\vert\partial_xf\right\vert^2\,{\rm d}x\,{\rm d}v
				\\
				&\leq \frac{k}{2}\|u\|_{L^\infty(I)}\int_{\R}\int_I\left(\frac{k-1}{k}\vert v\vert^{k} + \frac{1}{k}\right)\left\vert\partial_xf\right\vert^2\,{\rm d}x\,{\rm d}v + \frac{k}{2}\int_{\R}\int_I\,\vert v\vert^k\left\vert\partial_xf\right\vert^2\,{\rm d}x\,{\rm d}v
				\\
				&\leq C\left(1 + \|u\|_{L^\infty(I)}\right)\int_{\R}\int_I\left( 1 + \vert v\vert^k\right)\left\vert\partial_xf\right\vert^2\,{\rm d}x\,{\rm d}v.
			\end{aligned}
		\end{equation*}
		Next, consider the equation for $\partial_vf$
		\begin{equation*}
			\partial_t\partial_vf + v\partial_x\partial_vf + \partial_xf - \partial_vf + \partial_v\left(\left(u - v\right)\partial_vf\right) = 0.
		\end{equation*}
		Multiplying the above equation by $\left(1 + \vert v\vert^k\right)\partial_vf$ and integrating with respect to $x,v$ yields
		\begin{equation*}
			\frac{1}{2}\int_{\R}\int_I\left(1 + \vert v\vert^k\right)\left\vert\partial_vf\right\vert^2\,{\rm d}x\,{\rm d}v = I_5 + I_6 + I_7
		\end{equation*}
		where
		\begin{equation*}
			\begin{aligned}
				I_5 &= -\int_{\R}\int_I\left(1 + \vert v\vert^k\right)\partial_xf\,\partial_vf\,{\rm d}x\,{\rm d}v
				\\
				I_6 &= \int_{\R}\int_I\left(1 + \vert v\vert^k\right)\left\vert\partial_vf\right\vert^2\,{\rm d}x\,{\rm d}v
				\\
				I_7 &= -\int_{\R}\int_I\partial_v\left(\left(u - v\right)\partial_vf\right)\left(1 + \vert v\vert^k\right)\partial_vf\,{\rm d}x\,{\rm d}v. 
			\end{aligned}
		\end{equation*}
		After using the Young's inequality in $I_5$, we obtain
		\begin{equation*}
			I_5 \leq \int_{\R}\int_I\left(1 + |v|^k\right)\left( \left\vert\partial_xf\right\vert^2 + \left\vert\partial_vf\right\vert^2\right)\,{\rm d}x\,{\rm d}v.
		\end{equation*}
		An integration by parts yields
		\begin{equation*}
			I_7 = \int_{\R}\int_I \left(1 + \vert v\vert^k\right)\left\vert\partial_vf\right\vert^2\,{\rm d}x\,{\rm d}v + I_8
		\end{equation*}
		with 
		\begin{equation*}
			I_8 = -\frac{1}{2}\int_{\R}\int_I\left(1 + \vert v\vert^k\right)\left(u - v\right)\partial_v\left\vert\partial_vf\right\vert^2\,{\rm d}x\,{\rm d}v.
		\end{equation*}
		An integration by parts shows
		\begin{equation*}
			I_8 = -\frac{1}{2}\int_{\R}\int_I\left(1 + \vert v\vert^k\right)\left\vert\partial_vf\right\vert^2\,{\rm d}x\,{\rm d}v - \frac{k}{2}\int_{\R}\int_I\,\vert v\vert^k\left\vert\partial_vf\right\vert^2\,{\rm d}x\,{\rm d}v + I_9
		\end{equation*}
		with 
		\begin{equation*}
			I_9 = \frac{k}{2}\int_{\R}\int_I\,\vert v\vert^{k-2}vu\,\vert\partial_vf\vert^2\,{\rm d}x\,{\rm d}v.
		\end{equation*}
		A use of the Young's inequality leads to
		\begin{equation*}
			\begin{aligned}
				I_9 &\leq \|u\|_{L^\infty(I)}\frac{k}{2}\int_{\R}\int_I\,\vert v\vert^{k-1}\left\vert\partial_vf\right\vert^2\,{\rm d}x\,{\rm d}v
				\\
				&\leq \|u\|_{L^\infty(I)}\frac{k}{2}\int_{\R}\int_I\left(\frac{k-1}{k}\vert v\vert^{k} + \frac{1}{k}\right)\left\vert\partial_vf\right\vert^2\,{\rm d}x\,{\rm d}v
				\\
				&\leq C\int_{\R}\int_I\left(1 + \vert v\vert^k\right)\left\vert\partial_vf\right\vert^2\,{\rm d}x\,{\rm d}v.
			\end{aligned}
		\end{equation*}
		Altogether, we obtain
		\begin{equation*}
			\begin{aligned}
				\frac{{\rm d}}{{\rm d}t}&\left(\int_{\R}\int_I\left(1 + \vert v\vert^k\right)\left(\left\vert\partial_xf\right\vert^2 + \left\vert\partial_vf\right\vert^2\right)\,{\rm d}x\,{\rm d}v\right) \leq C\left(1 + \|u\|_{W^{1,\infty}(I)}\right)
				\\
				&\hspace{5cm}\int_{\R}\int_I\left(1 + \vert v\vert^k\right)\left(\left\vert\partial_x f\right\vert^2 + \left\vert\partial_vf\right\vert^2\right)\,{\rm d}x\,{\rm d}v.
			\end{aligned}
		\end{equation*}
		Let $ F(t) = \int_{\R}\int_I\left(1 + \vert v\vert^k\right)\left(\left\vert\partial_xf\right\vert^2 + \left\vert\partial_vf\right\vert^2\right)\,{\rm d}x\,{\rm d}v$, then
		\begin{equation*}
			\begin{aligned}
				F'(t) \leq C\left(1 + \|u(t,x)\|_{W^{1,\infty}(I)}\right)F(t)
			\end{aligned}
		\end{equation*}
	this implies 
	\begin{equation*}
		F(t) \leq CF(0)e^{\|u\|_{L^1(0,T;W^{1,\infty}(I))}}.
	\end{equation*}
This completes our proof.
	\end{proof}

	
	\begin{thm}\textbf{(Existence and Uniqueness of strong solution)}
		Let the initial data $\left(f_0,u_0\right)$ be such that
		\[
		f_0 \in L^1(I\times\R)\cap L^\infty(I\times\R)\cap H^1(I\times\R), \quad f_0 \geq 0,
		\]
		\[
		\int_{\R}\int_I\,\vert v\vert^p\left\{f_0 + \left\vert\partial_xf_0\right\vert^2 + \left\vert\partial_vf_0\right\vert^2\right\}\,{\rm d}x\,{\rm d}v \leq C,
		\]
		for $0 \leq p \leq 4$ and $u_0 \in H^1(I)$. Then, there exists a unique global-in-time strong solution $\left(f,u\right)$ to the Vlasov-viscous Burgers' system \eqref{eq:continuous-model}-\eqref{contburger}.
	\end{thm}
	
	\begin{proof}
		Let $0 < T < \infty$ and set $X = L^\infty(0,T;H^1(I))$, with norm 
		\[
		\|u\|^2_{X} = \sup_{t \in [0,T]}\left(\int_I\left(u^2 + |\partial_x u|^2\right)\,{\rm d}x\right).
		\]
		We now consider the map 
		\begin{equation}\label{T}
			\begin{aligned}
				\mathcal{T} &: X \rightarrow X
				\\
				u^* &\longmapsto u = \mathcal{T}(u^*)
			\end{aligned}
		\end{equation}
		defined by the following scheme.
		\begin{itemize}
			\item Solve the Vlasov equation
			\begin{equation}\label{k}
				\partial_t f + v\,\partial_x f + \partial_v \left(\left(u^* - v\right)f\right) = 0,
			\end{equation}
			with initial data $f_0$.
			\item Solve the viscous Burgers' equation 
			\begin{equation}\label{f}
				\begin{aligned}
					\left(u_t,\phi\right) + \left(uu_x, \phi\right) + \epsilon\left(u_x,\phi_x\right) = \left(\rho V - \rho u, \phi\right), \quad \forall\,\, \phi \in H^1(I)
				\end{aligned}
			\end{equation}
			with initial data $u_0$. Here $\rho$ and $\rho V$ are the local density and the momentum associated with the solution $f$ of \eqref{k}. 
		\end{itemize}
		To begin with, we show that the above map $\mathcal{T}$ is well-defined. For a given $u^* \in X$ and a given initial datum $f_0$, the Vlasov equation \eqref{k} is uniquely solvable (see Lemma \ref{fexist} below for details). Having solved \eqref{k} for $f(u^*)$, one gather that the corresponding local density $\rho \in L^\infty$ (see Lemma \ref{rhoinfty}) and the corresponding momentum $\rho V \in L^2$ (see Lemma \ref{density}). Hence, classical theory for the viscous Burgers' problem \cite{MR47234} yields a unique solution $u \in X$ for the problem \eqref{f}. Thus, the map $\mathcal{T} : X \to X$ that takes $u^*$ to $\mathcal{T}(u^*) = u$ is well-defined.
		Our next step in the proof is to show that $\mathcal{T}$ is a contraction map and that has been demonstrated in Lemma \ref{Tcontr} below. Therefore, an application of the Banach fixed-point theorem ensures the existence of a unique solution $(f,u)$ in a short time interval $(0,T^0)$. As the solution $(f,u)$ stays bounded at $t = T^0$, thanks to a priori estimates, we can employ continuation argument to extend the interval of existence upto $(0,T]$. As $T$ is arbitrary, we get global-in-time well-posedness of our system.
	\end{proof}
	
	Next we deal with Lemmata \ref{fexist} and \ref{Tcontr} which played a crucial role in the above proof.

	
	\begin{lem}\label{fexist}
		Let $u^* \in X$ and let $f_0 \in L^1(I\times \R)\cap L^\infty(I\times\R)$. Then, there exists a unique solution $f \in L^\infty(0,T;L^1(I\times\R)\cap L^\infty(I\times\R))$ to \eqref{k}.
	\end{lem}
	
	\begin{proof}
		Note that \eqref{k} can be rewritten as 
		\begin{equation*}
			\partial_t f + b \cdot \nabla_{x,v}f - f =0,
		\end{equation*}
		where $b = (v, u^*(t,x) - v)$, which lies in 
		\begin{equation*}
			L^1(0,T;H^1(I \times (-K,K))), \quad 0 < K < \infty.
		\end{equation*}
		Note that div$_{x,v}b = -1 \in L^\infty((0,T) \times \Omega)$. Furthermore, $|b|/(1 + |v|)$ is bounded. This setting appeals to the general results in \cite{diperna1989ordinary}. In particular, we can apply \cite[Corollaries II-1 and II-2, p.518]{diperna1989ordinary} to arrive at the existence of the unique solution.
	\end{proof}

	
	\begin{lem}\label{Tcontr}
		The map defined by \eqref{k} and \eqref{f} is a contraction map.
	\end{lem}
	\begin{proof}
		Let for given $u_i^*$ there exist a unique solution $f_i$ for equation \eqref{k}. Define $\bar{u} = u_1 - u_2, \bar{u}^* = u_1^* - u_2^*$ and $\bar{f} = f_1 - f_2$, then from \eqref{k}-\eqref{f} we find that
		\begin{equation}\label{131}
			\bar{f}_t + v\bar{f}_x + \partial_v\left(\bar{u}^*f_1 + u_2^*\bar{f} - v\bar{f}\right) = 0, \quad \mbox{with} \quad	\bar{f}(0,x,v) = 0,
		\end{equation}
		and
		\begin{equation}\label{141}
			\left(\bar{u}_t, \phi\right)_I - \epsilon\,\left(\bar{u}_{xx}, \phi\right)_I = \left(\int_{\R}\left(v\bar{f} + u_2\bar{f} - \bar{u}f_1\right)\,{\rm d}v, \phi\right)_I - \left(\left(\bar{u}\partial_xu_1 + u_2\bar{u}_x\right), \phi\right)_I,
		\end{equation}
		with $\quad \bar{u}(0,x) = 0.$

Choose $\phi = \bar{u} - \bar{u}_{xx}$ in  \eqref{141} to obtain
\begin{equation*}
	\begin{aligned}
		\frac{1}{2}\partial_t\|\bar{u}\|^2_{L^2(I)} + \frac{1}{2}\partial_t\|\bar{u}_x\|^2_{L^2(I)} + \epsilon\|\bar{u}_x\|^2_{L^2(I)} + \epsilon\,\|\bar{u}_{xx}\|^2_{L^2(I)} &= \left(\int_{\R}\left(v\bar{f} + u_2\bar{f} - \bar{u}f_1\right)\,{\rm d}v, \bar{u} - \bar{u}_{xx}\right)_I
		\\
		&\quad- \left(\left(\bar{u}\partial_xu_1 + u_2\bar{u}_x\right), \bar{u} - \bar{u}_{xx}\right)_I
	\end{aligned}
\end{equation*}
a use of the H\"older inequality with the Young's inequality, a kickback argument and integration in time yields
\begin{equation}\label{151}
	\begin{aligned}
		\|\bar{u}\|^2_{X} + \epsilon\|\bar{u}_x\|^2_{L^2([0,T]\times I)} + \epsilon\,\|\bar{u}_{xx}\|_{L^2([0,T]\times I)}^2 \leq& \quad 4\epsilon^{-1}\,\left\Vert\int_{\R}\left(v\bar{f} + u_2\bar{f} - \bar{u}f_1\right)\,{\rm d}v\right\Vert^2_{L^2([0,T]\times I)}
		\\
		& + 4\epsilon^{-1}\,\|\bar{u}\partial_xu_1 + u_2\bar{u}_x\|^2_{L^2([0,T] \times I)}  + \epsilon T\|\bar{u}\|^2_X.
	\end{aligned}
\end{equation}

		Now, the H\"older inequality followed by Sobolev imbedding shows
		\begin{equation}\label{161}
			\begin{aligned}
				&\left\Vert\int_{\R}\left(v\bar{f} + u_2\bar{f} - \bar{u}f_1\right)\,{\rm d}v\right\Vert_{L^2([0,T]\times I)} \leq \left\Vert\int_{\R}v\bar{f}\,{\rm d}v\right\Vert_{L^2([0,T]\times I)} 
				\\
				&\hspace{3.5cm}+ T\|\bar{u}\|_{X}\|m_0f_1\|_{L^\infty(0,T;L^2(I))} + \|u_2\|_{X }\|m_0\bar{f}\|_{L^2([0,T] \times I)},
			\end{aligned}
		\end{equation}
		and
		\begin{equation}\label{vv1.}
			\begin{aligned}
				\|\bar{u}\partial_xu_1 + u_2\bar{u}_x\|_{L^2([0,T]\times I)} &\leq \|\bar{u}\|_{L^\infty([0,T]\times I)}\|\partial_xu_1\|_{L^2([0,T]\times I)}
				\\
				&\quad + \|u_2\|_{L^2([0,T];L^\infty(I))}\|\bar{u}\|_{L^\infty([0,T];H^1(I))}
				\\
				&\leq T^\frac{1}{2}\left(\|u_1\|_{X} + \|u_2\|_X\right)\|\bar{u}\|_{X}.
			\end{aligned}
		\end{equation}
		For a sufficiently small $T > 0$ 
		\begin{equation}\label{171}
			\begin{aligned}
				4\epsilon^{-1}\,T\left(\|m_0f_1\|^2_{L^\infty(0,T;L^2(I))} + \|u_1\|^2_{X} + \|u_2\|^2_X + \frac{\epsilon^2}{4}\,\right) \leq \frac{1}{2}.
			\end{aligned}
		\end{equation}
		Using equation \eqref{151}-\eqref{171}, we obtain
		\begin{equation}\label{181}
			\|\bar{u}\|_{X}^2 \leq C\left\Vert\int_{\R}v\bar{f}\,{\rm d}v\right\Vert^2_{L^2([0,T]\times I)} + C\|u_2\|^2_{X}\|m_0\bar{f}\,{\rm d}v\|^2_{L^2([0,T] \times I)}.
		\end{equation}
		Now, a similar calculation as in proof of Lemma \ref{density} shows 
		\begin{equation}\label{191}
			\left\Vert\int_{\R}v\bar{f}\,{\rm d}v\right\Vert_{L^2([0,T] \times I)} \leq C\int_0^T\int_{\R}\int_I|v|^3|\bar{f}|\,{\rm d}x\,{\rm d}v\,{\rm d}t,
		\end{equation}
		and
		\begin{equation}\label{201}
			\left\Vert\int_{\R}\bar{f}\,{\rm d}v\right\Vert_{L^2([0,T]\times I)} \leq C\int_0^T\int_{\R}\int_I|v||\bar{f}|\,{\rm d}x\,{\rm d}v\,{\rm d}t.
		\end{equation}
		Multiplying equation \eqref{131} by $|v|^k\frac{\bar{f}}{\sqrt{\bar{f}^2 + \delta}}$; $k \geq 1$, we obtain
		\begin{equation*}
			\begin{aligned}
				|v|^k\partial_t\left(\sqrt{\bar{f}^2 + \delta}\right) &+ |v|^k\,v\,\partial_x\left(\sqrt{\bar{f}^2 + \delta}\right) + \bar{u}^*|v|^k\frac{\bar{f}}{\sqrt{\bar{f}^2 + \delta}}\,\partial_vf_1 
				\\
				&+ u_2^*|v|^k\,\partial_v\left(\sqrt{\bar{f}^2 + \delta}\right) - |v|^k\,\frac{\bar{f}}{\sqrt{\bar{f}^2 + \delta}}\,\partial_v\left(v\bar{f}\,\right) = 0.
			\end{aligned}
		\end{equation*}
		Then, an integrate with respect to $x,v$ yields
		\begin{equation*}
			\begin{aligned}
				\partial_t\int_\R&\int_I |v|^k\left(\sqrt{\bar{f}^2 + \delta}\right) \,{\rm d}x\,{\rm d}v - k\int_\R\int_I\bar{u}^*\,|v|^{k-2}v\,\frac{\bar{f}}{\sqrt{\bar{f}^2 + \delta}}\,f_1\,{\rm d}x\,{\rm d}v 
				\\
				&- \int_\R\int_I|v|^k\,\bar{u}^*\,f_1\,\partial_v\left(\frac{\bar{f}}{\sqrt{\bar{f}^2 + \delta}}\right)\,{\rm d}x\,{\rm d}v - k\int_\R\int_I |v|^{k-2}v\,u^*_2\,\left(\sqrt{\bar{f}^2 + \delta}\right)\,{\rm d}x\,{\rm d}v 
				\\
				&+ k\int_\R\int_I|v|^k\,\frac{\bar{f}^2}{\sqrt{\bar{f}^2 + \delta}} \,{\rm d}x\,{\rm d}v + \int_\R\int_I |v|^k\partial_v\left(\frac{\bar{f}}{\sqrt{\bar{f}^2 + \delta}}\right)v\bar{f}\,{\rm d}x\,{\rm d}v = 0.
			\end{aligned}
		\end{equation*}
		A use of Sobolev inequality with integration in time shows
		\begin{equation}\label{231}
			\begin{aligned}
				\sup_{t \in [0,T]}\int_{\R}\int_I|v|^k&\left(\sqrt{\bar{f}^2 + \delta}\right)\,{\rm d}x\,{\rm d}v + k\int_0^T\int_{\R}\int_I|v|^k\frac{\bar{f}^2}{\sqrt{\bar{f}^2 + \delta}}\,{\rm d}x\,{\rm d}v\,{\rm d}t 
				\\
				&\leq 
				\int_0^Tk\|\bar{u}^*\|_{H^1(I)}\|m_{k-1}f_1\|_{L^1(I)} \,{\rm d}t + |T_k^1| + |T^2_k|
				\\
				&\quad+ \int_{0}^T k\|u_2^*\|_{H^1(I)}\left\Vert\int_{\R}|v|^{k-1}\left(\sqrt{\bar{f}^2 + \delta}\right) \,{\rm d}v\right\Vert_{L^1(I)}\,{\rm d}t 
				\\
				&\leq T\|\bar{u}^*\|_{X}\|m_{k-1}f_1\|_{L^\infty(0,T;L^1(I))} + |T_k^1| + |T^2_k|
				\\
				&\,\,+ T\|u_2^*\|_{X}\left\Vert\int_{\R}|v|^{k-1}\left(\sqrt{\bar{f}^2 + \delta}\right)\,{\rm d}v\right\Vert_{L^\infty(0,T;L^1(I))}.
			\end{aligned}
		\end{equation}
		Here,
		\begin{equation*}
			T^1_k = \int_0^T\int_\R\int_I|v|^k\,\bar{u}^*\,f_1\,\partial_v\left(\frac{\bar{f}}{\sqrt{\bar{f}^2 + \delta}}\right)\,{\rm d}x\,{\rm d}v\,{\rm d}t = \int_0^T\int_\R\int_I|v|^k\,\bar{u}^*\,f_1\,\frac{\delta\,\bar{f}_v}{\left(\bar{f}^2+\delta\right)^\frac{3}{2}}{\rm d}x\,{\rm d}v\,{\rm d}t
		\end{equation*}
		and
		\begin{equation*}
			T^2_k = \int_0^T\int_\R\int_I |v|^k\partial_v\left(\frac{\bar{f}}{\sqrt{\bar{f}^2 + \delta}}\right)v\bar{f}\,{\rm d}x\,{\rm d}v\,{\rm d}t = \int_0^T\int_\R\int_I |v|^k v\bar{f}\frac{\delta\,\bar{f}_v}{\left(\bar{f}^2+\delta\right)^\frac{3}{2}}\,{\rm d}x\,{\rm d}v\,{\rm d}t.
		\end{equation*}
		As $\bar{f} \in L^1(0,T;L^\infty(I\times\R))$ and as fourth order velocity moments of $\left\vert\partial_vf\right\vert^2$ and $\bar{f}$ are bounded (see Lemma \ref{momentbddgrad}), $\left\vert T_k^1\right\vert \to 0$ and $\left\vert T_k^2\right\vert \to 0$ as $\delta \to 0$ for $k = 1,2,3$.
		Next, we multiply equation \eqref{131} by $\frac{\bar{f}}{\sqrt{\bar{f}^2+\delta}}$ and integrate with respect to $x,v$ and $t$ variables to obtain
		\begin{equation*}\label{11111}
			\begin{aligned}
				\int_\R\int_I\sqrt{\bar{f}^2+\delta}\,{\rm d}x\,{\rm d}v - \int_0^T\int_\R\int_I\frac{\delta\,\bar{f}_v}{\left(\bar{f}^2+\delta\right)^\frac{3}{2}}&\left(\bar{u}^*f_1 + u_2^*\bar{f} - v\bar{f}\right)\,{\rm d}x\,{\rm d}v 
				\\
				&= \int_\R\int_I\sqrt{\bar{f}^2(0,x,v)+\delta}\,{\rm d}x\,{\rm d}v,
			\end{aligned}
		\end{equation*}
		with $\bar{f}(0,x,v) = 0$. Arguing as we did with the $T_k^1$ and $T_k^2$ terms, in the $\delta \rightarrow 0$ limit in above equation, we arrive at
		\begin{equation}
			\int_\R\int_I|\bar{f}|\,{\rm d}x\,{\rm d}v = 0.
		\end{equation}

		A use of equation \eqref{231} for $k = 1$ in equation \eqref{201} and $\delta \rightarrow 0$ shows
		\begin{equation}\label{241}
			\left\Vert\int_{\R}\bar{f}\,{\rm d}v\right\Vert_{L^2([0,T]\times I)} \leq T\|\bar{u}^*\|_{X}\|f_0\|_{L^1(\Omega)}.
		\end{equation}
		Using equation \eqref{191} and by induction from equation \eqref{231}, we deduce
		\begin{equation}\label{251}
			\begin{aligned}
				\left\Vert\int_{\R}v\bar{f}\,{\rm d}v\right\Vert_{L^2([0,T]\times I)} &\lesssim T\|\bar{u}^*\|_{X}\left(\|m_2f_1\|_{L^\infty(0,T;L^1(I))} + \|u_2^*\|_{X}\|m_1f_1\|_{L^\infty(0,T;L^1(I))}\right.
				\\
				&\left. \quad  + \|u_2^*\|^2_{X}\|m_0f_1\|_{L^\infty(0,T;L^1(I))}\right).
			\end{aligned}
		\end{equation}
		Using \eqref{241}-\eqref{251} for a sufficiently small $T > 0$ from \eqref{181}, we obtain
		\[
		\|\bar{u}\|_{X} \leq \alpha_0\|\bar{u}^*\|_{X}, \quad \alpha_0 < 1.
		\]
		This shows that $\mathcal{T}$ is a contraction map.
	\end{proof}

\end{document}